\documentclass[12pt]{amsart}

\calclayout
\usepackage[utf8]{inputenc}

\usepackage{amsmath,amsthm,amssymb}
\usepackage{mathrsfs}

\usepackage{colonequals}
\usepackage{enumitem}
\usepackage{graphicx}
\usepackage{subcaption}
\usepackage[unicode,colorlinks=true,allcolors=blue,pdftitle={Vortex Lines Interaction in 3D Ginzburg--Landau}]{hyperref}

\usepackage[initials,nobysame,alphabetic,msc-links,backrefs]{amsrefs}
\usepackage{color}

\newtheorem{theorem}{Theorem}[section]
\newtheorem{theoremm}{Theorem}[section]

\newtheorem{proposition}{Proposition}[section]
\newtheorem{propositionn}{Proposition}[section]

\newtheorem{lemma}{Lemma}[section]
\newtheorem{definition}{Definition}[section]
\newtheorem{remark}{Remark}[section]
\newtheorem{condition}{Condition}[section]

\numberwithin{equation}{section}

\DeclareMathOperator{\diver}{div}
\DeclareMathOperator{\curl}{curl}
\DeclareMathOperator{\R}{R}

\newcommand{\distg}{\mathrm{dist}_{g^\perp}}

\newcommand{\de}{\delta}

\newcommand{\ep}{\varepsilon}
\newcommand{\tep}{{\tilde \varepsilon}}
\newcommand{\lep}{{|\log\ep|}}
\newcommand{\ga}{\Gamma}
\newcommand{\tga}{{\widetilde \Gamma}}

\newcommand{\nab}{\nabla}

\renewcommand{\L}{\mathbb L}
\newcommand{\N}{\mathbb N}
\newcommand{\RR}{\mathbb R}
\newcommand{\C}{\mathbb C}

\newcommand{\CC}{\mathcal C}
\newcommand{\F}{\mathcal F}

\newcommand{\NN}{\mathcal N}

\newcommand{\T}{\mathcal T}

\newcommand{\B}{\mathfrak B}

\renewcommand{\CC}{\mathscr C}

\newcommand{\e}{\mathbf e}
\newcommand{\dl}{\varrho}

\newcommand{\degone}{{f_0}}

\renewcommand{\deg}{\mathrm{deg}}
\newcommand{\dist}{\mathrm{dist}}
\newcommand{\ex}{\mathrm{ex}}
\newcommand{\he}{h_\ex}

\newcommand{\loc}{\mathrm{loc}}
\newcommand{\good}{\mathrm{good}}
\newcommand{\bad}{\mathrm{bad}}

\newcommand{\vu}{{\vec u}}
\newcommand{\vv}{{\vec v}}
\renewcommand{\(}{\left(}
\renewcommand{\)}{\right)}

\renewcommand{\u}{\mathbf{u}}
\newcommand{\A}{\mathbf{A}}

\newcommand{\fperp}{{F_\ep^\perp}}
\newcommand{\fepperp}{{F_\ep^\perp}}
\newcommand{\fpar}{{F_\ep^z}}
\newcommand{\eperp}{{e_\ep^\perp}}
\newcommand{\epar}{{e_\ep^z}}
\newcommand{\vol}{{\mathrm{vol}}}

\newcommand{\gperp}{{g^\perp}}
\newcommand{\gz}{{g_{33}}}
\newcommand{\tg}{{\tilde g}}
\newcommand{\tgz}{{\tg_{33}}}

\def\pr#1{\left\langle #1\right\rangle}
\def\pperp#1{\left\langle #1\right\rangle^{\perp}}

\def\prem#1{\left\langle #1\right\rangle^{\mathrm{2D}}}
\def\prll#1{\left\langle #1\right\rangle^{\parallel}}

\def\hal{{\frac12}}

\def\qell{{Q_\ell}}
\def\qelln{{Q_{\ell_n}}}

\def\Lb{{\mathscr L_B}}
\def\Ll{{\mathscr L_L}}
\def\Lag{{\mathscr L}}
\def\gell{{\tilde g}}

\def\bzn{{\underline z_n}}

\def\gan{{\ga_n}}
\def\vun{{\vu_n}}
\def\wun{{{\vec w}_n}}
\def\eln{{\ell_n}}
\def\cyl{{\mathscr C}}
\def\Cd{{\cyl_\delta}}
\def\Ud{{U_\delta}}
\def\ellzero{{L_0}}
\def\hci{{H_{c_1}}}

\newcommand{\fep}{{F_\ep}}
\newcommand{\tfep}{{\tilde F_\ep}}

\def\Xint#1{\mathchoice
{\XXint\displaystyle\textstyle{#1}}%
{\XXint\textstyle\scriptstyle{#1}}%
{\XXint\scriptstyle\scriptscriptstyle{#1}}%
{\XXint\scriptscriptstyle\scriptscriptstyle{#1}}%
\!\int}
\def\XXint#1#2#3{{\setbox0=\hbox{$#1{#2#3}{\int}$ }
\vcenter{\hbox{$#2#3$ }}\kern-.6\wd0}}

\def\dashint{\Xint-}

\usepackage[foot]{amsaddr}

\title[Vortex Lines Interaction in 3D Ginzburg--Landau]{Vortex lines interaction in the three-dimensional magnetic Ginzburg--Landau model}

\author[Carlos Rom\'{a}n]{Carlos Rom\'{a}n}
\address[Carlos Rom\'{a}n]{Facultad de Matem\'aticas e Instituto de Ingenier\'ia Matem\'atica y Computacional, Pontificia Universidad Cat\'olica de Chile, Vicu\~na Mackenna 4860, 7820436 Macul, Santiago, Chile}
\email{carlos.roman@uc.cl}

\author[Etienne Sandier]{Etienne Sandier}
\address[Etienne Sandier]{LAMA - CNRS UMR 8050, Universit\'e Paris-Est Cr\'eteil, 61 Avenue du G\'en\'eral de Gaulle, 94010 Cr\'eteil, France}
\email{sandier@u-pec.fr}

\author[Sylvia Serfaty]{Sylvia Serfaty}
\address[Sylvia Serfaty]{Courant Institute of Mathematical Sciences, New York University, 251 Mercer St., New York, NY 10012, United States\\ and Sorbonne Universit\'e,
 CNRS, Universit\'e de Paris,  Laboratoire Jacques-Louis Lions (LJLL), F-75005 Paris }
\email{serfaty@cims.nyu.edu}

\date{\today}

\begin{document}
\begin{abstract}We complete our study of the three dimensional Ginzburg--Landau functional with magnetic field,  in the asymptotic regime of a small inverse Ginzburg--Landau  parameter $\varepsilon$, and near the first critical field $H_{c_1}$ for which the first vortex filaments appear in energy minimizers.
   Under a nondegeneracy condition, we show a next order asymptotic expansion of $H_{c_1}$ as $\varepsilon \to 0$, and exhibit a sequence of transitions, with 
    vortex lines appearing one by one as the intensity of the applied magnetic field is increased: passing $H_{c_1}$ there is one vortex, then  increasing $H_{c_1}$ by an increment of order  $\log |\log\varepsilon|$ a second vortex line appears, etc. These vortex lines accumulate near a special curve $\Gamma_0$, solution to
    an isoflux problem. We derive a next order energy that the vortex lines must minimize in the asymptotic limit, after a suitable horizontal blow-up around $\Gamma_0$. This energy is  the sum of  terms where penalizations of  the length of the lines, logarithmic repulsion between the lines and magnetic confinement near $\Gamma_0$ compete. This elucidates the shape of vortex lines in superconductors.
\end{abstract}
\maketitle

\section{Introduction}
This work is the conclusion of our study of the emergence of vortex lines in the three-dimensional full Ginzburg--Landau model from physics, i.e.~the model with gauge and with external magnetic field. The Ginzburg--Landau model is important as the simplest gauge theory, where  the $\mathbb{U}(1)$-gauge is Abelian (it is also known as the Abelian Higgs model) and where topological defects in the form of vortices arise. It is also one of the most famous models of condensed matter physics,  the widely used and studied model for superconductivity,  and also very similar to models for superfluids and Bose--Einstein condensates \cites{SSTS,DeG,Tin,TilTil}.

In the two-dimensional  version of the model, vortices are point-like topological defects, arising when the applied magnetic field is large enough, as superconductivity  defects  in which the magnetic flux can penetrate. Vortices interact logarithmically, the magnetic field acting as an effective  confinement potential. As a result of this competition between repulsion and confinement, in energy minimizers vortices form very interesting patterns, including a famous triangular lattice pattern called in physics Abrikosov lattice. 
The program carried out in particular  by the last two authors, see  \cite{SanSerBook}, culminating with \cite{SS1}, was to mathematically analyze the formation of these vortices and derive effective interaction energies that the limiting vortex pattern must minimize in a certain asymptotic limit, thus relating the minimization of the Ginzburg--Landau energy to discrete minimization problems, some of them of number theoretic nature.

Our main goal here was to accomplish the same in three dimensions, deriving effective interaction energies for vortex lines in three dimensions in order to precisely  understand and describe the vortex patterns  in superconductors. This is significantly more delicate in three dimensions than in two, since vortex lines carry much more geometry than vortex points. In particular, curvature effects and regularity questions, requiring the use of fine geometric measure theoretic tools coupled with differential forms, appear. In the three dimensional situation, the energetic cost of a vortex is  the balance between  its length cost,  the logarithmic interaction with other vortices, and the confinement effect of the magnetic field. While the length cost effect had been analyzed in  a simplified setting in the mathematical literature  \cites{Riv,LinRiv2,San,BetBreOrl}, the logarithmic repulsion effect analyzed, again in a simplified setting, more recently in \cite{ConJer}, our paper is the first to handle all three effects at the same time in the completely realistic physical setting of the full gauged energy. In particular, it settles questions raised since the turn of the century (for instance \cites{Riv,AftRiv}) about whether vortices  in three dimensional superconductors will be asymptotically straight or curved.

\subsection{Description of the model}
Let us now get into the details of the Ginzburg--Landau model, whose physical background  can be found in  the standard texts \cites{SSTS,DeG,Tin}. 

After nondimensionalization of the physical constants, one may reduce to studying  the energy functional
\begin{equation}
   \label{modgl}
   GL_\ep (u, A)\colonequals \frac12 \int_\Omega |\nab_A u|^2 + \frac1{2\ep^2}(1-|u|^2)^2+ \hal \int_{\RR^3} |H-H_{\ex}|^2.\end{equation}

Here $\Omega $ represents the material sample, we assume it to be a bounded simply connected subset of $\RR^3$ with regular boundary.
The function $u : \Omega \to \C$ is the {\it order parameter}, representing the local state of the material in this macroscopic theory ($|u|^2\le 1$ indicates the local density of superconducting electrons), while the vector-field $A: \RR^3 \to \RR^3$ is the gauge of the magnetic field, and the magnetic field induced inside the sample and outside is $H\colonequals \nab \times A$, as is standard in electromagnetism. The covariant derivative $\nab_A$ means $\nab - i A$. The vector field $H_{\ex}$ here represents an applied magnetic field and we will assume that $H_{\ex}= h_{\ex} H_{0,\ex}$ where $H_{0,\ex} $ is a fixed vector field and $h_{\ex}$ is a real parameter, representing an intensity that can be tuned.
Finally, the parameter $\ep>0$ is the inverse of the so-called Ginzburg--Landau parameter $\kappa$, a dimensionless ratio of all material constants,  that depends only on the type of material. In our mathematical analysis of the model, we will study the asymptotics of $\ep \to 0$, also called ``London limit" in physics, which corresponds to extreme type-II superconductors (type-II superconductors are those with large $\kappa$). This is the limiting where the {\it correlation length} is much smaller than the {\it penetration depth} of the magnetic field, effectively this means that vortex cores are very small.
The Ginzburg--Landau theory is an effective Landau theory, describing the local state at the mesoscale level by the order parameter $u$, but it can be formally derived as a limit of the microscopic quantum Bardeen--Cooper--Schrieffer theory \cite{BCS} near the critical temperature. This has been partially accomplished rigorously in \cite{FHSS}.

The Ginzburg--Landau model is a $\mathbb U(1)$-gauge theory, in which  all
the meaningful physical quantities are invariant under the gauge transformations
$u \to  u e^{i\Phi}$, $ A \to A + \nab \Phi$ where
$\Phi$  is any regular enough real-valued function. The Ginzburg--Landau energy and its
associated free energy
\begin{equation} \label{freee}F_\ep(u, A): = \frac12 \int_\Omega |\nab_A u|^2 + \frac1{2\ep^2}(1-|u|^2)^2+ \hal \int_{\RR^3} |H|^2
\end{equation} are gauge-invariant, as well as the density of superconducting Cooper pairs $|u|^
   2$, the induced magnetic field $H$, and the vorticity defined below.

Throughout this paper, we assume that $H_\ex\in L^2_{\loc} (\RR^3, \RR^3)$ is such that $\diver H_\ex=0$ in $\RR^3$. Consequently, there exists a vector potential $A_\ex \in H^1_{\loc}(\RR^3,\RR^3)$ such that 
$$\curl A_\ex= H_\ex\quad \text{and} \quad \diver A_\ex=0 \ \text{in} \ \RR^3.$$
The natural space for minimizing $GL_\ep$ in 3D is 
$H^1(\Omega, \C) \times [A_\ex+ H_{\curl}]$ where 
$$H_{\curl}\colonequals \{ A \in H^1_{\loc} (\RR^3, \RR^3) | \curl A \in L^2(\RR^3,\RR^3)\};$$ see \cite{Rom2}. 
Critical points $(u,A)$ of $GL_\ep$ in this space satisfy the Ginzburg--Landau equations 
	\begin{equation}\label{GLeq}
		\left\lbrace 
		\begin{array}{rcll}
			-(\nabla_A)^2u&=&\displaystyle\frac1{\ep^2}u(1-|u|^2)&\mathrm{in}\ \Omega\\
			\curl(H-H_{\ex})&=&(iu,\nabla_A u)\chi_\Omega &\mathrm{in}\ \RR^3\\
			\nabla_A u\cdot \nu&=&0&\mathrm{on}\ \partial \Omega\\
			\lbrack H-H_{\ex}\rbrack\times\nu&=&0&\mathrm{on}\ \partial\Omega, 
		\end{array}
		\right.
	\end{equation}
where $\chi_\Omega$ is the characteristic function of $\Omega$, $[\, \cdot \, ]$ denotes the jump across $\partial\Omega$, $\nu$ is the outer unit normal to the boundary,  $\nabla_A u\cdot \nu=\sum_{j=1}^3 (\partial_ju-iA_ju)\nu_j$, and the covariant Laplacian $(\nabla_A)^2$ is defined by
$$
(\nabla_A)^2u=(\diver -iA\cdot )\nabla_Au.
$$

\smallskip
We also note that rotating superfluids and rotating Bose--Einstein condensates can be described through a very similar Gross--Pitaevskii model, which no longer contains the gauge $A$, and where the  applied field $ H_\ex$ is replaced by a rotation vector whose intensity can be tuned. In the regime of low enough rotation  these models can be treated  with the same techniques as those developed for Ginzburg--Landau, see  \cites{Ser3,aftalionjerrard,aftalionbook,BalJerOrlSon2,TilTil} and references therein.

Type-II superconductors are known to exhibit  several phase transitions as a function of the intensity of the applied field. The one we focus on is  the onset of  vortex-lines.   Mathematically, these  are zeroes of the complex-valued order parameter function $u$ around which $u$ has a nontrivial winding number or degree (the rotation number of its phase). 
More precisely, it is   established that  there exists a first critical field $\hci$ of order $\lep$, such that if the intensity of the applied field $h_\ex$ is below $\hci$ then the material is superconducting and $|u|$ is roughly constant equal to 1, while when the intensity exceeds $\hci$,   vortex filaments appear in the sample. The order parameter  $u$ vanishes at the core of each vortex tube and has a nonzero winding number around it.    Rigorously, this was first derived by   Alama--Bronsard--Montero \cite{AlaBroMon} in the setting of a ball.  Baldo--Jerrard--Orlandi--Soner \cites{BalJerOrlSon1,BalJerOrlSon2} derived a mean-field model for many vortices and the main order of the first critical field in the general case, and \cite{Rom2} gave  a more precise expansion of $\hci$ in the general case, and moreover  proved that global minimizers have no vortices below $\hci$, while they do above this value.
One may also point out the paper \cite{JerMonSte} that  constructs locally minimizing solutions with vortices. 

In these papers, the occurrence of the  first vortex line(s) and its precise location in the sample is connected to what we named an {\it isoflux problem}, which we studied for its own sake in \cite{RSS1} and which is described below. Moreover,  we showed in that paper that if the intensity of the magnetic field $h_\ex$ does not exceed $\hci$ by more than $K\log \lep$, then the vorticity remains bounded independently of $\ep$, i.e.~informally the total length of the curves remains bounded, and we expect only a finite number of curves. From there, it is however quite difficult to extract the optimal number of curves or the individual curve behavior. In particular the coercivity of the energy with respect to the curve locations is quite delicate to identify, as we will see.

On the other hand,  as mentioned above, the two-dimensional version of the gauged Ginzburg--Landau model, in which vortices are essentially points instead of lines, was studied in details in the mathematics literature. In particular,  in  \cites{Ser,Ser2,SanSer1,SanSer2} (see \cites{SanSerBook} for a recap) the  first critical field was precisely computed, and it was shown that under some nondegeneracy condition the vortices appear one by one near a distinguished point $p$  of the bounded domain, as the magnetic field is increased: at $\hci$ one vortex appears, then when the external field is increased by an order $ \log \lep$ a second vortex appears, then when the external field is increased by an additional order $ \log \lep$ a third vortex appears, etc. Moreover, the vortices were shown to minimize, in the limit $\ep \to 0$ and after a suitable  rescaling around $p$, an effective``renormalized" interaction energy (thus called by analogy with the renormalized energy $W$  of \cite{BetBreHel}) of the form
\begin{equation*}
   - \sum_{i \neq j} \log |x_i-x_j|+ N \sum_{i=1}^N Q(x_i),\end{equation*} where $Q$ is a positive quadratic function,
which results from the logarithmic vortex repulsion and the magnetic confinement effect, see in particular \cite{SanSerBook}*{Chapter 9}.
Specific mathematical techniques for analyzing vortices in the two-dimensional model  had been  first developed for the simplified model without magnetic field, which is obtained by setting $A\equiv 0$ and $H \equiv 0$ in the two-dimensional version of \eqref{freee}, in particular in \cites{BetBreHel,San0,Jer,JerSon} and then extended to the situation with gauge in \cites{BetRiv,Ser,Ser2,SanSer1,SanSer2}.

As alluded to above, vortices in the context of  the same simplified model without magnetic field but in three dimensions were also  analyzed in  the mathematical  literature in  \cites{Riv,LinRiv1,LinRiv2,San,BetBreOrl}.
These works demonstrated that, in the absence of magnetic field effects, vortex lines   $\Gamma$ carry a leading order  energy proportional to their length, $\pi |\ga|\lep$,  while their interaction effect is a lower order effect (of order 1). Thus, in order to minimize the energy,  vortices (which in that setting only  occur because of an imposed Dirichlet boundary condition) should be straight lines, and their interaction becomes a negligible effect in the $\ep \to 0$ limit.
It was thus not clear whether magnetic effects could suffice to curve the vortices. A formal derivation was attempted in \cite{AftRiv} in the context of Bose--Einstein condensates, proposing an effective energy where length effects and interaction effects compete.

More recently, in \cite{ConJer},  the authors found a setting where the length and the interaction compete at next order: they study the same simplified Ginzburg--Landau model without gauge in a cylindrical domain, choose the Dirichlet boundary condition to have  a degree $N$ vortex at one point (say the North pole) of the boundary and a degree $-N$ at another point (say the South pole). The energy minimizers must  then have $N$ vortex filaments which connect the two poles, moreover to minimize the leading order energy these should all be nearly parallel and close to vertical straight lines. Since the vortices  repel each other logarithmically, these lines curve a little bit when leaving the poles, in order to separate by an optimal  distance shown  to be $1/\sqrt{\lep}$. When rescaling horizontally at that lengthscale, one sees well separated vortex lines with competition between the linearization of the length and the logarithmic interaction. The authors are able to extract an effective limiting energy
\begin{equation}\label{toda}\pi \int_0^L \sum_{i=1}^N \hal |u'_i(z)|^2-\sum_{i\neq j} \log |u_i(z)-u_j(z)|dz,\end{equation} where $z :(0,L)\mapsto (u_i(z),z)$ represent the rescaled curves.
 The critical points of this energy happen to solve a
``Toda lattice" ODE system. In \cite{DavDelMedRod}, solutions of the  Ginzburg--Landau equations (without gauge) and with vortex helices which are critical points of \eqref{toda} were constructed. 
 
This setting is however, a little bit artificial due to this particular boundary condition. 
In addition, in  all the problems without magnetic gauge, the number of vortex lines ends up automatically bounded as a result of enforcing a Dirichlet boundary condition with finite vorticity. This significantly simplifies the analysis, and as in the two-dimensional case, we need to deal with a more realistic  situation  where the number of vortices may a priori be unbounded, for this we rely on \cite{RSS1}, itself relying on \cite{Rom}.  

What we do here is derive the first interaction energy where length, interaction and magnetic effects compete, in the setting of the full physical model with gauge.
In addition, we do not restrict to geometries where the filaments are almost straight. We consider general geometries and magnetic fields that can lead to the optimal location of the first vortex as $h_\ex$ passes $\hci$ -- the solution to the isoflux problem --  to be a curved line, called $\ga_0$. The next vortices  obtained by increasing slightly $h_\ex$
will be nearly parallel to $\ga_0$, hence to each other, leading to  a curved version of the situation of \cite{ConJer}. However, we have to work in coordinates aligned with $\ga_0$, which turns out to be equivalent to studying Ginzburg--Landau functionals on a manifold.  In addition, technical difficulties will arise near the boundary, at the endpoints of the vortex filaments, where these may diverge away from $\ga_0$ to meet the boundary orthogonally, in contrast with the almost parallel setup of \cite{ConJer} where all the curves meet at their (fixed) endpoints. 
 The problem thus combines all the possible technical difficulties of vortex analysis in three dimensions, in particular, dealing with a priori unbounded numbers of vortices, dealing with
estimates on manifolds, and  dealing with nonflat boundaries.
We will make intensive use of the toolbox assembled in \cite{Rom} for energy bounds in three dimensions, as well as various methods for obtaining two-dimensional estimates \cites{SanSerBook,almeidabethuel}, to which we are eventually able to  reduce modulo appropriate slicing. We also provide a completely new upper bound construction, based on the Biot-Savart law, approximating the optimal Ginzburg--Landau energy for configurations with vorticity carried by prescribed curves. This is the first upper bound construction applicable to general curves, to be compared with prior constructions in \cites{AlbBalOrl,MonSteZie,ConJer} which all involve almost straight vortex lines.

We will give further detail on the proof techniques after the statement of the main theorem.

Before stating the main theorem, we need to introduce various notions and notation.

\subsection{Energy splitting}

The Ginzburg--Landau model  admits a unique state, modulo gauge transformations,  that we will call ``Meissner state'' in reference to the Meissner effect in physics, i.e.~the complete repulsion of the magnetic field by the superconductor when the superconducting density  saturates at $|u|=1$ with no vortices. It is  obtained by minimizing   $GL_\ep(u,A)$ under the constraint $|u| = 1$, so that in particular it is independent of $\ep$.  In the  gauge  where $\diver A = 0$, this state is of the form
\begin{equation*}
   (e^{ih_\ex\phi_0},h_\ex A_0),
\end{equation*}
where $\phi_0$, $A_0$ depend only on $\Omega$ and $H_{0,\ex}$, and was  first identified in \cite{Rom2}.  
 It is not a true critical point of \eqref{modgl} (or true solution of the associated Euler--Lagrange equations \eqref{GLeq}, but is a good  approximation of one as $\ep \to 0$. 
The energy of this state is easily seen to be proportional to $h_\ex^2$, we write
\begin{equation*}
   GL_\ep(e^{ih_\ex\phi_0}, h_{\ex}A_0)=: h_{\ex}^2 J_0.
\end{equation*}
Closely related is a magnetic field $B_0$ constructed in \cite{Rom2}, whose definition will be recalled later.

The {\it superconducting current} of a pair $(u,A)\in H^1(\Omega,\C)\times H^1(\Omega,\RR^3)$ is defined as the $1$-form
\begin{equation*}
   j(u,A)=(iu,d_A u)=\sum_{k=1}^3 (iu,\partial_k u-iA_ku)dx_k
\end{equation*}
and  the gauge-invariant  {\it vorticity} $\mu(u,A)$ of a configuration $(u,A)$ as
\begin{equation*}
   \mu(u,A)=dj(u,A)+dA.
\end{equation*}
Thus $\mu(u,A)$ is an exact $2$-form in $\Omega$. 
It can also be seen as a $1$-dimensional current, which is defined through its action on $1$-forms by the relation
$$
   \mu(u,A)(\phi)=\int_\Omega \mu(u,A)\wedge \phi.
$$
The vector field corresponding to $\mu(u,A)$ (i.e. the $J(u,A)$ such that $\mu(u,A)\wedge \phi = \phi(J(u,A))\,dV$ where $dV$ is the Euclidean volume form, is at the same time a gauge-invariant analogue of twice the Jacobian determinant, see for instance \cite{JerSon}, and a three-dimensional analogue of the gauge-invariant vorticity of \cite{SanSerBook}.

The vorticity $\mu(u, A)$ is concentrated in the vortices and, in the limit $\ep \to 0$, it is exactly supported on the limit vortex lines.

We now recall the  algebraic splitting of the Ginzburg--Landau energy from \cite{Rom2}, which   allows to follow the roadmap of \cite{SanSerBook} in three dimensions: for any $(\u, \A)$,   letting $u=e^{-ih_\ex\phi_0}\u$ and $A=\A-h_\ex A_0$, we have 
   \begin{equation*}
      GL_\ep(\u,\A)=h_\ex^2 J_0+F_\ep(u,A)
      -h_\ex\int_\Omega \mu(u,A)\wedge B_0+o(1),
   \end{equation*}

This formula allows, up to a small error, to exactly separate the energy of the Meissner state $h_{\ex}^2 J_0$, the positive free energy cost $F_\ep$ and the magnetic gain $-h_{\ex} \int \mu(u,A) \wedge B_0$ which corresponds to the value of the magnetic flux of $B_0$ through the vortex or rather the loop formed by the  vortex line on the one hand, and any curve lying on $\partial \Omega$ that allows to close it, see Figure~\ref{figure0}.

\begin{figure}[ht]
   \centering
   \includegraphics[scale=0.4]{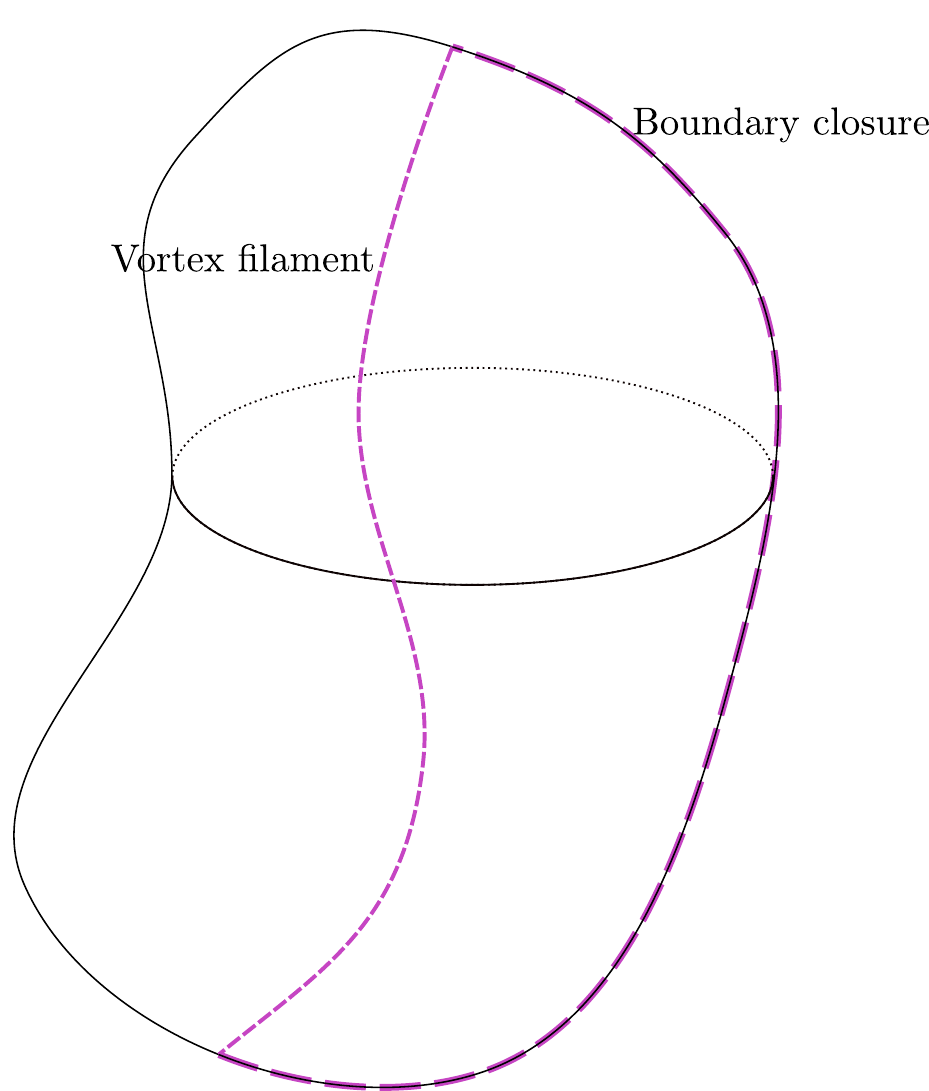}
   \caption{Vortex filament with boundary closure.}
   \label{figure0}
\end{figure}

 The first critical field is reached  when competitors with vortices  have an energy strictly less than that of the Meissner state, $ h_{\ex}^2 J_0$, that is when the magnetic gain beats the free energy cost. Approximating the energy cost of a curve $\Gamma$ by $\pi |\Gamma|\lep$ leads to the following isoflux problem.

\subsection{The isoflux problem and nondegeneracy assumption}The  isoflux problem characterizes the curves that maximize the magnetic  flux for a given length (hence the name isoflux, by analogy with isoperimetric), providing the critical value of $h_{\ex}$. 

Given a domain $\Omega \subset \RR^3$,
we  let $\NN$ be the space of normal $1$-currents supported in $\overline\Omega$, with boundary supported on $\partial\Omega$. We always denote
by $|\cdot |$ the mass of a current. Recall that normal currents are currents with finite mass whose boundaries have finite mass as well. 

We also let $X$ denote the class of currents in $\NN$ which are simple oriented Lipschitz curves. 
An element of $X$  must either be  a loop contained in $\overline\Omega$ or have its two endpoints on $\partial \Omega$.

Given $\sigma\in(0,1]$, we let $C_T^{0,\sigma}(\Omega)$ denote the space of vector fields $B \in C^{0,\sigma}(\Omega)$ such that $B\times\vec\nu=0$ on $\partial \Omega$, where hereafter $\vec\nu$ is the outer unit normal to $\partial\Omega$.  The symbol $^*$ will denote its dual space.
Such a $B$ may also be interpreted as 2-form, we will not distinguish the two in our notation.

For any vector field $B\in C_T^{0,1}(\Omega,\RR^3)$ and any  $\ga\in\NN$, we denote by $\pr {B,\ga}$ the value of $\ga$ applied to $B$, which corresponds to the circulation of the vector field $B$ on $\ga$ when $\ga$ is a curve.

We also let
\begin{equation}\label{defstar}
   \|\ga\|_*\colonequals \sup_{\|B\|_{C_T^{0,1}(\Omega,\RR^3)}\leq 1}\pr{B,\ga}
\end{equation}
be the dual norm to the norm in   $C_T^{0,1}(\Omega,\RR^3)$.

\begin{definition}[Isoflux problem]
   The isoflux problem relative to $\Omega $ and a vector field $B_0\in C_T^{0,1}(\Omega, \RR^3)$,  is  the question of maximizing  over $\NN$ the ratio
   \begin{equation}\label{defratio}
      \R(\ga)\colonequals\dfrac{\pr{B_0,\ga}}{|\ga|}.
   \end{equation}
\end{definition}

In \cite{RSS1}*{Theorem 1}, we proved that  the maximum is achieved, and under the additional condition
$$
\sup_{\mathcal C_{\mathrm{loops}}}\R<\sup_{\NN}\R,
$$
where $\mathcal C_{\mathrm{loops}}$ denotes the space of closed oriented Lipschitz curves (that is, loops) supported in $\overline \Omega$, then the supremum of the ratio $\R$ over $\NN$ is attained by an element of $X$ which is not a loop, and hence has two endpoints in $\partial \Omega$. We will denote it $\ga_0$.

A vortex line is thus seen to be  favorable if and only if
$h_{\ex} \ge H_{c_1}^0$ where
\begin{equation}\label{defhc1}
   H_{c_1}^0\colonequals \frac\lep{2\R_0},
\end{equation}
and 
\begin{equation}\label{R0}
   \R_0\colonequals \sup_{\Gamma \in X} \R(\Gamma)=\R(\ga_0).
\end{equation}
We refer the interested reader to the recent article \cite{pinning3D}, in which a weighted variation of the isoflux problem is derived within the framework of a pinned version of the 3D magnetic Ginzburg--Landau model.

\smallskip
To go further, we need to define {\em tube coordinates} around the optimal curve $\ga_0$, assumed to be smooth and meeting the boundary of $\Omega$ orthogonally at its two endpoints.

These coordinates, whose existence we prove in Proposition~\ref{prop:diffeo}, are defined in a $\delta$-tube around $\ga_0$. In this system of coordinates, $\ga_0(z)$ is mapped to $(0,0,z)$ and, if we denote by $g_{ij}$ the coefficients of the Euclidean metric in these coordinates, then we have $g_{13}=g_{23} = 0$ and $\gz(0,0,z) = 1$ for $z\in[0,\ellzero]$, where, throughout the article,  $\ellzero$ denotes the length of $\ga_0$.

We will denote by $\vu = (x,y)$ the two first coordinates and by $z$ the third coordinate.

Then $g(\vu,z)$ will denote the Euclidean metric in these coordinates, and we define $g^\bullet$ to be the metric along the $z$-axis, that is,  $g(0,0,z)$.

\begin{definition}[Strong nondegeneracy]\label{qdeu}
   We say that $\ga_0$ is a nondegenerate maximizer for the ratio $\R$ if it maximizes $\R$ and if the quadratic form
   $$Q(\vu): = -\frac{d^2}{dt^2}_{|t=0} \R(\ga_t)$$
   with $\ga_t(z)\colonequals\ga_0(z)+ t\vec{u}(z)$,
   is positive definite over the Sobolev space $H^1((0,\ellzero),\RR^2)$. We then let
   \begin{equation*}
   	\alpha_Q = \sup_{\substack{\vu\in H^1\\ \|\vu\|_{H^1}\le 1}} Q(\vu).\end{equation*}
\end{definition}
An explicit expression of $Q$ in terms of $B_0$, $\ga_0$ and $\Omega$ is given in Lemma \ref{lemformQ}.
We will show in Section \ref{sec:nondeg} that this strong nondegeneracy condition holds at least for small enough balls.

We may then define the {\it renormalized energy} of a family of curves $\ga_i^*(z) = \ga_0(z)+\vu_i^*(z)$, for $1\le i\le N$, by 
\begin{equation}\label{defW}
   W_N(\ga_1^*,\dots,\ga_N^*) = \pi\ellzero N\sum_{i=1}^N Q(\vu_i^*) - \pi\int_0^\ellzero \sum_{i\neq j} \log|\vu_i^*(z) - \vu_j^*(z)|_{g^\bullet},\end{equation}
where $|\cdot |_{g^\bullet}$ denotes the norm as measured in the $g^\bullet$ metric.

\subsection{Main theorem}
The next theorem shows that there exists a sequence of transitions at values of $\he$ that we now define.
\begin{definition}
 Given an integer $N \ge 1$, we define the $N$-th critical value by 
\begin{equation*}
H_N\colonequals \frac{1}{2\R_0}\( \lep+(N-1)\log \frac{\lep}{2\R_0} + k_N\),
\end{equation*} 
where 
$$k_N=(N-1)\log \frac1N + \frac{N^2-3N+2}{2}\log \frac{N-1}{N}+\frac1{\pi\ellzero}\Big( \min W_N -\min W_{N-1}+\gamma \ellzero+ (2N-1)C_\Omega\Big),$$
 $C_\Omega$ is a constant depending only on $\Omega$ and $H_{0,\ex}$ and defined in \eqref{comega}, and $\gamma$ is a universal constant first introduced in \cite{BetBreHel}.
\end{definition}

In particular $H_1$ will  coincide (up to $o(1)$) with $\hci$, defined as the first critical field  above  which the energy of a configuration with a vortex line becomes strictly smaller than that of the Meissner state. Then $H_2$ is the critical field above which the energy of a configuration with two vortex lines becomes strictly smaller than that with one, etc. 

The main theorem shows these transitions, and proves that  $W_N$ is the effective interaction energy of the (finite number of) vortex curves.
\begin{theorem}\label{thmain}
   Assume that the smooth simple curve $\ga_0$ is a unique nondegenerate maximizer (in the sense of  Definition \ref{qdeu}) of the ratio $\R$. There exists $c_\ep\to 0 $ as $\ep\to 0$  such that the following holds.
   Assume that 
   \begin{equation*}
   \he \in (H_N-c_\ep, H_N+c_\ep)
   \end{equation*} with $N\ge 1$ independent of $\ep$.
     
  Let $(\u,\A)$ be a minimizer (depending on $\ep$) of $GL_\ep$ in $H^1(\Omega,\C)\times [A_\ex+H_{\curl}]$ and $(u,A)=(e^{-ih_\ex\phi_0}\u,\A-h_\ex A_0 )$ as above.
   Then for any sequence $\{\ep\}$ tending to $0$, there exists a subsequence such that, for $\ep $ small enough (depending on $N$), letting $\mu_\ep^*$ be the pull-back of $\mu(u,A)$ under the horizontal rescaling map defined in the tube coordinates described above by $$(x,y,z) \mapsto  \(\sqrt{\frac{N}{h_\ex}}x, \sqrt{\frac{N}{h_\ex}} y,  z\),$$ then
   $$\lim_{\ep \to 0} \left\|\frac{\mu_\ep^*}{2\pi } - \sum_{i=1}^N \ga_i^*\right\|_* = 0,$$
   where $\ga_i^*(z) = \ga_0(z)+\vu_i^*(z)$ are $H^1$ graphs that minimize $W_N(\ga_1^*,\dots,\ga_N^*)$  as defined in \eqref{defW}.

   Moreover,    as $\ep \to 0$, defining  $K$ to  be such that 
  \begin{equation}\label{interhex}
      h_\ex = H_{c_1}^0 + K\log\lep,\end{equation} 
  we have 
   \begin{multline}\label{lbw0} 
   	GL_\ep(\u,\A)= \he^2J_0 + \frac\pi2 \ellzero  N(N-1)\log h_\ex- 2 \pi K \R_0 \ellzero N \log\lep \\-\frac{\pi}2 \ellzero N(N-1)\log N+ \min W_N+ \gamma N\ellzero+ N^2 C_\Omega+o(1),
   \end{multline}
   where $C_\Omega$ is the constant defined in \eqref{comega} and $\gamma$ is the universal constant of  \cite{BetBreHel}.
\end{theorem}
\begin{remark}
   We really show in the course of the proof that for $\he$ as in
   \eqref{interhex}, the functional 
  \begin{multline*}
  	GL_\ep-  \Big(\he^2J_0 + \frac\pi2 \ellzero  N(N-1)\log h_\ex- 2 \pi K \R_0 \ellzero N \log\lep\\-\frac\pi2 \ellzero N(N-1)\log N +\gamma N \ellzero+N^2 C_\Omega\Big) 
  \end{multline*} 
   $\Gamma$-converges to $W_N$.
\end{remark}

This theorem shows a sequence of phase transitions for intensities of the applied magnetic field equal to $H_1, H_2$, etc, at which vortex lines appear one by one, near the optimal curve $\ga_0$, and at a distance to $\ga_0$ of order $\lep^{-1/2}$. 
In particular, it shows {\it the first expansion up to $o(1)$} of the value of the first critical field:
$$H_{c_1}= \frac{1}{2\R_0}\( \lep+ \frac{\gamma\ellzero+C_\Omega}{\pi \ellzero}\)+o_\ep(1),$$
refining that of \cite{Rom2}.

The theorem also
 provides an asymptotic expansion up to order $o(1)$ of the minimal energy in the regime $h_\ex \le \hci +O(\log\lep)$, identifying 
the coefficients in factor of the leading order term in $h_\ex^2$, then the subleading order terms in $\log\lep$, and  finally the sub-subleading order terms of order $1$, which really contains the most detailed and interesting  information about the vortices, 
 showing that the curves congregate near $\ga_0$ and that one needs to zoom in ``horizontally" by a factor  $\sqrt {\frac{h_\ex}{N}}$ near $\ga_0$ to see well separated curves, which converge to a minimizer of $W_N $ in the limit. This is all consistent with the two-dimensional picture obtained in \cites{Ser,SanSer2,SanSerBook}. One may in particular compare the formulas to those in \cite{SanSerBook}*{Chapter 12}. To our knowledge, no result and description so precise can be found in the physics literature.

As a result, if $\ga_0$ is straight, which is the case for instance in a rotationally symmetric geometry with a rotationally symmetric $H_{0,\ex}$, then the vortex curves will be nearly parallel, however they will curve at next-to-leading order, especially as they get closer to the boundary at their endpoints, where they will tend to separate in order to meet the boundary orthogonally to minimize length -- in contrast with the setting of \cite{ConJer} where the curves meet at the poles; see Figure \ref{fig:vortices}. This can be compared to simulations done in rotating Bose--Einstein condensates (where the geometry is naturally rotationally symmetric), in particular the so-called ``S-shaped states" and ``U-shaped states", see \cites{aftalionjerrard,AftRiv},  physics papers \cites{RAVXK,RBD,brandt},  and numerics \cite{AftDan1}.

\begin{figure}[ht]
	\centering
	\begin{subfigure}[b]{0.45\textwidth}
		\centering
		\includegraphics[width=\textwidth]{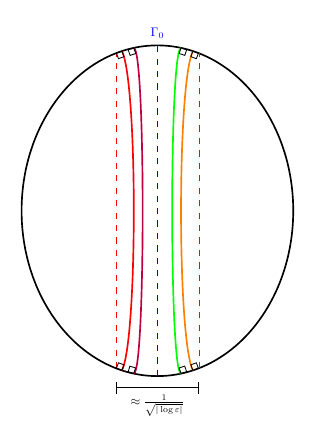}
	\end{subfigure}
	\hfill
	\begin{subfigure}[b]{0.45\textwidth}
		\centering
		\includegraphics[scale=1.6]{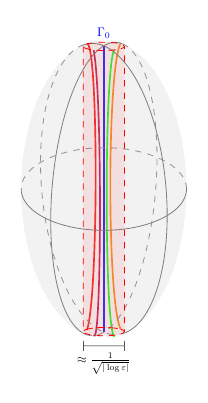}
	\end{subfigure}
	\caption{2D and 3D views of 4 vortex lines.}
	\label{fig:vortices}
\end{figure}

 If the domain and the applied field are such that $\ga_0$ is curved, then all vortex lines will also be curved, again with next-to-leading order deviations from $\ga_0$ which can in principle be estimated via $W_N $.

\subsection{Proof approach} As usual with $\Gamma$-convergence, the proof is based on establishing a general lower bound for the energy, and a matching upper bound obtained via a construction.
The main difficulty in proving the lower bound is that in order to extract the interaction energy, which is a (next-to) next-to-leading order contribution,  the energy needs to be estimated up to a $o(1)$ error as $\ep \to 0$, while the leading order is $\lep$. In contrast, all the prior three-dimensional studies of vortex filaments in Ginzburg--Landau energies, that is \cites{Riv,LinRiv1,LinRiv2,San,BetBreOrl}, had an error or a precision of $o(\lep)$,  with the exception of \cite{ConJer} in the very special setup described above.
To estimate the energetic cost $F_\ep$ of vortex lines, we crucially rely on the
approach of \cite{Rom} which is the only one with the  advantage of being $\ep$-quantitative and robust in the number of curves, i.e.~allowing it to  possibly blow up when $\ep \to 0$.  It proceeds by approximating the vorticity measure $\mu(u_\ep,A_\ep)$ by $2\pi $ times a polyhedral 1-dimensional integer-multiplicity current (in other words, a piecewise straight line), and bounding the energy $F_\ep$ from below by the mass of the current (in other words, the length of the polyhedral curve) times $\lep$. 
That approach  still involves an error at least $O(\log \lep)$, which at first seems to forbid any hope of a $o(1)$ error only.

On the other hand, one may  obtain quite precise lower bounds by slicing perpendicularly to $\ga_0$ and using two-dimensional techniques. 
However, integrating out the energy over slices  misses a part of the energy: the length part in the directions parallel  to the slices, corresponding to the  energy term $\hal\int |\partial_z u|^2$ in $F_\ep$, where $z$ is the coordinate along $\ga_0$.
To recover this energy, we use the following trick, which  is similar in spirit to   \cite{ConJer} but whose implementation somewhat differs, which allows to combine the two ways of obtaining lower bounds: we define (roughly)
$F_\ep^\perp $ as the part of the energy that contains only derivatives of $u$ tangential to the slices,  $\tilde F_\ep$ as  the energy obtained after a applying a  blow-up around $\ga_0$ of a factor $1/\ell\gg 1$ in the ``horizontal" direction perpendicular to $\ga_0$ and no blow-up in the $z$-direction parallel to $\ga_0$.  The lengthscale $\ell$ is a chosen small enough so that the vortex lines essentially (up to small excursions) remain in an $\ell$-neighborhood of $\ga_0$. A change of variables allows to write that
$$F_\ep \ge (1-\ell^2) F_\ep^\perp + \ell^2 \tilde F_\ep.$$

The part $\tilde F_\ep$ contains the missing $\int |\partial_z u|^2$ component. It can be bounded below by using the three-dimensional lower bounds of \cite{Rom}, which yields $O(\log \lep)$ errors. The crucial point is that, once multiplied by $\ell^2$ which is small, these errors become $o(1)$. 

The part $F_\ep^\perp$ is bounded below by slicing  along (curved) slices perpendicular to $\ga_0$, and integrating up  two-dimensional lower bounds obtained by the ball construction method of  \cite{SanSerBook}. These lower bounds are expressed in terms of two-dimensional effective vortices (the points which are the centers of the final balls), which a priori may in some slices  be quite different from the trace along the slice of the polyhedral approximation -- so a difficulty in the whole proof is to reconcile the various vorticity approximations. Another difficulty is that we have to use slices that are adapted to the boundary, as one approaches the endpoints of $\ga_0$, and we may lose information there.  In contrast to the three-dimensional bounds \`a la \cite{Rom}, these bounds already contain some part of the vortex interaction (i.e.~the repulsion effect).  This way all  the leading and subleading energy is recovered.

When combining with the energy upper bound obtained by our new precise construction, and performing a quadratic expansion of the magnetic term in the energy, we find a posteriori that the curves must be confined very near $\ga_0$, at distance of order $1/\sqrt{h_\ex}$ or $1/\sqrt{\lep}$.   It is only at this stage, when combining all the components to the energy, that compactness for the curves (which is a subsubleading order effect as well) can be retrieved, yielding the existence of limiting curves after blow-up.

At this stage, one may finally use even more refined two-dimensional lower bounds \`a la \cite{BetBreHel} -- more precisely we will do it by following \cites{almeidabethuel,Ser} and \cite{ignatjerrard} for the curved aspects -- which contain the exact logarithmic repulsion and can be made so precise to involve only an $o(1)$ error, provided the energy in each slice is bounded above by $O(\lep)$.
This is not known a priori but comes as a consequence of the analysis of \cite{RSS1} which shows that, for the regime of $h_\ex$ that we consider ($\he\le H_{c_1}^0+ O(\log \lep)$), the energy $F_\ep$ does remain bounded by $O( \lep)$  and the vorticity remains bounded in length.
Plugging into the prior estimates yields the energy lower bound up to $o(1)$, the optimal separation  from $\ga_0$, and the fact that the limiting curves must minimize $W_N $. 

The energy upper bound is interesting in itself, it involves an explicit construction relying on the Biot--Savart law, where again the difficulty lies in obtaining $o(1)$ precision on the energy.

\subsection{Plan of the paper} 
The paper starts in Section \ref{sec:prelim} with preliminaries on the splitting formula, the construction of a superconducting current and gauge field from the Biot--Savart law associated to a curve, energy lower bounds from \cite{Rom} and needed results from \cite{RSS1}. It then describes the tubes coordinates, the rewriting of the energy, and the horizontal rescaling.

Section~\ref{sec:graphs} gathers important preliminaries on the isoflux problem and its quadratic expansion. It proves  the nondegeneracy condition for graphs, then coercivity, then the fact that strong nondegeneracy implies weak nondegeneracy.
The section concludes by showing that the strong nondegeneracy condition holds at least for small enough balls.

In Section~\ref{sec:lowerboundslicing}, we prove two types of   energy lower bounds by horizontal slicing: one by vortex balls construction,  and a more precise one, recovering the constant order term, under a very strong upper bound assumption, by  applying two-dimensional  lower bound techniques \`a la  \cites{BetBreHel,almeidabethuel,Ser}.

Section~\ref{sec8}  is the core of the proof, it assembles the prior elements to prove  the main lower bound.
Finally, Section~\ref{sec:upperbound} is devoted to  the upper bound construction.

\noindent
\\
{\bf Acknowledgements:}
C.R. was supported by ANID FONDECYT 1231593. He thanks the Courant Institute of Mathematical Sciences and Université Paris-Est Créteil for their support and kind hospitality during the completion of part of this work. S. S.  was supported by NSF grant DMS-2000205, DMS-2247846,  and by the Simons Foundation through the Simons Investigator program.  This work was also supported by a Labex Bezout funded invitation to Universit\'e Paris-Est Cr\'eteil of the third author.

\section{Preliminaries}\label{sec:prelim}
\subsection{Vector fields, forms, currents and notation}
We introduce certain concepts and notation from the theory of currents and differential forms.
In Euclidean spaces, vector fields can be identified with $1$-forms. In particular, a vector field $F=(F_1,F_2,F_3)$ can be identified with the $1$-form $F_1dx_1+F_2dx_2+F_3dx_3$. We
use the same notation for both the vector field and the $1$-form.

Given $\alpha\in(0,1]$, we let $C_T^{0,\alpha}(\Omega)$ denote the space of vector fields $B \in C^{0,\alpha}(\Omega)$ such that $B\times \vec\nu=0$ on $\partial \Omega$, where hereafter $\vec\nu$ is the outer unit normal to $\partial\Omega$. Such a $B$ may also be interpreted as 2-form, we will not distinguish the two in our notation.

It is worth recalling that the boundary of a $1$-current $T$ relative to a set $\Theta$ is a $0$-current $\partial T$, and that $\partial T=0$ relative to $\Theta$ if $T(d\phi)=0$ for all $0$-forms $\phi$ with compact support in $\Theta$.
In particular, an integration by parts shows that the $1$-dimensional current $\mu(u,A)$ has zero boundary relative to $\Omega$.

We let $\mathcal D^k(\Theta)$ be the space of smooth $k$-forms with compact support in $\Theta$. For a $k$-current $T$ in $\Theta$, we define its mass by
\begin{equation*}
   |T|(\Theta)\colonequals \sup \left\{T(\phi) \ | \ \phi\in \mathcal D^k(\Theta), \ \|\phi\|_{L^\infty}\leq 1\right\}
\end{equation*}
and by
\begin{equation}\label{flatnormdef}
   \|T\|_{\F(\Theta)}\colonequals\sup \left\{T(\phi) \ | \ \phi\in \mathcal D^k(\Theta), \ \max\{\|\phi\|_{L^\infty},\|d\phi\|_{L^\infty}\}\leq 1\right\}
\end{equation}
its flat norm.
\begin{remark}
   For $0$-currents, the flat and $(C_0^{0,1})^*$ norms coincide, whereas for $k$-currents the former is stronger than the latter.

   It is not completely obvious from the definitions \eqref{defstar}  and \eqref{flatnormdef} that $\|\ga\|_*\le C\|\ga\|_{\F(\Omega)}$, since $\|\ga\|_*$ involves testing with vector fields that are not necessarily compactly supported in $\Omega$. Nevertheless it is true if $|\ga|$ is assumed to be bounded, because if  $\|X\|_{L^\infty}, \|\nabla X\|_{L^\infty}\le 1$, then we may consider $X_n(\cdot) = (1-n\dist(\cdot,\partial\Omega))_+ X(\cdot)$, and we have $\pr{X_n,\ga}\to\pr{X,\ga}$ as $n\to +\infty$ while $\|\curl X\|_{L^\infty}$ remains bounded independently of $n$ if $X$ is normal to $\partial\Omega$.
\end{remark}
\subsection{Reference magnetic field and splitting formula}
We may now recall the definition of 
the magnetic field   $B_0$ that we will work with, constructed in
\cite{Rom2}. It appears in the Hodge decomposition of $A_0$ in $\Omega$, where $(e^{ih_\ex\phi_0},h_\ex A_0)$ is the Meissner state,  as $A_0 = \curl B_0 + \nabla \phi_0$,  supplemented with the conditions $\diver B_0 = 0$, and $B_0\times\vec\nu = 0$ on $\partial\Omega$. Moreover, it is such that
\begin{equation*}
   \int_\Omega (- \Delta B_0+ B_0-H_{0,\ex})\cdot A=0,
\end{equation*}
for any divergence-free $A\in C_0^\infty(\Omega,\RR^3)$.  Also, we recall that $\phi_0$ is supplemented with the conditions $\int_\Omega \phi_0=0$ and $\nabla \phi_0\cdot \vec \nu=A_0\cdot \vec \nu$ on $\partial\Omega$.

We now recall the  precise algebraic splitting of the Ginzburg--Landau energy from \cite{Rom2}.
\begin{propositionn} For any sufficiently integrable $(\u,\A)$, letting $u=e^{-ih_\ex\phi_0}\u$ and $A=\A-h_\ex A_0$, where $(e^{ih_\ex\phi_0},h_\ex A_0)$ is the approximate Meissner state, we have
   \begin{equation}\label{Energy-Splitting}
      GL_\ep(\u,\A)=h_\ex^2 J_0+F_\ep(u,A)
      -h_\ex\int_\Omega \mu(u,A)\wedge B_0+r_0,
   \end{equation}
   where $F_\ep(u,A)$ is as in \eqref{freee} and
   $$
      r_0=\frac{h_\ex^2}2 \int_\Omega (|u|^2-1)|\curl B_0|^2.
   $$
   In particular, $|r_0|\leq C\ep h_\ex^2F_\ep(|u|,0)^{\frac12}$.
\end{propositionn}

\subsection{Biot-Savart vector fields and a new constant}

\begin{definition} \label{defbs}The {\em Biot--Savart vector field} associated to a smooth simple closed curve in $\RR^3$ is
   \begin{equation}\label{biotsavart} 
      X_\ga(p) = \frac12\int_t \frac{\ga(t) - p}{|\ga(t) - p|^3}\times \ga'(t)\,dt.
   \end{equation}
   It is divergence free, satisfies $\curl X_\ga = 2\pi\ga$, and belongs to $L^p_\loc$ for any $1\le p<2$.

   Moreover, denoting $p_\ga$ the nearest point to $p$ on $\ga$, and $U$ a bounded neighborhood of $\ga$ on which this projection is well defined, the difference
   \begin{equation}\label{approxBS}X_\ga(p) - \frac{p_\ga - p}{|p_\ga - p|^2}\times\ga'(p_\ga)
   \end{equation}
   is in $L^q(U)$, for any $q\ge 1$.
\end{definition}

The approximation \eqref{approxBS} is classical (see for instance \cite{AKO}) and may be derived from \eqref{biotsavart}. The difference is in fact $O(\log|p-p_\ga|)$.

In the next proposition, to any nice curve $\ga$, we associate via $X_\Gamma$ a current and magnetic gauge pair. This will be in particular useful for the upper bound construction.  
\begin{proposition}\label{jA}
   Assume $\Omega\subset \RR^3$ is a smooth and bounded domain. Assume $\ga$ is a smooth simple closed curve in $\RR^3$ which intersects $\partial\Omega$ transversally.

   Then there exists a  unique  divergence free $j_\ga:\Omega\to \RR^3$, belonging to $L^p$ for any $p<2$, and a unique divergence free $A_\ga\in H^1(\RR^3,\RR^3)$ such that,  $\curl(j_\ga+A_\ga)= 2\pi\ga$ in $\Omega$, such that $\nu\cdot j_\ga = 0$ on $\partial\Omega$, and such that the following equation is satisfied in the sense of distributions in $\RR^3$:
   \begin{equation*}
      \Delta A_\ga  + j_\ga\mathbf 1_\Omega = 0.
   \end{equation*}
   In particular, $j_\ga$ and $A_\ga$ only depend on $\ga\cap\Omega$. It also holds that $A_\ga\in W^{2,\frac32}_\loc(\RR^3,\RR^3)$ and that $j_\ga-X_\ga \in W^{1,q}(\Omega)$, for any $q<4$.

   Moreover, there exists $1<p<2$ such that if $\ga$ and $\ga'$ are two curves as above then,
  \begin{multline}\label{diffjA}\int_{\RR^3}|\nabla A_\ga - \nabla A_{\ga'}|^2 +\int_{\Omega} |j_\ga - X_\ga - (j_{\ga'}-X_{\ga'})|^2\\
  	 \le C (n(\ga)+n(\ga'))\(\|X_\ga- X_{\ga'}\|_{L^p(\partial\Omega)}+ \|X_\ga-X_{\ga'}\|_{L^p(\Omega)}\),
  \end{multline}
  where 
  $$ n(\ga)=\|X_\ga\|_{L^{p}(\partial\Omega)}+ \|X_\ga\|_{L^{p}(\Omega)}.$$

\end{proposition}

\begin{proof} Assume $\ga$ is as above. For $(f,A)\in H^1(\Omega)\times H^1(\RR^3,\RR^3)$ let
   $$I_\ga(f,A) = \hal\int_{\RR^3}|\nabla A|^2+\hal\int_\Omega|\nabla f - A|^2 + \int_{\partial\Omega}fX_\ga\cdot\nu - \int_\Omega X_\ga\cdot A.$$
   When minimizing $I_\ga$ (recall that $X_\Gamma$ is divergence-free), we have the freedom of adding a constant to $f$ hence we may assume that $f$ has zero average over  $\Omega$. Then, we use the embeddings $H^\hal(\partial\Omega)\hookrightarrow L^4$ and $H^1(\RR^3)\hookrightarrow L^6$ to find 
   $$\int_{\partial\Omega}\left|fX_\ga\cdot\nu\right|\le C\|X_\ga\|_{L^{\frac32}(\partial\Omega)}\|\nabla f\|_{L^2(\Omega)},\quad  \left|\int_\Omega X_\ga\cdot A\right|\le C\|X_\ga\|_{L^{\frac32}(\Omega)}\|\nabla A\|_{L^2(\RR^3)}.$$
   It is not difficult then to check that the infimum of $I_\ga$, which is nonpositive, is achieved by some couple $(f_\ga,A_\ga)$. Moreover, still assuming $f_\ga$ to have zero average over $\Omega$,  we have 
   \begin{equation}\label{estAf}\int_{\RR^3}|\nabla A_\ga|^2, \int_\Omega|\nabla f_\ga - A_\ga|^2, \int_\Omega|\nabla f_\ga|^2\le C\(\|X_\ga\|_{L^{\frac32}(\partial\Omega)}+ \|X_\ga\|_{L^{\frac32}(\Omega)}\).
   \end{equation}

   We let
   $$j_\ga =  X_\ga +\nabla f_\ga- A_\ga.$$
      The Euler--Lagrange equations for $I_\ga$ can then be written as
   \begin{equation}\label{JA}
      \left\{\begin{aligned}
         \Delta A_\ga +j_\ga\mathbf 1_\Omega & = 0\quad       & \text{ in $\RR^3$}   \\
         \diver(j_\ga \mathbf 1_\Omega)      & = 0 \quad      & \text{ in $\RR^3$}   \\
         \curl(j_\ga + A_\ga)                & = 2\pi\ga\quad & \text{ in $\Omega$.}
      \end{aligned}\right.
   \end{equation}
   From \eqref{estAf} and the definition of $j_\ga$ we deduce a bound for $\|j_\ga\|_{L^{\frac32}}$, which in turn implies using the equation for $A_\ga$ that $A_\ga\in W^{2,\frac32}_\loc$ and that 
   $$ \|A_\ga\|_{W^{2,\frac32}(\Omega)}\le C\(\|X_\ga\|_{L^{\frac32}(\partial\Omega)}+ \|X_\ga\|_{L^{\frac32}(\Omega)}\).$$
   Taking the divergence of the first equation in \eqref{JA}, we also find that $A_\ga$ is divergence-free in $\RR^3$.
   
   Moreover, the function $f_\ga$ is harmonic, and its normal derivative on $\partial\Omega$ is  $(X_\ga - A_\ga)\cdot \nu$,  which belongs to $L^p(\partial\Omega)$ for any $p<2$. Since $L^2(\partial\Omega)$ embeds into $W^{-\hal,4}(\partial\Omega)$, we deduce that $f_\ga\in W^{1,q}(\Omega)$ for any $q<4$ (see for instance \cite{Lieberman}). Together with the fact that $A\in W^{2,\frac32}_\loc$, this implies that $j_\ga - X_\ga$ belongs to  $W^{1,q}(\Omega)$ for any $q<4$. We also deduce along these lines the estimate that for some $p<2$
   $$\|f_\ga\|_{L^3(\partial\Omega)} \le C \|X_\ga\|_{L^p(\partial\Omega)}.$$

   To check uniqueness, assume both $(j,A)$ and $(j',A')$ satisfy \eqref{JA} and let $B = A-A'$ and $k = j-j'$. Then $\curl(k+B) = 0$ in $\Omega$ hence there exists $g$ such that $k+B=\nabla g$. From the equation $-\Delta B = k\mathbf 1_\Omega$ integrated against $B$ we find that
   $$\|B\|_{L^2(\RR^3)}^2 = \langle k,B\rangle_{L^2(\Omega)}.$$
   Since $\diver k = 0$ and $k\cdot\nu = 0$ on $\partial\Omega$, it holds that $\nabla g$ and $k$ are orthogonal in $L^2(\Omega)$. Thus we have
   $$\langle k,B\rangle_{L^2(\Omega)} = \langle k,\nabla g - k\rangle_{L^2(\Omega)}= -\|k\|_{L^2(\Omega)}^2.$$
   It follows that $B$ and $k$ are equal to $0$.

   To prove the last assertion of the proposition, assume that $(j,A) = (j_\ga,A_\ga)$ and that $(j',A') = (j_{\ga'},A_{\ga'})$. Then $D(I_\ga)_{(f,A)}(f-f',A-A') = 0$ and $D(I_{\ga'})_{(f',A')}(f-f',A-A') = 0$. Taking the difference of the two identities we find that
   $$ \int_{\RR^3}|\nabla(A-A')|^2+\int_\Omega|\nabla(f-f') - (A-A')|^2 = \int_{\partial\Omega}(f-f') (X_\ga - X_{\ga'})\cdot\nu - \int_\Omega (X_\ga - X_{\ga'})\cdot(A-A').$$
   Then we use the estimates for $ \|A_\ga\|_{W^{2,\frac32}(\Omega)}$ and $\|f_\ga\|_{L^3(\partial\Omega)}$ to find that \eqref{diffjA} holds, noting that $j - X_\ga = \nabla f -A$ and similarly for $j' - X_{\ga'}$.
\end{proof}

We may now define the constant $C_\Omega$ which appears in the expansion of the energy of minimizers of the Ginzburg--Landau energy. It is the equivalent of the renormalized energy of \cite{BetBreHel} for the optimal  curve $\ga_0$.

\begin{definition} Assume $\Omega$ is a smooth bounded domain such that the maximum of the ratio is achieved at a smooth simple curve $\ga_0$ which can be extended to a smooth simple closed curve in $\RR^3$, still denoted $\ga_0$. We set $A_\Omega = A_{\ga_0}$ and $j_\Omega = j_{\ga_0}$ where $(A_{\ga_0},j_{\ga_0})$ are given in Proposition~\ref{jA}.  We define 
\begin{equation}\label{comega}C_\Omega = \frac12\int_{\RR^3}|\curl A_\Omega|^2 + \lim_{\rho\to 0}\left(\frac12\int_{\Omega\setminus \overline{T_\rho(\ga_0)}} |j_\Omega|^2 +\pi\ellzero\log\rho\right),\end{equation}
where $T_\rho(\ga_0)$ denotes the tube of radius $\rho$ around $\ga_0$, intersected with $\Omega$.
\end{definition}

\subsection{\texorpdfstring{$\ep$}{ε}-level lower bounds}

We next recall the $\ep$-level estimates provided in \cites{Rom,RSS1}.
\begin{theoremm}\label{theorem:epslevel} For any $m,n,M>0$ there exist $C,\ep_0>0$ depending only on $m,n,M,$ and $\partial\Omega$, such that, for any $\ep<\ep_0$, if $(u_\ep,A_\ep)\in H^1(\Omega,\C)\times H^1(\Omega,\RR^3)$ is a configuration such that $F_\ep(u_\ep,A_\ep)\leq M|\log\ep|^m$ then there exists a polyhedral $1$-dimensional current $\nu_\ep$ and a measurable set $S_{\nu_\ep}$ such that
   \begin{enumerate}[leftmargin=*,font=\normalfont]
      \item $\nu_\ep /2\pi$ is integer multiplicity.
      \item $\partial \nu_\ep=0$ relative to $\Omega$,
      \item $\mathrm{supp}(\nu_\ep)\subset S_{\nu_\ep}\subset \overline \Omega$ with  $|S_{\nu_\ep}|\leq C|\log\ep|^{-q}$, where  $|\cdot |$ denotes the measure of the set and $q(m,n)\colonequals\frac32 (m+n)$,
      \item
            $$
               \int_{S_{\nu_\ep}}|\nabla_{A_\ep} u_\ep|^2+\frac{1}{2\ep^2}(1-|u_\ep|^2)^2
               \geq  |\nu_\ep|(\Omega)\left(\log \frac1\ep-C \log \log \frac1\ep\right)-\frac{C}{|\log\ep|^n},
            $$
      \item for any $\sigma\in(0,1]$ there exists a constant $C_\sigma$ depending only on $\sigma$ and $\partial\Omega$, such that
            \begin{equation*}
               \|\mu(u_\ep,A_\ep) -\nu_\ep\|_{C_T^{0,\sigma}(\Omega)^*}\leq C_\sigma\frac{F_\ep(u_\ep,A_\ep)+1}{|\log \ep|^{\sigma q}};
            \end{equation*}
      \item and for any $\alpha\in (0,1)$,
            \begin{equation*}
               \|\mu(u_\ep,A_\ep) -\nu_\ep\|_{\F(\Omega_\ep)}\leq C \frac{F_\ep(u_\ep,A_\ep)+1}{|\log \ep|^{\alpha q}},
            \end{equation*}
                      where the flat norm was defined in \eqref{flatnormdef} and 
            $$
               \Omega_\ep\colonequals \left\{x\in \Omega \ | \ \mathrm{dist}(x,\partial\Omega)> |\log\ep|^{-(q+1)}\right\}.
            $$
   \end{enumerate}
\end{theoremm}

\subsubsection{Bounded vorticity}
We recall that $X$ denotes the class of oriented Lipschitz curves, seen as $1$-current with multiplicity $1$, which do not self intersect and which are either a loop contained in $\Omega$ or have two different endpoints on $\partial \Omega$. We also recall that $\NN$ be the space of normal $1$-currents supported in $\overline\Omega$, with boundary supported on $\partial\Omega$.
\begin{condition}[Weak Nondegeneracy condition]\label{nondegencond} There exists a unique curve $\ga_0$ in $X$ such that
   \begin{equation*}
      \R(\ga_0) = \sup_{\ga\in\NN}\frac{\pr{B_0,\ga}}{|\ga|}.
   \end{equation*}
   Moreover, there exists constants $c_0,P>0$ depending on $\Omega$ and $B_0$ such that
   \begin{equation}\label{cond2}
      \R(\ga_0)-\R(\ga)\geq C_0\min\left(\|\ga-\ga_0\|_*^P,1\right)
   \end{equation}
   for every $\ga\in X$.
\end{condition}
We will see in Section \ref{secstrongweak} that the strong nondegeneracy condition of Definition \ref{qdeu} implies this one.
We now state a result adapted from \cite{RSS1} that, under this weak nondegeneracy condition, locates the vortices of almost minimizers near $\ga_0$. 
\begin{theoremm}\label{teo:boundedvorticity}
   Assume that Condition \ref{nondegencond} holds. Then, for any $K>0$ and $\alpha\in(0,1)$, there exists positive constants $\ep_0,C>0$ depending on $\Omega$, $B_0$, $K$, and $ \alpha$, such that the following holds.

   For any $\ep<\ep_0$ and any $h_\ex<H_{c_1}^0+K\log |\log \ep|$, if $(\u,\A)$ is a configuration in $H^1(\Omega,\C)\times [A_\ex+H_{\curl}]$ such that $GL_\ep(\u,\A)\leq h_\ex^2 J_0$,
   then, letting $(u,A)=(e^{-ih_\ex \phi_0}\u,\A-h_\ex A_0)$, there exists ``good'' Lipschitz curves $\ga_1,\dots,\ga_{N_0}\in X$ and $\tga\in\NN$ such that $N_0\le C$ and
   \begin{enumerate}[leftmargin=*,font=\normalfont]
      \item for $1\le i\le N_0$, we have
            $$
               \R(\ga_0) - \R(\ga_i)\le C\frac{\log\lep}{\lep};
            $$
      \item for $1\le i\le N_0$, we have
            $$
               \|\ga_i-\ga_0\|_*\leq C\(\dfrac{\log\lep}{\lep^\alpha}\)^\frac1P \quad\mathrm{and}\quad \big||\ga_i|-\ellzero\big|\leq C\(\dfrac{\log\lep}{\lep^\alpha}\)^\frac1P;
            $$
      \item $\tga$ is a sum in the sense of currents of curves in $X$ such that
            $$
               |\tga|\leq C\frac{\log\lep}{\lep^{1-\alpha}};
            $$
      \item we have
            \begin{equation*}
               \left\|\mu(u,A)-2\pi\sum_{i=1}^{N_0} \ga_i-2\pi \tga\right\|_{\(C_T^{0,1}(\Omega)\)^*}\leq C\lep^{-2}
            \end{equation*}
            and
            $$
            \left\|\mu(u,A)-2\pi\sum_{i=1}^{N_0} \ga_i-2\pi \tga\right\|_{\F(\Omega_\ep)}\leq C\lep^{-2},
            $$
            where \begin{equation*}
            \Omega_\ep\colonequals \left\{x\in \Omega \ | \ \mathrm{dist}(x,\partial\Omega)> |\log\ep|^{-2}\right\}. \end{equation*}
   \end{enumerate}
\end{theoremm}

\begin{remark}
In \cite{RSS1}*{Theorem 3} we did not state the vorticity estimate for the flat norm above. However, it appears from the proof that
$$
\nu_\ep=\sum_{i=1}^{N_0}\ga_i+\tga
$$
and that the vorticity estimate stated in \cite{RSS1}*{Theorem 3} directly follows from applying Theorem \ref{theorem:epslevel}. A straightforward use of this theorem then also yields the vorticity estimate for the flat norm stated above.

\end{remark}

\subsection{Tube coordinates, energy rewriting and horizontal rescaling}\label{sec:coordinates}
\subsubsection{Tube coordinates}
\begin{proposition}\label{prop:diffeo}
   Assume $\ga_0:[0,\ellzero]\to\overline\Omega$ is a smooth curve parametrized by arclength with endpoints $p=\ga_0(0)$ and $q=\ga_0(\ellzero)$ on $\partial\Omega$ and meeting the boundary orthogonally. Then there exists $\delta>0$ and coordinates defined in $T_\delta\cap \overline \Omega$, where $T_\delta$ is the tube of radius $\delta$ around $\ga_0$ (see remark below), such that:
   \begin{itemize}
      \item[-] In these coordinates, $z\mapsto (0,0,z)$ is an arclength parametrization of $\ga_0$.
      \item[-] Denoting by $g(x,y,z)$ the Euclidean metric in these coordinates, we have $g_{13}=g_{23} = 0$ and, from the previous property, $\gz(0,0,z) = 1$ for $z\in[0,\ellzero]$.
      \item[-] These coordinates map a neighborhood of $p$ in $\partial\Omega$ into $\RR^2\times\{0\}$ and  a neighborhood of $q$ into $\RR^2\times\{\ellzero\}$.
   \end{itemize}
   We will denote by $\Cd$ the coordinate patch corresponding to  $T_\delta\cap \overline \Omega$, and by $\Phi:\Cd\to T_\delta\cap \overline \Omega$ the inverse of the coordinate chart.
\end{proposition}

\begin{remark}
   To define $T_\delta$ we first need to extend $\ga_0$ a little bit outside $\Omega$. Then the tube really refers to this extended curve, so that $T_\delta\cap \overline \Omega$ indeed contains a neighborhood in $\partial\Omega$ of each of the endpoints of $\ga_0$.
\end{remark}

\begin{remark}
   If $\ga_0$ is a critical point of the ratio $\R$, it is easy to check using the fact that $B_0\times\nu = 0$ on $\partial\Omega$ that $\ga_0$ intersects $\partial\Omega$ orthogonally.
\end{remark}

\begin{proof}[Proof of the proposition]
   Let $s\mapsto\ga_0(s)$ be a parametrization of $\ga_0$ by arclength on the interval $[0,\ellzero]$. We define $f(x) = s$ if $x = \ga_0(s)$, let $f = 0$ in a neighborhood of $p\colonequals\ga_0(0)$ in $\partial\Omega$ and let $f = \ellzero$ in a neighborhood of $q\colonequals\ga_0(\ellzero)$ in $\partial\Omega$. The function $f$ may be extended smoothly in a neighborhood $W$ of $\ga_0$ in $\overline\Omega$ in such a way that the level sets of $f$ meet $\ga_0$ orthogonally  and, restricting the neighborhood if necessary, we may assume its gradient does not vanish there.

   If $\eta>0$ is small enough, then $\Sigma_\eta = \bar B(p,\eta)\cap\partial\Omega$ is included in $W$ and diffeomorphic to the disk $\bar D(0,\eta)$, and we let $\varphi$, defined in $ \bar D(0,\eta)$ be such a  diffeomorphism. For $(x,y,z)\in \cyl\colonequals\bar D(0,\eta)\times[0,\ellzero]$,  we define $\Phi(x,y,z)$ to be equal to $\gamma(z)$, where $\gamma$ is the integral curve of the vector field $\nabla f/|\nabla f|^2$ originating at $\varphi(x,y)$, so  that $f(\gamma(z))' = 1$, and hence $f(\Phi(x,y,z)) = z$. In particular, $z\mapsto\Phi(0,0,z)$ is a parametrization of $\ga_0$ by arclength.

   If $\eta$ is chosen small enough, then $\Phi$ is well-defined, injective  and smooth, and its differential is invertible. Moreover $\cyl\cap\{z=0\}$ and $\cyl\cap\{z=\ellzero\}$ are mapped to the boundary of $\Omega$, since $f(\Phi(x,y,z)) = z$. Therefore $\Phi$ is  a diffeomorphism  from $\cyl$ to a neighborhood of $\ga_0$ in $\overline\Omega$. 

   We choose $\delta>0$ small enough so that $T_\delta\cap\overline \Omega\subset \Phi(\cyl)$. Then the coordinate system defined by $\Phi$ has all the desired properties. The fact that, if $g$ denotes the pull-back of the Euclidean metric by $\Phi$, we have $g_{13}=g_{23} = 0$ follows from the fact that $\Phi(\cdot,\cdot,z)$ is mapped to $\{f = z\}$, hence $\partial_x\Phi$ and $\partial_y\Phi$ are orthogonal to $\nabla f$, hence to $\partial_z\Phi$.
\end{proof}

\subsubsection{Energy in the new coordinates}
Then we express the quantities of interest in the new coordinates: $g$ is the metric, pull-back of the Euclidean metric by $\Phi$. Given a configuration $(u,A)$ in $\Omega$, the order parameter transforms as a scalar field and $A$ as a $1$-form. The field $B_0$ transforms as a one-form. Keeping the same notation for the quantities in the new coordinates, we may define the superconducting current $j(u,A) = (iu,du- iAu)$ as in the original coordinates, it is a $1$-form, and the vorticity $\mu(u,A) = dj(u,A) + dA$, which is a two-form. The new coordinates are defined in $\Ud = T_\delta\cap\overline\Omega$, using the notation of the previous proposition, and we let 
$$\Cd = \Phi^{-1}(\Ud).$$

We define
$$ \fep(u,A,\Ud) = \frac12 \int_\Ud e_\ep(u,A),$$
where
$$e_\ep(u,A)=\frac12\(|\nabla_{A}u|^2+\frac1{2\ep^2}(1-|u|^2)^2+|\curl A|^2\).$$

In the new coordinates, this becomes
$$ \fep(u,A,\Ud) = \frac12 \int_\Cd  |du - iAu|_g^2+\frac1{2\ep^2}(1-|u|^2)^2+|d A|_g^2\,d\vol_g,$$
where  $\vol_g$ is the volume form relative to the metric $g$.

Given a $k$-form $\omega$, its norm relative to $g$ is defined by the identity $|\omega|_g^2 \,d\vol_g= \omega\wedge *_g\omega$. Alternatively, if $\e_1,\e_2,\e_3$ is a $g$-orthonormal frame, $|\omega|_g^2 = \sum \omega(\e_{i_1},\dots,\e_{i_k})^2$, where the sum runs over ordered $k$-tuples $1\le i_1<\dots<i_k\le 3$.

\subsubsection{Vertical and perpendicular parts of the energy}
The energy in the new coordinates splits as follows. We let $\gperp = g - \gz\,{dz}^2$.
We have
$$|dA|_g^2 = |dA|_{\gperp}^2+ \sum_{1\le i,j\le2} (dA)_{i3}(dA)_{j3}g^{ij}g^{33}. $$
Note that, since $(g_{ij})_{1\le i,j\le 2}$ is a positive definite matrix, then so is its inverse matrix and hence the sum above is a nonnegative number.

Then we decompose
\begin{align*}
	   \fep(u,A,\Ud)=\fperp (u,A)+ \fpar (u,A),
\end{align*}
where, writing $\Sigma_z$ for the intersection of $\Cd$ with the horizontal plane at height $z$,
\begin{align*}
   \fperp(u,A) & \colonequals\int_{z=0}^{\ellzero}\int_{\Sigma_z} \eperp(u,A)\,d\vol_\gperp\,dz, \\
  \fpar(u,A)  & \colonequals\int_\Cd \epar(u,A)\,d\vol_g, 
\end{align*}
and
\begin{align}\label{defeperp}
   \eperp(u,A) & \colonequals \frac12  \(|du - iAu|_\gperp^2+\frac1{2\ep^2}(1-|u|^2)^2+|d A|_\gperp^2\)\sqrt\gz,                   \\
   \epar(u,A)  & \colonequals \frac12 \( g^{33} |\partial_z u - i A_z u|^2 + \sum_{1\le i,j\le2} (dA)_{i3}(dA)_{j3}g^{ij}g^{33}\).\notag
\end{align}

\subsubsection{Horizontal blow-up}
We now perform the announced horizontal rescaling, at a scale $\ell$ which will be optimized later.
Given $\ell>0$, we now consider in $\Cd$ the metric $\tilde g$  defined by  $\tilde g_{ij} = \ell^{-2} g_{ij}$ if $1\le i,j\le2$ and $\tilde g_{ij} = g_{ij}$ otherwise (recall that $g_{13} = g_{23} = 0$).
The volume element for $\tilde g$ is  $d\vol_{\tilde g} = {\ell^{-2}}d\vol_g$. We let $\tilde\ep = \ep/\ell$,
and
\begin{equation}\label{deftfep}
   \tilde \fep(u,A) = \frac12 \int_\Cd|du - iA u|_{\tilde g}^2+\frac1{2\tilde\ep^2}(1-|u|^2)^2+|d A|_{\tilde g}^2\,d\vol_{\tilde g}.
\end{equation}
We have
\begin{align*}
   |du - iA u|_{\tilde g}^2 & = \ell^2|du - iA u|_\gperp^2 + g^{33} |\partial_z u - i A_z u|^2,                    \\
   |d A|_{\tilde g}^2       & = \ell^4 |dA|_{\gperp}^2+ \ell^2 \sum_{1\le i,j\le2} (dA)_{i3}(dA)_{j3}g^{ij}g^{33},
\end{align*}
so that,  if $\ell \le 1$,  then
\begin{equation}\label{decompose}\ell^2\tilde \fep(u,A)\le \ell^2\fperp(u,A) + \fpar(u,A).\end{equation}
Therefore
\begin{equation}\label{decompose2}\fep(u,A,\Ud)\ge (1-\ell^2)\fperp(u,A)+\ell^2\tilde \fep(u,A).\end{equation}

\section{Preliminaries on the ratio function}\label{sec:graphs}
\subsection{Non-degeneracy condition for graphs and piecewise graphs}
Here we compute the second derivative of the ratio $\ga\mapsto \R(\ga)$ for $\ga$'s that are graphs over $\ga_0$. In addition, we consider a more general class of curves, which we call piecewise graphs over $\ga_0$, and which naturally appear in our setting since, by Theorem~\ref{teo:boundedvorticity}, the approximation of the vorticity is essentially composed of a sum of $N_0$ such objects. We provide a kind of nondegeneracy condition for the ratio among piecewise graphs, and then establish a relation between the second derivative of the ratio and this nondegeneracy condition.

Throughout this section, $B_0$ is the Meissner magnetic field associated to the domain $\Omega$. It is in particular true that the restriction of $B_0$ to a plane tangent to  the boundary of $\Omega$ is zero. Without changing notation, we will work in the coordinates defined in Proposition~\ref{prop:diffeo} in a neighborhood of $\ga_0$. In these coordinates, $\ga_0$ is the interval $[0,\ellzero]$ on the $z$-axis.

We have $\R(\ga) = \pr{B_0,\ga}/|\ga|$ where, in the new coordinates, given a curve $\ga$ in $\Cd$,
$$|\ga| = \int |\ga'(s)|_g\,ds,\quad \pr{B_0,\ga} = \int_\ga B_0.$$

For a map $f$ defined on $\Cd$, we denote by $f_z^\bullet$ the map
$$f_z^\bullet(x,y,z) = f(0,0,z).$$
The map $f_z^\bullet$ may be seen either as a function of $x$, $y$, $z$, or just of $z$. Recall from the introduction that for the metric evaluated along the axis $g(0,0,z)$, the variable $z$ is omitted from the notation, and we write $g^\bullet$ instead of $g^\bullet_z$.

We use the notation $\vu = (x,y,0)$ and $d\vu = dx\,dy$. For any curve or sum of curves $\ga$, we define
\begin{align*}
   \pperp{\ga}& \colonequals \int_\ga  \frac12(d\gz)^\bullet_z(\vu)\,dz\\
    \prll{\ga}& \colonequals \int_\ga  \Ll(z,\vu)\ dz,
\end{align*}
where the linearization of the length is defined as 
\begin{equation}\label{defLl}
	\Ll(z,\vu) =\frac14(d^2\gz)^\bullet_z(\vu,\vu) - \frac18 (d\gz)^\bullet_z(\vu)^2.
\end{equation}

\begin{remark} Note that
   \begin{equation}\label{suplagl} \Ll(z,\vu)\le C|\vu|^2.\end{equation}
\end{remark}

Hereafter $\chi$ is a smooth cutoff function on $\Cd$ taking values in $[0,1]$ such that
\begin{equation}\label{propchi} 
	\text{$\chi (\cdot) = 1$  if $\dist_{g}(\cdot,\ga_0)<\delta/2$}, \quad |\nabla\chi|\le C/\delta,\quad  \text{$\chi (\cdot)= 0$ if $\dist_g(\cdot ,\ga_0)> 3\delta/4$.}
\end{equation}

For any $2$-form $\mu = \mu_{12}\,dx\wedge dy +  \mu_{23}\,dy\wedge dz+  \mu_{31}\,dz\wedge dx$, we let
\begin{equation}\label{defprem}
	\prem{\mu} =  \int_z\int_{\Sigma^z} \chi \sqrt\gz\ \mu_{12}\,d\vu\,dz = \pr{\chi\frac{\partial}{\partial z},\mu}.\end{equation}

We also define, for any curve or sum of curves $\Gamma$, 
\begin{equation*}
   \prem{\ga} \colonequals \int_\ga \chi \ \sqrt\gz \,dz.
\end{equation*}

The following lemma motivates the notation $\prem{\ga}$.

\begin{lemma}
   Assume $\ga$ is a curve or sum of curves,  and that
   $$\left\|\mu - 2\pi\ga\right\|_*\le \eta,$$
   where $\|\cdot\|_*$ is defined in \eqref{defstar}. Then 
   $$\left|\prem{\mu}-2\pi \prem{\ga}\right|\leq C\eta.$$
\end{lemma}

\begin{proof}
   We have
   $$\prem{\mu}  =  \int_z\int_{\Sigma^z}\chi\, \sqrt\gz \mu_{12}\,d\vu\,dz = \int_\Cd \(\chi \, \sqrt\gz \,dz\)\wedge \mu.$$
   Since the restriction of $\chi \ \sqrt\gz \,dz$ to $\partial \Cd$ vanishes,  it belongs to $C_T^{0,1}(\Cd)$,  and thus
   $$\left|\int_\Cd \(\chi \,\sqrt\gz \,dz\)\wedge \mu - 2\pi \int_\ga \chi \, \sqrt\gz \,dz\right|\le C\eta.$$
\end{proof}

\begin{figure}
   \centering
   \includegraphics[scale=0.6]{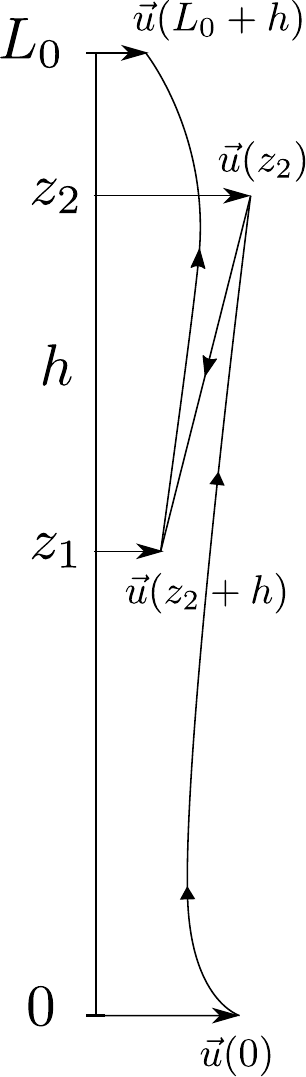}
   \caption{Piecewise graph: $z'(t) = 1$ on $[0,z_2]$ and $[z_2+h,\ellzero+h]$, $z'(t) = -1$ on $[z_2,z_2+h]$.}
   	\label{pwgraph}
\end{figure}

We will need  to expand up to second order the ratio of a curve which is more general than a graph over $\ga_0$, what we call a {\em piecewise graph}.
Such a piecewise graph is defined as a curve $\ga:[0,L]\to \Cd$, where $\ga(t) = \ga_0(z(t))+\vu(t)$, where $z(t):[0,L]\to[0,\ellzero]$ is a sawtooth function --- i.e. its derivative is piecewise constant  with values $\pm1$ --- and $\vu(t)$ is horizontal for any $t\in[0,L]$. An example is shown in Figure~\ref{pwgraph}.

\begin{proposition}\label{lexpansion}
   For any piecewise graph $\ga:[0,L]\to \cyl_{\delta/2}$, where $\ga(t) = \ga_0(z(t))+\vu(t)$ with  $z(0) = 0$ and $z(L) = \ellzero$, we have
   \begin{equation*}
   	\prem{\ga} = \ellzero+\pperp{\ga}+\prll{\ga}+ O\(\|\vu\|_{L^3([0,L])}^3\).\end{equation*}
\end{proposition}

\begin{proof}
   Since $\|\vu\|_{L^\infty} < \delta/2$, we have $\chi(\ga(t)) = 1$ for any $t$. Moreover, since  $\ga' (t)= z'(t)\e_3+\vu'(t)$, we have $dz(\ga'(t)) = z'(t)$. Thus
   $$\int_\ga \chi \ \sqrt\gz \,dz = \int_0^L z'(t) \sqrt{\gz\(\ga_0(z(t)) +\vu(t)\)}\,dt.$$
   We Taylor expand to find that, for every $t$,
   $$\gz\(\ga_0(z) +\vu\) =  1+(d\gz)_z^\bullet(\vu)+\frac{1}2(d^2\gz)_z^\bullet(\vu,\vu) + O\(|\vu|^3\),$$
   where $z = z(t)$. We deduce that
   \begin{equation*}
      \sqrt{\gz\(\ga_0(z(t)) +\vu(t)\)} = 1+ \frac12 (d\gz)_z^\bullet(\vu)+\(\frac14(d^2\gz)^\bullet_z(\vu,\vu)  - \frac18 (d\gz)_z^\bullet(\vu)^2\) + O\(|\vu|^3\).\end{equation*}
   Integrating with respect to $t$ yields, since $\chi(\ga) \equiv 1$ and $dz(\ga') = z'$, that
   $$\prem{\ga} = \int_\ga \,dz+\pperp{\ga}+\prll{\ga}+ O\(\|\vu\|_{L^3}^3\).$$
   To conclude, we note  that since $z(0) = 0$ and $z(L) = \ellzero$, we have
   $$\int_\ga \,dz = \int_0^L z'(t)\,dt = \ellzero.$$
\end{proof}

\begin{proposition}\label{bexpansion}
   Assume $\ga:[0,L]\to \cyl_{\delta}$, where $\ga(t) = \ga_0(z(t))+\vu(t)$ with $z(0) = 0$ and $z(L) = \ellzero$.
   Then,
   $$\pr{B_0,\ga} = \pr{B_0,\ga_0} +  \pperp{B_0,\ga} +\prll{B_0,\ga} + O\(\|\vu\|_{L^3}^3+\|\vu\|_{L^\infty}^2\|\vu'\|_{L^1} \),$$
   where
   \begin{align}
   	\pperp{B_0,\ga}& \colonequals \int_0^L z'(t)  (dB_0)^\bullet_{z(t)}(\vu,\e_3) \,dt,\notag\\
   	\prll{B_0,\ga}& \colonequals \int_0^L z'(t)\Lb(z(t),\vu(t),z'(t)\vu'(t))\,dt,\notag \\
   	\label{LB}\Lb(z,\vu,\vv)&\colonequals\frac12 \( (dB_0)^\bullet_{z}(\vu,\vv)+(\partial_{\vu}dB_0)^\bullet_{z}(\vu,\e_3) \).
  \end{align}	
\end{proposition}

\begin{remark} Note that
   \begin{equation}\label{suplagb} \Lb(z,\vu,z'\vu')\le C\(|\vu|^2 +|\vu||\vu'|\).\end{equation}
\end{remark}

\begin{proof}
   First note that, as currents, the curves $t\to\ga_0(t)$ and $t\to\ga_0(z(t))$ are equal since for any smooth vector field $X$ we have
   $$\int_0^L z'(t) \ga_0'(z(t))\cdot X(\ga_0(z(t))\,dt = \int_0^{\ellzero} \ga_0'(s)\cdot X(\ga_0(s))\,ds.$$

   We consider the surface $\Sigma$ parametrized by $(s,t)\to \ga_0(z(t))+ s\vu(t)$, for $s\in(0,1)$ and $t\in(0,L)$. Then the boundary of $\Sigma$, oriented by the $2$-form $ds\wedge dt$ is equal to $\ga - \ga_0$, plus two horizontal segments on the top and bottom boundaries of $\Cd$. On these boundaries  $B_0$ vanishes for horizontal vectors therefore, from Stokes's formula,
   \begin{equation}\label{stokesj}\int_{\ga} B_0 - \int_{\ga_0} B_0 = \int_{\Sigma} dB_0.\end{equation}
   We have
   $$\int_{\Sigma} dB_0 = \int_0^1\int_{0}^{L} (dB_0)_{\ga_0(z(t))+s\vu(t)}(\vu(t),z'(t)\e_3+s\vu'(t))\,dt\,ds.$$
   Then we expand
   $$(dB_0)_{\ga_0(z(t))+s\vu(t)} = (dB_0)^\bullet_{z(t)}+ s(\partial_{\vu(t)}dB_0)^\bullet_{z(t)} + O(|\vu(t)|^2 ),$$
   to obtain, omitting the variable $t$ for clarity,
   \begin{multline*}(dB_0)_{\ga_0(z)+s\vu}(\vu,z'\e_3+s\vu') = (dB_0)^\bullet_z(\vu,z'\e_3) \\+ s(dB_0)^\bullet_z(\vu,\vu') +s(\partial_{\vu}dB_0)^\bullet_z(\vu,z'\e_3)+ O(|\vu|^3+|\vu'||\vu|^2),
   \end{multline*}
   and then
   \begin{multline*}\int_{\Sigma} dB_0 = \int_{0}^{L} \( (dB_0)^\bullet_z(\vu,z'\e_3) +\frac12(dB_0)^\bullet_z(\vu,\vu')+\frac12(\partial_{\vu}dB_0)^\bullet_z (\vu,z'\e_3) \)\,dt \\ + O\(|\vu|^3+|\vu'||\vu|^2)\).\end{multline*}
   Then we replace in \eqref{stokesj} to deduce that
   $$\pr{B_0,\ga} = \pr{B_0,\ga_0} +  \pperp{B_0,\ga} + \prll{B_0,\ga} + O(\|\vu\|_{L^3}^3+\|\vu'\|_{L^1}\|\vu\|_{L^\infty}^2).$$
\end{proof}

\begin{proposition} \label{criticality}
   Assume the maximum of $\ga\mapsto \R(\ga)$ is achieved at $\ga_0(z) = (0,0,z)$ (in tube coordinates). Then for any $z\in(0,\ellzero)$ and for any horizontal vector $\vu$  we have
   \begin{equation}\label{critical}(d B_0)^\bullet_z(\vu, \e_3)=  \frac{\R(\ga_0)}2  (d\gz)^\bullet_z(\vu).\end{equation}
   In particular, for any piecewise graph $\ga$ in $\cyl_{\delta}$ we have
   \begin{equation}\label{critperp} \pperp{\ga}  = \frac{\pperp{B_0,\ga}}{\R(\ga_0)}.\end{equation}
\end{proposition}

\begin{proof}
   Consider a smooth variation of $\ga_0$ written as $\ga_t(z) = \ga_0(z) +t\vu(z)$, where $\vu(z)$ is horizontal for any $z$. We have $\ga_t'(z) = \e_3+t\vu'(z)$ and
   $$|\ga_t'(z)|_g = \sqrt{\gz(\ga_t(z)) + t^2|\vu'(z)|_g^2},$$
   therefore
   $$\frac d{dt}_{|t=0} |\ga_t'(z)|_g = \frac12 (d\gz)_{(0,0,z)}(\vu(z)),$$
   and
   $$\frac d{dt}_{|t=0} |\ga_t| = \frac12 \int_0^{\ellzero} (d\gz)_{(0,0,z)}(\vu(z))\,dz.$$

   To compute the derivative of $\pr{B_0,\ga_t}$ with respect to $t$, we let $\Sigma_t$ be the surface parametrized by $(s,z)\to (0,0,z)+ s\vu(z)$, for $s\in(0,t)$ and $z\in(0,\ellzero)$. Then the boundary of $\Sigma_t$, oriented by the $2$-form $ds\wedge dz$ is equal to $\ga_t - \ga_0$, plus two horizontal segments on the top and bottom boundaries of $\Cd$. On these boundaries  $B_0$ vanishes for horizontal vectors therefore, from Stoke's formula,
   $$\int_{\ga_t} B_0 - \int_{\ga_0} B_0 = \int_{\Sigma_t} dB_0,$$
   and
   \begin{align*} \int_{\Sigma_t} dB_0 & = \int_0^t\int_0^{\ellzero} (dB_0)_{\ga_0(z)+s\vu(z)}(\vu(z),\e_3+s\vu'(z))\,dz\,ds \\
                                    & = t\int_0^{\ellzero} (d B_0)_{(0,0,z)}(\vu(z), \e_3)\,dz +O\(t^2|\vu|_\infty^2\).
   \end{align*}

   If $\ga_0$ is critical for the ratio, then
   $$ \ellzero \frac d{dt}_{|t=0}  \pr{B_0,\ga_t} = \pr{B_0,\ga_0} \frac d{dt}_{|t=0} |\ga_t|, $$
   therefore
   $$ \ellzero \int_0^{\ellzero} (d B_0)_{(0,0,z)}(\vu(z), \e_3)\,dz =  \frac12 \pr{B_0,\ga_0}\int_0^{\ellzero} (d\gz)_{(0,0,z)}(\vu(z))\,dz. $$
   Since this is true for any variation $\ga_t(z) = \ga_0(z) +t\vu(z)$, we deduce \eqref{critical}, and \eqref{critperp} follows in view of the definition of $\pperp{\ga}$ and $\pperp{B_0,\ga}$.
\end{proof}

The above computations imply the following.
\begin{lemma}\label{lemformQ}
   Assume the maximum of $\ga\mapsto \R(\ga)$ is achieved at $\ga_0(z) = (0,0,z)$ and that  $\ga$ is a  graph over $\ga_0$ included in $\Cd$ and parametrized as $\ga(z) = \ga_0(z)+\vu(z)$. We define $\ga_t(z) = \ga_0(z)+t\vu(z)$. Then
   $$Q(\vu) \colonequals - \frac{d^2}{dt^2}_{|t=0}\R(\ga_t) =\frac{2\R(\ga_0)}{\ellzero}\int_0^{\ellzero}\frac12|\vu'|^2_{g^\bullet}+  \Lag(z,\vu(z),\vu'(z))\,dz,$$
   where
   $$\Lag(z,\vu,\vu') = \Ll(z,\vu)-\frac1{\R(\ga_0)}\Lb(z,\vu,\vu'),$$
   with $\Lb$ as in \eqref{LB} and $\Ll$ as in \eqref{defLl}.
\end{lemma}
\begin{proof}
   Let $f(t) = \R(\ga_t)$. Then $f(t) = g(t)/h(t)$, with $g(t) = \pr{B_0,\ga_t}$ and $h(t) = |\ga_t|$. Then, since $f'(0) = 0$, we have
   \begin{equation}\label{derquotient}f''(0) = \frac{g''(0)h(0) - g(0)h''(0)}{h(0)^2}.\end{equation}

   From Lemmas~\ref{lexpansion} and \ref{bexpansion} we have
   $$\frac{d^2}{dt^2}_{|t=0}\prem{\ga_t} = 2\prll{\ga}, \quad \frac{d^2}{dt^2}_{|t=0}\pr{B_0,\ga_t} = 2\prll{B_0,\ga}.$$
   Moreover,
   $$|\ga_t| - \prem{\ga_t} = \int_0^{\ellzero}\sqrt{\gz(\ga_t)+t^2|\vu'|^2_{g(\ga_t)}} - \sqrt{\gz(\ga_t)}\,dz = \frac{t^2}2\int_0^{\ellzero} |\vu'|^2_{g^\bullet}\,dz + O(t^3).$$
   It follows that
   \begin{equation}\label{derpr} \frac{d^2}{dt^2}_{|t=0}\(|\ga_t| - \prem{\ga_t} \) = \int_0^{\ellzero} |\vu'|^2_{g^\bullet}\,dz,\end{equation}
   and then, in view of \eqref{derquotient} and \eqref{derpr}, that
   $$f''(0) =\frac2{\ellzero} \int_0^{\ellzero} \Lb(z,\vu(z),\vu'(z)) - \R(\ga_0)\( \frac12|\vu'|^2_{g^\bullet}+\Ll(z,\vu(z))\)\,dz.$$
\end{proof}

The above computation motivates Definition \ref{qdeu}.

\subsection{Coercivity}

The quantity $\qell(\ga)$ defined below contains the interesting terms in the lower bound for $GL_\ep(\u,\A)$ which we can compute using the idea of horizontal blow-up introduced in \cite{ConJer}. Here $\ell$ represents the scale of the horizontal blow-up and tends to zero with $\ep$, and $\ga$ represents the vortex filaments associated to $(\u,\A)$. In this section we show that the quadratic form $Q$ defined above is essentially the limit of $\qell$ as $\ell\to 0$.

\begin{definition}\label{defqell}
   Assume that $\ga$ is a piecewise graph over $\ga_0$. For $\ell>0$, define
   $$\qell(\ga) \colonequals \ell^2\(|\ga|_\gell - \prem{\ga}\) + (\prem{\ga} - \ellzero)- \frac{\pr{B_0,\ga - \ga_0}}{\R(\ga_0)}, $$
   where $|\ga|_\gell$ is the length of $\ga$ with respect to the metric $\gell = \gz dz^2 + \ell^{-2} \gperp$.
\end{definition}

\begin{lemma}
   If $\ga$ is a piecewise graph in $\cyl_{\delta/2}$,  parametrized as $\ga(t) = \ga_0(z(t)) +\vu(t)$, $t\in[0,L]$, with $z(0) = 0$ and $z(L) = \ellzero$, and if the maximum of $\ga\mapsto \R(\ga)$ is achieved at $\ga_0(z) = (0,0,z)$, then
   \begin{multline}\label{qell} \qell(\ga) = \int_0^L \left\{\ell^2\(\sqrt{\gz(\ga) +\ell^{-2}|\vu'|_{g(\ga)}^2} - z'\sqrt{\gz(\ga)}\)\right.\\
      \left. + z'\( \Ll(z,\vu)-\frac{\Lb(z,\vu,z'\vu')}{\R(\ga_0)}\)\right\}\,dt + O\(\|\vu\|_{L^\infty}^3+\|\vu\|_{L^\infty}^2 \|\vu'\|_{L^1}\).
   \end{multline}
\end{lemma}

\begin{proof}
   The result follows from the definition of $\qell$, $\prem{\ga}$ and Propositions~\ref{lexpansion}, ~\ref{bexpansion}, taking into account the fact that,  because of the criticality of $\ga_0$,  we have from Proposition~\ref{criticality} that $\R(\ga_0)\pperp{\ga} = \pperp{B_0,\ga}$.
\end{proof}

\begin{proposition}\label{coercive}
   Assume that $\ga_0$ is a nondegenerate maximizer of the ratio $\R$.
   \begin{enumerate}
      \item There exists a small constant $c>0$ such that for any $\ell\le 1$ and for any piecewise graph $\ga$ parametrized as $\ga(t) = \ga_0(z(t)) +\vu(t)$ with $\|\vu\|_{L^\infty} < c \ell$  it holds that 
            $$\qell(\ga)\ge c \|\vu\|_{L^\infty}^2,$$ 
            so that in particular $\qell(\ga)\ge 0$.

      \item Assume $\{\ell_n\}_n$, $\{\alpha_n\}_n$ and the piecewise graphs $\{\ga_n\}_n$ --- parametrized as $\ga_n(t) = \ga_0(z_n(t)) +\vun(t)$ for $t\in[0,L_n]$ --- are such that, as $n\to +\infty$, 
      $$0<\alpha_n\ll \ell_n\le 1,\quad \qelln(\ga_n) = O({\alpha_n}^2),\quad \|\vun\|_{L^\infty} =o(\ell_n).$$
      Then, defining $\vec v_n = \vun/\alpha_n$, there exists a subsequence  $\{n\}$ such that $L_n\to \ellzero$ and such that $\{v_n\}_n$ converges in the sense of distributions to some $\vu_*$ which belongs to $H^1((0,\ellzero),\RR^2)$. Moreover, as $n\to +\infty$, 
      \begin{equation}\label{cvcoercive}\max_t |\vec v_n(t) - \vu_*(z_n(t))|\to 0,\quad \frac{\|\ga_n -\ga_{*,n}\|_*}{\alpha_n}\to 0,
      \end{equation}  
      where $\ga_{*,n}(z) = \ga_0(z) + \alpha_n\vu_*(z)$ is a graph over $\ga_0$.
      
      Finally,
      \begin{equation}\label{limq}
               \liminf_{n\to +\infty} \frac{\qelln(\ga_n)}{{\alpha_n}^2} \ge \frac{\ellzero}{2\R(\ga_0)} Q(\vu_*).
            \end{equation}
   \end{enumerate}

\end{proposition}
\begin{remark}
 If the metric $g$ was Euclidean, we could state the result in terms of the horizontal blow-up of the curves $\ga_n$ converging to a curve $\ga_*$. Here, it is really the metric which is blown-up horizontally.
\end{remark}
\begin{remark}
   The first assertion in \eqref{cvcoercive} means that the distance of any point of $\ga_n$ with vertical coordinate $z$ to $\ga_{*,n}(z)$ is $o(\alpha_n)$, uniformly with respect to the point chosen and to $z$. Note that $\ga_n$ need not be a graph, it may even have a diverging number of points of a given height $z$.
   \end{remark}

Note that the first statement follows from the second one, since, by reasoning by absurdity, if we choose $\alpha_n = \|\vun\|_{L^\infty}$, then $\|\vu_*\|_{L^\infty} = 1$ and therefore $Q(\vu_*)$ is bounded below by a constant depending only on $Q$. Thus we only need to prove the second item of the proposition.

\begin{proof}
   Assume $\{\ell_n\}_n$, $\{\alpha_n\}_n$, and  $\{\ga_n\}_n$ are as above, and that $\gan$ is parametrized as $\gan(t) = \ga_0(z_n(t))+\vun(t)$, for $t\in[0,L_n]$. We define $\bzn(t)$ to be the monotone increasing envelope of $t\to z_n(t)$, i.e. $\bzn(t) = \max_{s\le t}z_n(s)$. Note that $\bzn$ is continuous and piecewise affine, with  $\bzn'(t)$ equal to either $0$ or $1$, except at the finite number of points of discontinuity of $\bzn'$.

   Let $q_\eln(\gan)(t)$ denote the integrand in \eqref{qell}. We rewrite \eqref{qell} as
   \begin{equation}\label{qlngan}\qelln(\gan) = \int_{A_n} q_\eln(\gan)(t)dt + \int_{A_n^c} q_\eln(\gan)(t)dt + O\(\|\vu\|_{L^\infty}^3+ \|\vu\|_{L^\infty}^2 \|\vu'\|_{L^1}\),\end{equation}
   where $A_n \colonequals \{t\in[0,L_n]\mid\bzn = z_n\}.$ In particular we have ${z_n}' = 1$, almost everywhere in $A_n$.

   The second integral in \eqref{qlngan} may be bounded below by noticing that the integral of $z_n'$ over $A_n^c$ is equal to $0$. Since $\gz(\ga_0(z)) = 1$ for any $z$, we have
   $$\int_{A_n^c} z_n' \sqrt{\gz(\gan(t))}\,dt = \int_{A_n^c} z_n' \(\sqrt{\gz(\gan(t))} - \sqrt{\gz(\ga_0(z_n(t)))}\) \,dt,$$
   and since $|\sqrt{\gz(\gan(t))} - \sqrt{\gz(\ga_0(z_n(t)))}|\le C\|\vun\|_{L^\infty}$ we find that
   $$\left|\int_{A_n^c} z_n' \sqrt{\gz(\gan(t))}\,dt\right|\le C \left|A_n^c\right|\|\vun\|_{L^\infty}.$$
   Using the bounds \eqref{suplagb} and \eqref{suplagl} we deduce that
   \begin{equation*}\int_{A_n^c} q_\eln(\gan)(t)dt \ge   \int_{A_n^c}\eln^2\sqrt{\gz(\gan) +\eln^{-2}|\vun'|_{g(\gan)}^2}
      - C \|\vun\|_{L^\infty} \int_{A_n^c} \eln^2+|\vun|+ |\vun'|\,dt,\end{equation*}
   and therefore
   $$\int_{A_n^c} q_\eln(\gan)(t)dt \ge  \hal\int_{A_n^c} \eln^2 + \eln|\vun'|_{g(\gan)}\,dt - C \|\vun\|_{L^\infty} \int_{A_n^c} \eln^2+|\vun|+ |\vun'|\,dt.$$
   Since $\|\vun\|_{L^\infty}/\eln\to 0$ as $n\to +\infty$, it follows that if $n$ is large enough, then
   \begin{equation}\label{ldots2}\int_{A_n^c} q_\eln(\gan)(t)dt \ge \frac14 \int_{A_n^c} \eln^2 + \eln|\vun'|_{g(\gan)}\,dt .\end{equation}

   Next, we analyze further the first integral in \eqref{qlngan} by splitting it as the integral on a set  $B_n$ and on its complement: We choose $\delta_n>0$ such that, as $n\to +\infty$, 
   \begin{equation}\max(\alpha_n,\|\vun\|_\infty)\ll\delta_n\ll\eln,\label{deln}\end{equation}
   and then we set
\begin{equation*}
	B_n = \{t\in A_n\mid |\vun'(t)|_{g(\gan)} \le \delta_n\}.\end{equation*}
   If $t\in B_n$, we have
   $$\frac{|\vun'|_{g(\gan)}^2}{\eln^2} \le \frac{\delta_n}\eln,$$
   and the right-hand side goes to $0$ as $n\to +\infty$. Therefore a Taylor expansion yields, as $n\to +\infty$,
   $$\eln^2\int_{B_n} \sqrt{\gz(\gan) +\eln^{-2}|\vun'|_{g(\gan)}^2} - \sqrt{\gz(\gan)}\,dz = \(\hal - o(1)\) \int_{B_n} |\vun'|_{g(\ga_0)}^2\,dz,$$
   where we have also used the fact that $\gz(\gan)\to 1$ and $g(\gan)\to g(\ga_0)$ uniformly as $n\to +\infty$, since $\|\vun\|_{L^\infty}\to 0$. We deduce that, writing $g^\bullet = g(\ga_0)$,
   \begin{equation*}
   	\int_{B_n} q_\eln(\gan)(t)dt \ge \(\hal - o(1)\) \int_{B_n} |\vun'|_{g^\bullet}^2\,dt+  \Ll(z_n(t),\vun)-\frac{\Lb(z_n(t),\vun,\vun')}{\R(\ga_0)}\,dt.\end{equation*}
   Using \eqref{suplagl} and \eqref{suplagb}, we deduce in particular that
   \begin{equation}\label{lowbn}\int_{B_n} q_\eln(\gan)(t)dt \ge -C \|\vun\|_{L^\infty}^2.\end{equation}

   On the other hand, since $(\sqrt{a+x} - \sqrt{a})/\sqrt{x}$ is increasing, we have for any $t$ in $A_n\setminus B_n$ that
   $$\frac{\sqrt{\gz(\gan) +\frac{|\vun'|_{g(\gan)}^2}{\eln^{2}}} - \sqrt{\gz(\gan)}}{|\vun'|_{g(\gan)}/\eln}\ge \frac{\sqrt{\gz(\gan) +\frac{\delta_n^2}{\eln^{2}}} - \sqrt{\gz(\gan)}}{\delta_n/\eln}. $$
   Since $\delta_n\ll \eln$ and since  $\gz(\gan)\to 1$ and $g(\gan)\to g^\bullet$, we deduce
   $$\sqrt{\gz(\gan) +\frac{\delta_n^2}{\eln^{2}}} - \sqrt{\gz(\gan)}\ge \(\hal - o(1)\) \frac{\delta_n^2}{\eln^{2}}, $$
   and therefore 
   $$\eln^2\(\sqrt{\gz(\gan) +\frac{|\vun'|_{g(\gan)}^2}{\eln^{2}}} - \sqrt{\gz(\gan)}\) \ge \(\hal - o(1)\) |\vun'|_{g(\gan)}\delta_n.$$

   From \eqref{suplagb} and \eqref{suplagl} we know that $\Ll(z,\vun)$ and $\Lb(z,\vun,\vun')$ are both negligible compared to $|\vun'|\delta_n$ if $t\in A_n\setminus B_n$, therefore we deduce that
   \begin{equation}\label{ldots12} \int_{A_n\setminus B_n}  q_\eln(\gan)(t)dt \ge  \(\hal - o(1)\) \delta_n \int_{A_n\setminus B_n}|\vun'|_{g^\bullet}\ge \(\hal - o(1)\) \delta_n^2 |A_n\setminus B_n|.\end{equation}

   From the hypothesis $\qelln(\ga_n) = O({\alpha_n}^2)$ and the estimates  \eqref{qlngan}, \eqref{ldots2}, \eqref{lowbn} and \eqref{ldots12}, we find that 
   $$\eln\int_{A_n^c} \eln+ |\vun'|_{g(\gan)}\,dt + \delta_n \int_{A_n\setminus B_n} \delta_n + |\vun'|_{g(\gan)}\,dt\le C\(\alpha_n^2+\|\vun\|_{L^\infty}^2(1+\|\vun'\|_{L^1})\).$$
   Then we can decompose $\|\vun'\|_{L^1} = \|\vun'\|_{L^1(B_n^c)} + \|\vun'\|_{L^1(B_n)}$ , use the fact that $|\vun'|\le \delta_n$ on $B_n$ and use \eqref{deln} to deduce that
   \begin{equation}\label{onbadset} \|\vun'\|_{L^1(B_n^c)} = o\(\alpha_n + \|\vun\|_{L^\infty}\),\quad |B_n^c| = o(1).\end{equation}
   
   We now define $\wun:[0,\ellzero]\to \RR^2$ as follows. First we require that $\wun(0) = \vun(0)$. Then we note that, except for a finite set of values of $z$ --- which we denote $J$ --- there exists a unique $t$ such that $\bzn(t) = z$ and therefore a unique $t\in A_n$ such that $z_n(t) = z$. We then require that for any $t\in A_n$,
   \begin{equation*}
      \wun'(z_n(t)) =
      \begin{cases}
         \vun'(t) & \text{if $t\in B_n$} \\
         0        & \text{if $t\in A_n\setminus B_n$.}
      \end{cases}
   \end{equation*}
   This defines unambiguously $\wun'(z)$ if $z\notin J$, thus $\wun$ is well defined. Moreover, using \eqref{onbadset}, we have for any $t\in A_n$
   \begin{equation}\label{wuinfty}|\wun(z_n(t))-\vun(t)|\le \int_{B_n^c} |\vun'(s)|\,ds = o(\alpha_n+\|\vun\|_{L^\infty}).\end{equation}
   Since $\Ll(z,\cdot)$ is a quadratic form, we thus have that for any $t\in A_n$ 
   $$ \Ll(z_n(t),\wun(z_n(t))) - \Ll(z_n(t),\vun(t))\,dt = o(\alpha_n^2+\|\vun\|_{L^\infty}^2).$$
   Since $\Lb(z,\cdot,\cdot)$ is bilinear, we also deduce that 
   $\Lb(z_n(t),\wun(z_n(t)),\wun'(z_n(t)))  = 0 $ for  $t\in A_n\setminus B_n$ while for $t\in B_n$ we have 
   $$\Lb(z_n(t),\wun(z_n(t)),\wun'(z_n(t)))   - \Lb(z_n(t),\vun(t),\vun'(t))= o((\alpha_n+\|\vun\|_{L^\infty})|\wun'|).$$
   This allows us to write 
   $$\int_{B_n}\Ll(z,\vun)  =  \int_0^\ellzero\Ll(z,\wun)\,dz + o({\alpha_n}^2+\|\vun\|_{L^\infty}^2),$$
   $$\int_{B_n}\Lb(z,\vun,\vun') - \int_0^{\ellzero}\Lb(z,\wun,\wun')\,dz = o({\alpha_n}+\|\vun\|_{L^\infty})\int_0^{\ellzero}|\wun'|.$$
   Note that we have use the change of  variables $z = z_n(t)$, to pass from integrals over $A_n$ to integrals over $[0,\ellzero]$.

   We now use again the hypothesis $\qelln(\ga_n) = O({\alpha_n}^2)$ with the information we have gathered so far to obtain that 
   \begin{equation}
\begin{split} 
   C\alpha_n^2 &\ge \int_{B_n} q_\eln(\gan)(t)dt + O\(\|\vun\|_{\infty}^3+\|\vun\|_{\infty}^2\|\vun'\|_{L^1}\) \\
   &=  \int_0^{\ellzero} (\frac12 - o(1))|\wun'|_{g^\bullet}^2 + \Lag(z,\wun,\wun')\,dz+ o({\alpha_n}^2+\|\vun\|_{\infty}^2) + o({\alpha_n}+\|\vun\|_{\infty})\int_0^{\ellzero}|\wun'|.
 \end{split}\label{beforeQ}
\end{equation}
 
   Then, from the nondegeneracy of the maximizer $\ga_0$ as defined in Definition~\ref{qdeu}, that is, the positive definiteness of $Q$, we deduce that 
   $$\int_0^{\ellzero} |\wun'|_{g^\bullet}^2 + \Lag(z,\wun,\wun')\,dz \ge c_0 \|\wun\|_{H^1}^2,$$
   for some $c_0>0$ independent of $n$. Using also the fact the $H^1$ norm controls the $L^\infty$ norm in one dimension, we see that the error terms in \eqref{beforeQ} may be absorbed by the left-hand side and the first term on the right-hand side to 
   deduce that $\wun/\alpha_n$ subsequentially converges weakly in $H^1([0,\ellzero])$ to some $\vu_*$, and that 
   $$\liminf_n \frac{\qell(\ga_n)}{{\alpha_n}^2} \ge \frac{\ellzero}{2\R(\ga_0)} Q(\vu_*).$$
   The weak $H^1$ convergence also implies that $\wun/\alpha_n$ converges to $\vu_*$ uniformly if we take a further subsequence.   It then follows from \eqref{wuinfty} that $|\vec{v_n}(t) - \vu_*(z_n(t))|$ converges to $0$ as $n\to +\infty$, uniformly w.r.t. $t$, hence proving the first assertion in \eqref{cvcoercive}. 

   We also deduce from \eqref{onbadset} that $L_n\to\ellzero$, since ${z_n}'=1$ on $A_n$ and the integral of ${z_n}'$ on $[0,L_n]$ is equal to $\ellzero$. It also follows from \eqref{onbadset} that $z_n(t)\to t$ uniformly. Together with the uniform convergence of $\vec{v_n}(t)- \vu_*(z_n(t))$ to $0$, this implies that $v_n\to \vu_*$ in the sense of distributions. Note that the fact that  $z_n(t)\to t$ tells us that the limit of the piecewise graphs $\ga_n$ is in fact a graph.

   It remains to prove that $\|\ga_n-\ga_{*,n}\|_* = o(\alpha_n)$. This is a direct consequence of the above and Lemma~\ref{flatstarlem} below, noting that the length of $\ga_n$ is bounded as a consequence of \eqref{onbadset}, and the fact that $\wun'$ is bounded in $L^2$, hence in $L^1$ also. 
\end{proof}

\begin{lemma}\label{flatstarlem} Let $\ga$, $\tga\subset T_\delta\cap\overline\Omega$
   be two curves such that, in tube coordinates, $\ga(t) = \ga_0(z(t))+\vu(t)$ is a piecewise graph defined over $[0,L]$ with $z(0) = 0$ and $z(L) = \ellzero$ and  $\tga(z) = \ga_0(z)+\vv(z)$ is a graph over $\ga_0$. Then
   $$\|\ga-\tga\|_{\F(\Omega)},\ \|\ga-\tga\|_* \le C \(\max_t|\vu(t) - \vv(z(t))|\) \(|\ga|+|\tga|\),$$
   where all norms are understood with respect to the metric $g$. 
\end{lemma}
\begin{proof}
   Let $X$ be any vector field defined in $\overline\Omega$, normal to $\partial\Omega$, and such that $\|X\|_\infty$ and $\|\curl X\|_\infty$ are less than or equal to $1$. We need to show that
   $$\left|\pr{\ga,X} - \pr{\tga,X}\right|\le C \(\max_t|\vu(t) - \vv(z(t))|\) \(|\ga|+|\tga|\).$$
   For this we define $\Sigma(s,t) = (1-s)\ga(t) + s\tga(z(t))$ for $(s,t)\in[0,1]\times[0,L]$.

   Then, applying Stoke's Theorem and the fact that $X$ is normal to $\partial\Omega$, we have
   $$ \pr{\ga,X} - \pr{\tga,X} = \int_{\Sigma} \curl X\cdot\nu_\Sigma\le |\Sigma|,$$
   and therefore 
   $$ \left|\pr{\ga,X} - \pr{\tga,X}\right| \le \int_0^L\int_0^1|\partial_s\Sigma||\partial_t\Sigma|\,ds\,dt.$$
   Since $\partial_s\Sigma = \tga(z(t)) - \ga(t)$ and $|\partial_t\Sigma(s,t)|\le |\ga'(t)|+|\tga'(t)|,$    the result follows.
\end{proof}

\subsection{Strong nondegeneracy implies  weak nondegeneracy}\label{secstrongweak}

\begin{lemma} \label{curveintube}
   There exists $\eta>0$ depending on $\Omega$ such that the following holds.

   Assume   $\ga$ is Lipschitz curve with no boundary in $\Omega$ such that
   \begin{equation}\label{flatclose}\pr{\chi_\eta\frac{\partial_z}{\sqrt\gz},\frac{\ga_0}{\ellzero}-\frac{\ga}{|\ga|}}\le \ell,\end{equation}
   where $\chi_\eta$ is a cut-off function for the cylinder $\cyl_\eta = B(0,\eta)\times (0,\ellzero)$, that is, $\chi_\eta$ is  equal to $1$ in $\cyl_{\eta/2}$, equal to $0$ outside $\cyl_\eta$, takes values in $[0,1]$ and has gradient bounded by $C/\eta$.

   Then, if $\ell$ is small enough depending on $\Omega$, $\ga$ is included in a neighborhood of $\ga_0$ where tubular coordinates are defined, so that we may write $\ga(t) = \ga_0(z(t))+\vu(t)$ for a horizontal vector $\vu$. Then, parametrizing $\ga$ and $\ga_0$ by arclength with a proper orientation, we have
   \begin{equation}\label{flatclosebound}\int_0^L |1-z'(t)|dt\le C\sqrt\ell,\quad \|\vu\|_{L^\infty}\le C\sqrt\ell,\quad \left||\ga|-\ellzero\right|\le C\sqrt \ell,\end{equation}
   where $C$ depends on $\Omega$. Moreover $z(0) = 0$ and $z(L) = \ellzero.$
\end{lemma}

\begin{proof}
   Fix $\eta>0$ small enough so that tubular coordinates are defined in $\cyl_\eta$ and such that $\sqrt\gz\in(1/2, 2)$ in $\cyl_\eta$. The domain in $\Omega$ corresponding to $\cyl_\eta$ is denoted $T_\eta$. We assume $\ga$ is parametrized by arclength and whenever $\ga(t)\in T_\eta$, we denote by $(x(t),y(t),z(t))$ the coordinates of $\ga(t)$. We let $X = \chi_\eta\frac{\partial_z}{\sqrt\gz}$.

   Since $\pr{X,\ga_0/\ellzero}=1$, denoting $L = |\ga|$ we have by definition and using the fact that $|\ga'|_g=1$ from the arclength parametrization hypothesis,  that
   $$\pr{X,\frac{\ga}{|\ga|}} = \dashint_0^{L} \pr{X,\ga'}_g, \quad \pr{X,\ga'}_g \le 1.$$
   It follows that
   $$ 0 \le \pr{X,\frac{\ga_0}{\ellzero}-\frac{\ga}{|\ga|}} = \dashint_0^L (1-|X|_g) + \(|X|_g|\ga'|_g -  \pr{X,\ga'}_g\)\le C\ell,$$
   Both terms in the integral being nonnegative, they are each bounded by $C\ell$.

   When $\ga(t)\in \cyl_\eta$, we may write it as $\ga(t) = \ga_0(z(t)) + \vu(t)$, where $\vu(t)$ is horizontal. Then, replacing $|X|_g$ and $\pr{X,\ga'}_g$ by their value, we have
   \begin{equation}\label{sepell}\dashint_0^L (1 - \chi_\eta) \le C\ell,\quad\dashint_0^L \chi_\eta \(1  - \sqrt\gz z'\)\le  C\ell.\end{equation}

   We now show that $L=|\ga|$ is bounded by a constant depending only on $\Omega$. For this we begin by extending the coordinate function $z$ defined in $T_\eta$ to a $C^1$ function defined in $\Omega$, whose $C^1$ norm is then a constant depending on $\Omega$. Since $\sqrt\gz\in(1/2, 2)$ on $\cyl_\eta$, we have
   $$\dashint_0^L \chi_\eta  \(\frac12 - z'\) \le \dashint_0^L \chi_\eta  \(\frac1{\sqrt\gz} - z'\) \le 2 \dashint_0^L \chi_\eta\sqrt\gz\(\frac1{\sqrt\gz}-z'\) \le C\ell,$$
   while, using the first inequality in \eqref{sepell},
   $$\dashint_0^L (1-\chi_\eta) \(\frac12 - z'\) \le C \dashint_0^L (1-\chi_\eta)\le C\ell.$$
   The two inequalities imply
   $$\dashint_0^L \frac12 - z'\le C\ell.$$
   Since the integral of $z'$ is bounded above by $2 \max_{\Omega} |z|\le C$, we deduce that
   $$\frac12 - C\ell \le \frac CL,$$
   and then that $L$ is bounded above by a constant depending only on $\Omega$, if $\ell$ is small enough.

   Next, we claim that $\ga$ is included in the cylinder $\cyl_{C\sqrt\ell}$ for some $C>0$ depending only on $\Omega$. First, from \eqref{flatclose}, there exists $a\in(0,L)$ such that $\ga(a)\in \cyl_{C\ell}$. If $\ga$ exited from $\cyl_{C'\sqrt\ell}$,  then the first exit point $b$ (which we assume larger than $a$ w.l.o.g) is such that $\ga(t)\in \cyl_{\eta/2}$ for $t\in(a,b)$, so that $\chi_\eta(\ga) = 1$ in $(a,b)$, and such that
   \begin{equation}\label{absurd}\int_a^b|\vu'|_{\gperp} \ge C'\sqrt\ell - C\ell.\end{equation}
   On the other hand, since $|\ga'|_g^2 = |\vu'|_{\gperp}^2 + \gz {z'}^2 =1$ the latter inequality in \eqref{sepell} implies  that
   $$ \int_a^b \sqrt{ |\vu'|_{\gperp}^2 + \gz {z'}^2} - \sqrt\gz z'\le  C\ell,$$
   and then --- using the fact that $\sqrt{1+x}-1\ge c \min(\sqrt x,x)$ for any $x\ge 0$ and that $\ell\le\sqrt\ell$ since $\ell$ is assumed small --- that
   \begin{equation}\label{ab} \int_a^b|\vu'|_{\gperp}\le  C\sqrt \ell,\end{equation}
   which would contradict \eqref{absurd} if $C'$ is chosen large enough. Therefore $\ga$ is included in  $\cyl_{C\sqrt\ell}$ for some $C>0$ and therefore $\|\vu\|_{L^\infty} \le C\sqrt \ell$, which proves the second inequality in \eqref{flatclosebound}

   Now, since $\ga$ is included in $\cyl_{C\sqrt\ell}$, we have that $(x(t), y(t), z(t))$ are defined in $[0,L]$. Since $\ga(0)$ and $\ga(L)$ belong to $\partial\Omega$, we must have $z(0)$, $z(L)\in\{0,\ellzero\}$, and then, in view of the latter inequality in \eqref{sepell}, that $z(0) = 0$ and $z(L) = \ellzero$, if $\ell$ is chosen small enough.  Then, arguing as above we find that \eqref{ab} holds with $a=0$ and $b = L$, so that
   \begin{equation}\label{vusql}\|\vu'\|_{L^1([0,L])}\le C\sqrt \ell.\end{equation}
   Moreover, since $\gz = 1$ on $\ga_0$ from the definition of $g$ in Proposition~\ref{prop:diffeo}, we have $|1 - \sqrt\gz|\le C\sqrt\ell$ on $C_{C\sqrt\ell}$, so that \eqref{sepell} implies
   $$\int_0^L  |1  -  z'|\le C\sqrt\ell,$$
   which using $z(0) = 0$ and $z(L) = \ellzero$ implies that $|L-\ellzero|\le C\sqrt \ell$ and then, together with \eqref{vusql} implies that $\left||\ga|-\ellzero\right|\le C\sqrt \ell$.
\end{proof}

\begin{proposition}\label{weakstrong}
   Assume that the maximum of $\R$ among normal $1$-currents which are divergence free in $\Omega$ is uniquely achieved (modulo a multiplicative constant) at a simple smooth curve $\ga_0$ with endpoints on $\partial\Omega$ and  that the second derivative of $\R$ at $\ga_0$ is definite negative. Then there exists $\alpha>0$ such that for any curve $\ga$ we have
   $$\R(\ga_0) - \R(\ga) \ge \alpha \min(\|\ga-\ga_0\|_*^2,1).$$
\end{proposition}

\begin{proof}
   The statement is equivalent to proving that if $\{\gan\}_n$ is a maximizing sequence of curves with no boundary in $\Omega$, then
   \begin{equation*}
      \liminf_{n\to +\infty} \frac{\R(\ga_0) - \R(\gan)}{\|\gan-\ga_0\|_*^2} >0.
   \end{equation*}
   Consider such a sequence $\{\ga_n\}_n$. Then, using the compactness of currents, $\gan/|\gan|$ converges to $\ga_0/\ellzero$ in the flat metric, hence in $(C_T^{0,1}(\Omega,\RR^3))^*$. It follows using Lemma~\ref{curveintube} that  if $n$ is large enough, then $\gan$ lies in $\Cd$ and, using tubular coordinates, we may parametrize it as $\gan(t) = \ga_0(z_n(t)) + \vun(t)$ in $[0,L_n]$, where $z_n'(t)\in\{\pm 1\}$ and where $z_n(0)=0$, $z_n(L_n) = \ellzero$.  It is also the case that $|\gan|\to\ellzero$ and that $\|1-z_n'\|_{L^1([0,L_n])}\to 0$, as $n\to +\infty$.

   Then, applying  Proposition~\ref{coercive}  with $\ell = 1$, we  deduce that
   \begin{equation}\label{qellnondeg}
      \liminf_{n\to +\infty} \frac{Q_1(\gan)}{\|\vun\|_{L^\infty}^2}>0.
   \end{equation}
   But, since $\pr{B_0,\ga_0}/\R(\ga_0) = \ellzero$, we have
   \begin{equation*}\begin{split}
         \frac{|\gan|}{\R(\ga_0)}\(\R(\ga_0) - \R(\gan)\) &= |\gan| - \frac{\pr{B_0,\gan}}{\R(\ga_0)}\\
         &=  \(|\gan| - \prem{\gan}\) + \(\prem{\gan} - \ellzero\) + \ellzero - \frac{\pr{B_0,\gan}}{\R(\ga_0)}\\
         &=\( |\gan| - \prem{\gan}\) + \(\prem{\gan} - \ellzero\)  - \frac{\pr{B_0,\gan-\ga_0}}{\R(\ga_0)}\\
         &=  Q_1(\gan).
      \end{split}
   \end{equation*}
   It follows, in view of \eqref{qellnondeg} and since $|\gan|\to\ellzero$,   that $\R(\ga_0) > \R(\gan)+  c \|\vun\|_{L^\infty}^2$.

   To prove the proposition, it remains to note that $\|\ga_n-\ga_0\|_* < C \|\vun\|_{L^\infty}$. This is proved for instance by choosing some vector field $X$ in $\Omega$ such that $\|X\|_{L^\infty}$, $\|\nabla X\|_\infty \le 1$ and bounding $|\pr{X,\gan-\ga_0}|$ by $C\|\vun\|_{L^\infty}$:
   \begin{equation*}
      \begin{split}
         &\left|\int_0^{L_n} X(\ga_0(z_n(t))+\vun(t))\cdot {(\ga_0(z_n(t))+\vun(t))}'\,dt - \int_0^{L_n} X(\ga_0(z_n(t))\cdot {\ga_0(z_n(t))}'\,dt \right| \\
         &\le \|\vun\|_{L^\infty} \|\nabla X\|_{L^\infty}|\ga_n| + \left|\int_0^{L_n} X(\ga_0(z_n(t)))\cdot {\vun'(t)}\,dt\right|\\
         &\le C \|\vun\|_{L^\infty} + \left|\int_0^{L_n} {X(\ga_0(z_n(t))}'\cdot \vun(t)\,dt\right|\\
         &\le C \|\vun\|_{L^\infty}.
      \end{split}\end{equation*}
\end{proof}

\subsection{Strong nondegeneracy for small balls}\label{sec:nondeg}
In this section we show that the strong nondegeneracy condition is satisfied in the case of a ball with small enough radius $\rho.$ We recall that it was proved in \cite{Rom2} (see also \cite{AlaBroMon}) that the diameter $\ga_0 = \{(0,0,z)\in B_\rho\}$ is the unique maximizer of the ratio $\R$ among simple curves, and in \cite{RSS1} it was proved that it satisfies the weak nondegeneracy condition \eqref{cond2}. We in fact have:

\begin{proposition}
   The ratio $\R$ in the ball $B_\rho$ satisfies the strong nondegeneracy condition of Definition~\ref{qdeu} if $\rho$ is small enough.
\end{proposition}

\begin{proof}

We may define coordinates $x\in[0,1)$ and $z\in(-1,1)$ on the half disk $D_{\rho}^+ = \{R+iZ\mid \text{$R^2+Z^2\le \rho^2$ and $R>0$}\}$ in $\C$ as follows~: The point with coordinates $x$, $z$ is $\rho$ times the image of the complex number $iz$ by the Moebius transform $\varphi_x(w) = (x+w)/(1+xw)$, which maps the vertical segment $i[-1,1]$ to the intersection of the circle centered at $(1+x^2)/2x$ with the unit disk.   We thus have
$$R+iZ = \rho \frac{x+iz}{1+ixz}.$$

We may then extend straightforwardly this coordinate system to the ball $B_\rho$ in $\RR^3$ by requiring that a point with cylindrical coordinates $(R,\theta,Z)$ in $B_\rho\setminus\{(0,0,\pm\rho)\}$ has coordinates $(x,\theta,z)$, where $x\in[0,1)$ and $z\in(-1,1)$ such that 
$$R = \rho x \frac{1+z^2}{1+x^2z^2},\quad Z = \rho z \frac{1-x^2}{1+x^2z^2}.$$

The Euclidean metric $dR^2+dZ^2+R^2d\theta^2$ then becomes
$$ g(x,z,\theta) = \frac{\rho^2}{(1+x^2z^2)^2}\( (1+z^2)^2\,dx^2+(1-x^2)^2\,dz^2 + x^2(1+z^2)^2\,d\theta^2\).$$
It is straightforward to compute the second derivative of the length. Let
$$\ga_t (z) = (tx(z),z,\theta(z))$$
be a family of curves parametrized by $z\in[-1,1]$ in our coordinates. Then 
\begin{equation}\label{hesslength}
   \frac{d^2}{dt^2}_{|t=0}|\ga_t| = \rho\int_{-1}^1 \hal{x'}^2(1+z^2)^2 -x^2 (1+z^2) +\hal x^2(1+z^2)^2{\theta'}^2\,dz.
\end{equation}
To compute the second derivative of $\pr{B_0,\ga_t}$, we resort to the explicit expression for $\curl B_0$ computed in \cite{Lon}: 
\begin{equation*}
	\curl B_0 = \frac{3\rho}{2\sinh\rho}\(\cosh r - \frac{\sinh r}r\)\frac{\sin\phi}r \mathbf e_\theta,
\end{equation*}
where $(r,\phi,\theta)$ are spherical coordinates on $B_\rho$, so that $R = r\sin\phi$ and  $Z = r\cos\phi.$
It follows that
$$\curl B_0\cdot \mathbf e_\theta = \frac32\frac{r^2}{3} \frac{\sin\phi}r +O(\rho^3) = \frac R2+ O(\rho^3).$$
We deduce that, denoting $D_\rho^+ = \{(R,Z)\mid \text{$R^2+Z^2\le \rho^2$ and $R>0$}\}$, 
\begin{equation}\label{bogo}
   \pr{B_0,\ga_0} = \iint_{D_\rho^+} \curl B_0 \cdot \mathbf e_\theta= \int_{R=0}^\rho\frac R2 \times 2\sqrt{\rho^2-R^2}\,dR + O(\rho^5) = \frac{\rho^3}3+O(\rho^5).
\end{equation}
We also find that 
$$\(\curl B_0 \cdot \mathbf e_\theta\)\,rdr\wedge\,d\phi = \frac{\rho^3}2\frac{x(1-x^2)(1+z^2)^2}{(1+x^2z^2)^3} dx\wedge dz+O(\rho^5).$$
This allows us to compute 
$$\pr{B_0,\ga_0} - \pr{B_0,\ga_t}  = \iint_{A_t}\curl B_0 \cdot \mathbf e_\theta\,dx\,dz,$$
where $A_t=\{(s,z)\mid\text{$-1\le z\le 1$ and $ 0\le s\le tx(z)$}\}$, from which we compute
\begin{equation}\label{hessmag}
   \frac{d^2}{dt^2}_{|t=0} \pr{B_0,\ga_t} =  - \frac{\rho^3}4\int_{-1}^1 x^2(1+z^2)\,dz+O(\rho^5).
\end{equation}
Since $\ga_0$ maximizes the ratio $\R$, its differential at $\ga_0$ vanishes and then, for $t=0$,
$$\frac{d^2}{dt^2}_{|t=0}\R(\ga_t) = \frac{\frac{d^2}{dt^2}_{|t=0}\pr{B_0,\ga_t}\ellzero - \pr{B_0,\ga_0}\frac{d^2}{dt^2}_{|t=0}|\ga_t|}{\ellzero^2},$$
so that, in view of \eqref{hesslength}, \eqref{bogo}, \eqref{hessmag} and the fact that $\ellzero = 2\rho$, we have after simplification
$$ \frac{d^2}{dt^2}_{|t=0} \R(\ga_t) = -\frac{\rho^2}{24}\int_{-1}^1 (1+z^2)^2({x'}^2+x^2{\theta'}^2)+ x^2(1+z^2)\,dz+O(\rho^3).$$
Therefore the quadratic form $d^2\R(\ga_0)$ is definite negative if $\rho$ is small enough.
\end{proof}

\section{Lower bounds by slicing}\label{sec:lowerboundslicing}
In this section we prove two types of lower bounds for the free energy $F_\ep$ contained in a tubular neighborhood of $\ga_0$, and obtained by integrating in the $z$ coordinate two-dimensional lower bounds over slices. The first type is very robust and is obtained by the ball construction method (\`a la \cites{San0,Jer}) but in a two-step growth process that allows to retain more information on the degrees at small scales. It retains an error which is larger than a constant. In this construction, the varying weight is approximated by a constant weight, leading to errors that can be absorbed into the ball construction errors. 

 The second type of lower bound is more precise, it recovers the exact constant $\gamma$ and an error only $o(1)$. To obtain such a precise error, the varying weight can no longer be approximated by a constant, instead one needs to resort to performing the ball construction in the precise geometry we are working in, i.e.~we need to grow geodesic balls. The techniques thus combine ball construction methods within the geometric framework to capture the energy on large scales, with the refined \cite{BetBreHel} analysis found in of \cites{almeidabethuel,Ser} to capture the  precise energy at small scales while approximating the weight by a uniform one.
 
\subsection{Lower bounds by ball growth method}

First, recalling \eqref{propchi} and \eqref{defprem}, which correspond respectively to the definitions of the smooth cutoff $\chi$ and $\prem{\mu}$, we have the following result.

\begin{proposition}\label{lowperp}
   Assume that $F_\ep(u,A) < C_0\lep$, that $\ga_0$ is critical for the ratio, and that $N_0\in\N$, $\dl\in(\lep^{-1},\delta/2)$  are such that
   $$ \left\|\mu(u,A) - 2\pi N_0\ga_0\right\|_{\mathcal F(\Omega)} \le  \dl.$$
   Then, writing for short $\mu$ instead of $\mu(u,A)$, we have
   \begin{equation*}
   	\fepperp(u,A) \ge \frac\lep 2\ \prem{\mu}+\pi \ellzero N_0 (N_0-1)\log\frac1{\dl} -\mathrm{Rem},\end{equation*}
   where $\mathrm{Rem}\le C\(1+\dl\log\lep\)$.

\end{proposition}

This result will be  proven by 
obtaining  lower bounds for the energy on the slices $\Sigma_z = \Cd\cap\RR^2\times\{z\}$ and integrating them with respect to $z$.

Let us first state and prove the two-dimensional lower bounds on slices 
 $\Sigma_z$. This is an adaptation of the vortex ball construction (here, specifically \cites{San0,SanSerBook}), with a two-stage growth process made to handle the varying weight. Such a two-stage growth process was already used in \cite{SanSerBook}*{Chapter 9}.
 \begin{proposition} \label{2Dlb}
   For any $q>0$, $C>0$ and integer $N$, there exist $\ep_0$, $C'>0$ such that for any  $z$,  any $(u, A)$, and any $\ep<\ep_0$, the following holds.

   Assume that
   \begin{equation}\label{boundclep}\int_{\Sigma^z}
      e_\ep^\perp(u,A) d\vol_\gperp<C \lep, \end{equation}
      where $e_\ep^\perp$ is defined as in \eqref{defeperp}, 
   and  let
   $$\dl^z =\max\( \|\mu^z - \pi N\delta_0\|_{\mathcal F(\Sigma^z)},\lep^{-1}\),\quad\text{where $\mu^z = \mu_{12}\, dx\wedge dy$}.$$
   Then there exists points $a_1,\dots,a_k\in\Sigma^z$ and integers $d_1,\dots,d_k$ such that, denoting $n = \left|d_1+\dots+d_k\right|$ and $D = |d_1|+\dots+|d_k|$, we have
   \begin{multline*} \int_{\Sigma^z} e_\ep^\perp(u,A)\,d\vol_\gperp \ge \pi \sum_{i=1}^k |d_i|\sqrt{\gz(a_i)} \(\log\frac {\dl^z}{\ep} - C'\) \\
      +\pi N^2 \log\frac \delta {\dl^z} - C'\log\lep \({\dl^z} +(D-n)\), \end{multline*}
   where $\delta$ is as in Proposition~\ref{prop:diffeo}, and
   \begin{equation}\label{jacslice} \left\|\pi\sum_{i = 1}^k d_i \delta_{a_i} - \mu^z\right\|_{\mathcal F(\Sigma^z)}\le C' \lep^{-q},\quad D\le C'.\end{equation}

\end{proposition}
Later, we will be able to show that ${\dl^z}$ is very small and $D=n=N$, which will allow to translate this lower bound into a lower bound with $O(1)$ error.

\begin{proof}[Proof of Proposition~\ref{2Dlb}]
   In what follows, $C'$ is any constant depending on $q$, $C$, and the inequalities involving $\ep$ are understood to hold for any $\ep$ small enough depending on $C$ and $q$. We write for short $p(x) = \sqrt{\gz(x)}$ and will use that notation throughout the rest of the section.

{\bf Step 1: lower bound by two-stage ball growth}.
We use the  ball construction, or ball growth procedure relative to the metric $g^\perp$ on $\Sigma^z$,  see \cite{SanSerBook} in the Euclidean case,  the metric does not affect the construction as it corresponds to making locally an affine change of coordinates, see  \cite{ssproductestimate} for detail (constructions for a surface may be found in \cite{ignatjerrard}*{Proposition 8.2} or \cite{DueSer} for instance). 
 Given a  final radius $r = \lep^{-Q}$, with $Q$ a large number, we thus obtain a collection $\B = \{B_i\colonequals B(a_i,r_i)\}_i$ of disjoint balls such that $r = \sum_i r_i$, such that
   \begin{equation}
   \label{lowerbd33}\int_{B_i} \frac{1}{p} e_\ep^\perp(u,A)\,d\vol_\gperp \ge \pi |d_i| \(\log\frac r{D\ep} - C'\),\end{equation}
   the degree of $u$ over $\partial B_i$ is equal to $d_i$ and \eqref{jacslice} is satisfied.
   Moreover, $D = |d_1|+\dots+|d_k|$ is bounded by $C'$, for otherwise by disjointness of the balls the lower bound \eqref{lowerbd33} would contradict \eqref{boundclep}.

   Since on $B_i$ we have $p(x)\ge p(a_i) - Cr_i$ and $r_i\le \lep^{-Q}$, it  follows that
   \begin{equation}\label{smball}\int_{B_i}  e_\ep^\perp(u,A)\,d\vol_\gperp \ge \pi |d_i| p(a_i)\(\log\frac r{\ep} - C'\).\end{equation}
   Then, since from its definition ${\dl^z}>r$,  using the ball growth method  of \cite{San0} (see for instance \cite{SanSerBook}*{Chapter 4}) we may grow the balls from  total radius $r$ to total radius ${\dl^z}$, at which point the balls will have grown into a collection $\B' = \{B_j'\colonequals B(a_j',r_j')\}_j$, with degrees $d_j'$, such that 
   for each $B'_j$ we have
   \begin{equation*}
   	\int_{B_j\setminus\B}   e_\ep^\perp(u,A)\,d\vol_\gperp\ge \pi |d_j'| \(p(a_j') - C{\dl^z}\)\( \log\frac {\dl^z}{r} - C'\).\end{equation*}
   Moreover, by additivity of the degree and disjointness of the balls,   we must have 
   \begin{equation}
   \label{additidegree}
   d_j'= \sum_{i, a_i\in B_j'} d_i.\end{equation}

   Since $\log({\dl^z}/r)\le C'\log\lep$ we find
   \begin{equation}\label{midball}\int_{B_j\setminus\B} \  e_\ep^\perp(u,A)\,d\vol_\gperp\ge \pi |d_j'| p(a_j')\( \log\frac {\dl^z}{r} - C' - C'{\dl^z}\log\lep\).\end{equation}

{\bf Step 2: lower bound by integration over large circles.} We next retrieve the degree squared contribution to the energy outside of the small balls following the method of integration along circles of \cite{SanSer2}.
   
   Denote by $\L$ a minimal connection between the measure $\sum_{i = 1}^k d_i \delta_{a_i}$ and $N \delta_0$, relative to the metric $g^\perp$, allowing connections to the boundary. Then  in view of \eqref{jacslice}, $|\L|\le C'({\dl^z}+ \lep^{-q})$. The fact that ${\dl^z}$ is the flat distance for the Euclidean metric and not $g^\perp$ is accounted for by the constant $C'$.  Moreover, if the circle $\CC_s$ (relative to $g^\perp$) of center $0$ and radius $s$ does not intersect $\L$, then the winding number of $u$ on $\CC_s$ is equal to $N$.  Since we have $p(x)\ge 1 - Cs$ on $\CC_s$, and since the length of $\CC_s$ relative to $g^\perp$ is bounded above by $2\pi s + C's^2$, we obtain for such an $s$ that
   $$\int_{\partial \CC_s} e_\ep^\perp\,d\vol_\gperp \ge (1 - C's)\frac{\pi N^2}s.$$
   But the measure of the set of $s$ such that $\CC_s$ intersects either $\B'$ or $\L$ is bounded above by $2{\dl^z}+|\L|$, since $\B'$ has total radius ${\dl^z}$. Since $|\L|$ is the flat distance of $\sum_{i = 1}^k d_i \delta_{a_i}$ to $N \delta_0$, and from \eqref{jacslice} and the definition of ${\dl^z}$, this measure is less than $4{\dl^z}$, if $\ep$ is small enough. Integrating the circle bound above with respect to $s$ such that $\CC_s$ intersects neither $\B'$ nor $\L$, we obtain
   $$ \int_{\Sigma\setminus \B'}e_\ep^\perp\,d\vol_\gperp  \ge \int_{4{\dl^z}}^\delta (1 - C's)\frac{\pi N^2}s\,ds\ge \pi N^2\(\log\frac \delta {\dl^z} - C'\).$$
   Note that if $4 {\dl^z}> \delta$ the lower bound remains true, since the left-hand side is nonnegative.
   Adding this to \eqref{midball} and \eqref{smball}  we find
   \begin{multline}\label{uneq}
      \int_\Sigma  e_\ep^\perp\,d\vol_\gperp  \ge \pi\sum_i |d_i| p(a_i) \(\log\frac r\ep - C'\) +  \pi\sum_j |d_j'| p(a_j') \log\frac {\dl^z} r \\
      + \pi N^2 \log\frac \delta {\dl^z} - C'- C'{\dl^z}\log\lep.
   \end{multline}
   Then we note that, for any $j$, in view of \eqref{additidegree}, we have
   $$\sum_{i, a_i\in B_j'} |d_i|p(a_i)\le |d_j'| (p(a_j')+C{\dl^z}) + C'\(\Big(\sum_{i, a_i\in B_j'}|d_i|\Big) - |d_j'|\).$$
   Summing with respect to $j$, we find that
   $$\sum_j |d_j'| p(a_j') \ge \sum_i|d_i|p(a_i) - C'\({\dl^z} + D-n\),$$
   where now the sums run over all indices $i$ and $j$, respectively. Inserting into \eqref{uneq}, we deduce that
   \begin{multline*}\int_\Sigma e_\ep^\perp\,d\vol_\gperp  \ge \pi\sum_i |d_i| p(a_i) \(\log\frac {\dl^z}\ep -C'\)
      + \pi N^2 \log\frac L{\dl^z} - C'\(d + D - n\)\log\lep - C',
   \end{multline*}
   which together with \eqref{jacslice} proves the proposition.
\end{proof}

For slices $\Sigma^z$   for which we have a weaker energy bound, we use a cruder estimate.

\begin{lemma} \label{2Dlbvar}
   For any $q>0$ and $C>0$, there exist $\ep_0$, $C'>0$ such that, for any  $z$ such that
   $$\int_{\Sigma^z}  e_\ep^\perp \,d\vol_\gperp <C \lep^{q}, $$
   there exists points $a_1,\dots,a_k$ and integers $d_1,\dots,d_k$ such that, denoting  $\mu^z = \mu_{12}\, dx\wedge dy$ and $D = |d_1|+\dots+|d_k|$, we have
   $$ \int_{\Sigma^z} e_\ep^\perp \,d\vol_\gperp  \ge \pi \sum_{i=1}^k |d_i||\sqrt{\gz(a_i)} \(\log\frac 1{\ep} - C'\log\lep\), $$
   $$ \left\|2\pi\sum_{i = 1}^k d_i \delta_{a_i} - \mu^z\right\|_{\mathcal F(\Sigma^z)}\le C' \lep^{-q},$$
   $$ D\le C'\lep^{q-1}.$$
\end{lemma}

\begin{proof}
We apply the ball construction (see \cite{SanSerBook}*{Chapter 4}) with final radius $r = \lep^{-Q}$, with $Q$ a large number, we get a collection $\B = \{B_i\colonequals B(a_i,r_i)\}_i$ of balls for which
   $$\int_{\B} \frac{1}{p}e_\ep^\perp\,d\vol_\gperp \ge \pi D\(\log\frac r{D\ep} - C'\),$$
   where $D = |d_1|+\dots+|d_k|$, the last statement of the lemma holds and the second one holds by the Jacobian estimate.
   Since on $B_i$ we have $\sqrt\gz(x)\ge \sqrt\gz(a_i) - Cr_i$, it  also follows that
   \begin{equation*}
   	\int_{B_i} e_\ep^\perp \,d\vol_\gperp \ge \pi |d_i| \sqrt{\gz(a_i)}\(\log\frac r{D\ep} - C'\).\end{equation*}
\end{proof}

\begin{proof}[Proof of Proposition~\ref{lowperp}]
   We denote throughout the proof by $C,q$ generic large positive constants depending only on $\Omega$, $C_0$, and $\chi$ is as above.  We  write
   \begin{equation*}
   	f(z) =\int_{\Sigma^z} e_\ep^\perp(u,A)\,d\vol_\gperp,\end{equation*}
	so that the desired result is a lower bound on $\int_z f(z)dz$.

   We define for any $z$ such that $\Sigma^z$ is not empty  
   $$\dl^z = \max(\|\mu_{12}\,d\vu - 2\pi N_0\delta_0\|_{\mathcal F(\Sigma^z)},\lep^{-1})$$
 Then we define three sets of slices:
   \begin{itemize}
   \item
   We denote by $I$ the set of $z$'s such that $f(z) \le M\lep$, for some $M>0$ to be chosen below.
\item
   We denote by $J$ the set of $z\notin I$ such that $f(z) \le \lep^q$.
\item
   We denote by $K$ the set of $z\notin I\cup J$.
\end{itemize}

Let us first treat the case $z\in K$.
   Since on $\Sigma^z$, the 2-form $\mu_{12}\,d\vu$ is the exterior differential of the current $j(u,A)+A$ restricted to $\Sigma^z$, it follows that
   \begin{equation*}
      \int_{\Sigma^z}\chi\  \sqrt\gz\,\mu_{12}\,d\vu = \int_{\Sigma^z} d(\chi\  \sqrt\gz)\wedge j(u,A)+\int_{\Sigma^z}\chi\  \sqrt\gz\,dA\,d\vu\le Cf(z)^{1/2},\end{equation*}
   from the Cauchy-Schwarz inequality and the definition of $e_\ep^\perp(u,A)$.
   It then follows that  for any  $z\in K$, since $f(z) > \lep^q$ and $q>2$, we have
   \begin{equation}\label{Kterm}f(z) \ge  \frac\lep 2\int_{\Sigma^z}\chi\ \sqrt\gz\ \mu_{12}\,d\vu + \pi N_0 (N_0-1)\log\frac1{\dl}.\end{equation}

   For  $z\in I$, we may apply Proposition~\ref{2Dlb} on $\Sigma^z$  with $N = N_0$ to find that there exists points $a_1,\dots,a_k$ and integers $d_1,\dots,d_k$ such that, denoting $n = \left|d_1+\dots+d_k\right|$ and $D = |d_1|+\dots+|d_k|$, we have $D\le C$ and
   \begin{equation*} 
   	f(z)  \ge \pi \sum_i |d_i|\sqrt{\gz(a_i)} \(\log\frac{ \dl^z}{\ep} - C'\)
      +\pi N_0^2\log\frac\delta{\dl^z} - C'\log\lep \(\dl^z +(D-n)\). \end{equation*}
  Moreover
   \begin{equation}\label{Islicejac}\left\|2\pi\sum_{i = 1}^k d_i \delta_{a_i} - \mu_{12}\,d\vu \right\|_{\mathcal F(\Sigma^z)}\le C' \lep^{-q},\end{equation}
   so that, for $z\in I$,
   \begin{equation*}
   	\left|\int_{\Sigma^z}  \chi\   \sqrt\gz\,\mu_{12}\,d\vu  - 2\pi\sum_{i} d_i \chi(a_i)  \sqrt{\gz(a_i)}\right|\le C\lep^{-q}.\end{equation*}
   It follows that
   \begin{multline*}f(z)\ge \frac12\int_{\Sigma^z}\chi\   \sqrt\gz\,\mu_{12}\,d\vu  \log\frac{\dl^z}\ep+ \pi\(\log\frac{\dl^z}\ep\)\sum_i \(|d_i| - d_i \chi(a_i)\) \sqrt{\gz(a_i)}+\\
      +\pi{N_0}^2\log\frac 1{\dl^z}  - C\log\lep\(D-n+\dl^z\) - C.
   \end{multline*}
   If $z$ is such that $D>n$,  then the error terms on the right-hand side may be absorbed in the second term, noting that $\pi\log(\dl^z/\ep)\sqrt{\gz(a_i)}>\lep/C$ . If not, then  the error term may be written $C(\dl^z\log\lep+1)$. Therefore, integrating with respect to $z\in I$ we find
   \begin{multline}\label{Iterma}\int_{z\in I}f(z)\, dz\ge \int_{z\in I}\left\{\int_{\Sigma^z} \frac12\log\frac{\dl^z}\ep\chi\   \sqrt\gz\,\mu_{12}\,d\vu + \right.\\
      \left. +\pi{N_0}^2\log\frac 1{\dl^z}  - C\(\dl^z\log\lep+1\)\right\}\,dz.
   \end{multline}
   Using the fact that $\dl^z$ bounds the flat distance between $\mu_{12}\,d\vu$ and $2\pi N_0\delta_0$, we have
   $$\log{\dl^z}\int_{\Sigma_z}\chi\   \sqrt\gz\,\mu_{12}\,d\vu \ge 2\pi N_0 \log \dl^z - C\dl^z|\log \dl^z|,$$
   and $\dl^z$ being bounded, since $D\le n$,  the error term above is bounded by a constant. Replacing in \eqref{Iterma}, we obtain
   \begin{multline}\label{Iterm}\int_{z\in I}f(z)\, dz\ge \int_{z\in I}\left\{ \frac\lep 2\int_{\Sigma^z}\chi\   \sqrt\gz\,\mu_{12}\,d\vu  \right.\\
      \left. +\pi({N_0}^2 - N_0)\log\frac 1{\dl^z}  - C\(\dl^z\log\lep+1\)\right\}\,dz.
   \end{multline}

   Let us finally consider $z\in J$.  Then $f(z)\ge M\lep$, for some large constant $M$. We first exclude the case where
   $$\int_{\Sigma^z}\chi\   \sqrt\gz\,\mu_{12} < \frac{f(z)}\lep,$$
   because in this case, noting that $\pi N_0(N_0-1)\log(1/\dl)$  is bounded by $C\log\lep$, the lower bound \eqref{Kterm} is clearly satisfied.

   In the opposite case, we have from the definition of $\dl^z$ and the lower bound $f(z)\ge M\lep$ that
   \begin{equation*}
   	M - 2\pi N_0\le \int_{\Sigma^z}\chi\   \sqrt\gz\,\mu_{12}d\vu - 2\pi N_0 \le C\dl^z.\end{equation*}
   In particular, if $M$ is chosen large enough compared to $N_0$, we deduce that $\dl^z \ge N_0$ and then
   $$\int_{\Sigma^z}\chi\   \sqrt\gz\,\mu_{12}d\vu\le C\dl^z.$$
   We then apply Lemma~\ref{2Dlbvar} to find that
   \begin{equation} \label{Jlow}f(z)   \ge \pi \sum_i |d_i|\sqrt{\gz(a_i)} \(\log\frac{ 1}{\ep} - C\log\lep\).\end{equation}
   Using \eqref{Islicejac} which remains true, we may then rewrite \eqref{Jlow} as
   $$ f(z)\ge \frac\lep 2\int_{\Sigma^z}\chi\   \sqrt\gz\,\mu_{12}\,d\vu - C\dl^z\log\lep.$$
   It follows that
   \begin{equation*}
   	f(z)\ge \frac\lep 2\int_{\Sigma^z}\chi\   \sqrt\gz\,\mu_{12}\,d\vu + \pi \(N_0^2 - N_0\) \log\frac 1{\dl^z} - C\(\dl^z\log\lep+\log\frac1{\dl^z}\),\end{equation*}
   where we have added and substracted the middle term on the right-hand side. This adds an extra undesired error term, which may be absorbed into a constant since $\dl^z \ge N_0$ in this bad case. We deduce that
   \begin{multline}\label{Jterm}\int_{z\in J}f(z)\, dz\ge \int_{z\in J} \left\{ \frac\lep 2\int_{\Sigma^z}\chi\   \sqrt\gz\,\mu_{12}\,d\vu \right.\\ \left.+ \pi \(N_0^2 - N_0\) \log\frac 1{\dl^z} - C\(\dl^z\log\lep+1\)\}\right\}\,dz.\end{multline}

   Adding \eqref{Iterm}, \eqref{Jterm} and the integral of \eqref{Kterm} with respect to $z\in K$, we obtain in view of \eqref{Islicejac} and $\int_z \dl^z\,dz \le C\dl$ (see \cite{Fed}*{4.3.1}) that
   \begin{multline}\label{lowperpalmost}\int_z f(z)\,dz \ge \frac\lep 2\int_z\int_{\Sigma^z}\chi\   \sqrt\gz\,\mu_{12}\,d\vu\,dz + \int_{K} \pi \(N_0^2 - N_0\) \log\frac 1{\dl}\,dz + \\ + \int_{z\in I\cup J}\pi \(N_0^2 - N_0\) \log\frac 1{\dl^z} - C\dl^z\log\lep\,dz - C.\end{multline}

   Finally, we use the concavity of  $\log$ together with the fact that the integral of $\dl^z$ with respect to $z$ is bounded by $C\dl$ --- and also the fact that since 
   $f(z)\le C\lep$ we must have $|I\cup J|\ge \ellzero/2$ if $\ep$ is small enough --- to find that
   $$\int_{z\in I\cup J}\log\frac 1{\dl^z}\,dz\ge \int_{z\in I\cup J}\log\frac 1\dl - C,$$
   which together with \eqref{lowperpalmost} yields
   \begin{multline*}\int_z f(z)\,dz \ge \frac\lep 2\int_z\int_{\Sigma^z}\chi\   \sqrt\gz\,\mu_{12}\,d\vu\,dz + \pi \ellzero\(N_0^2 - N_0\) \log\frac 1{\dl} - C\(\dl\log\lep +1\),\end{multline*} which is the desired estimate.
\end{proof}

\subsection{Lower bound up to \texorpdfstring{$o(1)$}{o(1)} error}

\begin{proposition}\label{lowerprecise} 
 For any $q>0$, $C>0$  and integer $N$, there exist $\ep_0$, $\kappa$,  $C'>0$ such that for any  $z$,  any $(u, A)$, and any $\ep<\ep_0$, the following holds.

   Assume that
   \begin{equation}\label{bsupfictive}\int_{\Sigma^z}
      e_\ep^\perp(u,A) d\vol_\gperp<\pi (N+m) \lep, \end{equation}
      where $e_\ep^\perp$ is defined as in \eqref{defeperp},  and $m<\hal$ is a small enough constant, to be determined below.
      Assume that $\dl^z\le C\lep^{-1/2}$, where 
   $$\dl^z =\max\( \|\mu^z - 2\pi N\delta_0\|_{\mathcal F(\Sigma^z)},\lep^{-1}\),\quad\text{ $\mu^z = \mu_{12}\, dx\wedge dy$}.$$
   Then there exists points $a_1,\dots,a_N\in\Sigma^z$ such that we have 
   \begin{equation*}
		\left\| \mu^z(u, A) - 2\pi \sum_{i=1}^N \delta_{a_i}\right\|_{\mathcal F(\Sigma^z)}\le C\lep^{-q}
	\end{equation*}
	and 
   \begin{multline*}
		\int_{\Sigma^z} e_\ep^\perp(u,A)\,d\vol_\gperp \\
		\ge \pi \sum_{i=1}^N \sqrt{\gz( a_i)} \log \frac1\ep+\pi N^2 \log \delta  - \pi \sum_{i\neq j} \log \dist_{g^\perp}(a_i-a_j) + N\gamma  +o_\ep(1)+ o_\delta(1),\end{multline*}  where $\delta$ is as in Proposition~\ref{prop:diffeo} and $\gamma$ is the universal constant introduced in \cite{BetBreHel}. 

\end{proposition}

The rest of this subsection is devoted to proving this proposition. 

The first step is to show we can neglect $A$ by a good choice of gauge.
\begin{lemma} For any $q,M>0$ there exists $\ep_0,C>0$ such that if $\ep<\ep_0$, then any $(u, A)$ such that $ \int_{\Sigma^z}
      e_\ep^\perp(u,A) d\vol_\gperp \le M \lep$ is gauge-equivalent to some $(v,B)$  such that 
      \begin{equation*} \|\mu^z(u,A) - \mu^z(v,0)\|_{\mathcal F(\Sigma^z)}\le C\lep^{-q}\end{equation*}
      and such that
      \begin{equation}\label{noA}
         \int_{\Sigma^z}
         e_\ep^\perp(u,A) d\vol_\gperp \ge  \int_{\Sigma^z} e_\ep^\perp(v,0) d\vol_\gperp -o_\ep(1)  -o_\delta(1).
      \end{equation}
   \end{lemma}
\begin{proof}
We let $p = \sqrt{\gz}$. Then we chose $B = *d\xi/p$ where $\xi$ solves
\begin{equation*}
\left\{\begin{array}{ll}
   d^* d\xi - dp\cdot d\xi = p *dA & \text{in $\Sigma^z$}\\
   \xi= 0 & \text{on $\partial\Sigma^z$}.\end{array}\right.
\end{equation*}
Since $dB = dA$, $(u,A)$ is gauge-equivalent to $(v,B)$, for some $v$. Moreover, using the fact that $p$ is smooth and $|p-1|\le C\delta$ in $\Sigma^z$, elliptic estimates yield bounds for $\xi$ in $H^2(\Sigma^z)$ and imply, in particular, that 
$$\|B\|_{L^2}\le C \delta \|dA\|_{L^2}, \quad \sup_{x,y\in\Sigma^z}\frac{|\xi(x) - \xi(y)}{|x-y|^{1/2}} \le C\delta^{1/2}\|dA\|_{L^2}.$$
The $L^2$ bound for $B$ implies straightforwardly that 
\begin{equation}\label{lepforv}\int_{\Sigma^z} e_\ep^\perp(v,0)d\vol_\gperp \le C \int_{\Sigma^z} e_\ep^\perp(v,B) d\vol_\gperp \le CM\lep\end{equation}
and, since
$$\mu^z(v,0) - \mu^z(u,A) = \mu^z(v,0) - \mu^z(v,B) = d((1-|u|^2)B),$$
that 
$$\|\mu^z(v,0) - \mu^z(u,A)\|_{\mathcal F(\Sigma^z)}\le C\|B\|_{L^2}\|1-|v|^2\|_{L^2}\le C\delta \ep \lep.$$
It remains to prove that \eqref{noA} holds. We write
\begin{equation}\label{splitv}\int_{\Sigma^z} e_\ep^\perp(u,A) d\vol_\gperp - \int_{\Sigma^z} e_\ep^\perp(v,0) d\vol_\gperp \ge  \int_{\Sigma^z} \frac p2\(|dB|^2 - 2\langle j(v,0), B\rangle\) d\vol_\gperp,\end{equation}
and then note that, writing $j$ for the restriction of $j(v,0)$ to $\Sigma^z$ and $\mu=dj$, 
$$\int_{\Sigma^z} p \langle j, B\rangle d\vol_\gperp = \int_{\Sigma^z} j\wedge d\xi = \int_{\Sigma^z}\xi \mu.$$
It is well-known  that \eqref{lepforv} implies that $\mu$ may be approximated by a sum of Dirac masses with total mass bounded by $CM$ and error bounded by $CM\ep^{1/2}\lep$ in the dual of $C^{0,1/2}$. Therefore 
$$\int_{\Sigma^z}\xi \mu \le CM\(\|\xi\|_\infty  + \ep^{1/2}\lep\|\xi\|_{C^{0,1/2}}\) \le \(o_\delta(1) + o_\ep(1)\)\|dA\|_{L^2}.$$
Together with \eqref{splitv} and since $dA = dB$, \eqref{noA} follows.
\end{proof}

The proof of Proposition \ref{lowerprecise} makes use of the parabolic regularization method of \cite{almeidabethuel} (itself borrowed from a preliminary version of \cite{BetBreHel}) to define ``essential balls" for the map $u$, for that we will follow a bit also the presentation in \cites{Ser,Ser2}. 
Let $0<\eta<1$.  We define $u^\eta$ as the minimizer of 
\begin{equation*}
	\min_{v\in H^1(B_{g^\perp} (0, \delta) , \C) } \int_{\Sigma^z}\hal  \sqrt{g_{33}}  \( |dv|_{g^\perp}^2 + \frac1{2\ep^2} (1-|v|^2)^2+  \frac{|u-v|^2 }{2\ep^{2\eta}}\) d\vol_{g^\perp} .\end{equation*}
The minimum is achieved by some map $u^\eta$ (which is not necessarily unique, but we make an arbitrary choice).
The solution $u^\eta$ is regular and satisfies $|u^\eta|\le 1$ and $|d u^\eta|\le \frac{C}{\ep}$ by maximum principle. Also, by obvious comparison
\begin{equation*}
	 \int_{\Sigma^z}
      e_\ep^\perp(u,0) d\vol_\gperp\ge \int_{\Sigma^z}
      e_\ep^\perp(u^\eta,0) d\vol_\gperp.\end{equation*}
We will denote by $B_{g^\perp}$ the geodesic ball with respect to the metric $g^\perp$.
The next lemma provides  vortex balls of  small size (a power of $\ep$) which are well separated and on which the energy is well bounded below.
\begin{lemma}\label{lemab}Let $0<\eta<\beta<\mu<1$. Under the assumptions of Proposition \ref{lowerprecise}, for $\ep$ small enough, 
	there exists a set $\mathcal I$ with $\# \mathcal I$ bounded by some constant independent of $\ep$, points $(a_i)_{i\in \mathcal I}$, $\bar \mu> \beta$, and a radius $\rho>0$ such that 
	$$\ep^\mu \le \rho\le \ep^{\bar \mu}<\ep^{\beta}$$
	and such that, letting $B_i =  B_{g^\perp} (a_i, \rho) $ and $d_i = \deg \( u^\eta, \partial B_i\)$ we have:
	\begin{enumerate}[leftmargin=*,itemsep=3pt,label=\normalfont{(\roman*)}]
		\item $|a_i-a_j|\ge 8\rho$ for all $i\neq j \in \mathcal I$;
		\item $\dist (a_i, \partial \Sigma^z) \ge \ep^\beta$;
		\item if  $x\notin \cup_{i\in \mathcal I} B_i$ then
		\begin{equation}\label{modug}
			|u^\eta(x)|\ge \hal,
		\end{equation}
		whereas
		\begin{equation*}
		|u^\eta|\ge 1- \frac{2}{\lep^2}\  \mathrm{on} \  \cup_{i\in \mathcal I} \partial B_i;\end{equation*}
		\item for $i\in \mathcal I$, it holds that
		\begin{equation}\label{bornbor}
		\int_{\partial B_i} e_\ep^\perp(u^\eta,0) d\vol_{g^\perp}\le \frac{C(\beta, \mu)}{\rho};
		\end{equation}
		\item for $i\in \mathcal I$, it holds that for some $c>0$ independent of $\ep$,
		\begin{equation}\label{minodansbi}
			  \int_{B_i} e_\ep^\perp(u^\eta,0) d\vol_{g^\perp}  \ge \max \( p(a_i)  \pi |d_i| \log \frac{\rho}{\ep} + O(1), c \lep\);
		\end{equation}
		\item and for any $0<\alpha\le 1$, there exists $\kappa>0$ such that 
		\begin{equation}\label{proxjac}
			 \left\| 2\pi \sum_{i\in \mathcal I}d_i \delta_{a_i}- d (iu, du)\right\|_{\mathcal F(\Sigma^z)} \le \ep^\kappa
		\end{equation}
	\end{enumerate}

\end{lemma}
We omit the proof which follows closely the lines of \cite{almeidabethuel} or \cites{Ser,Ser2}, except with balls replaced by metric balls, and with the weight. The relation \eqref{proxjac} at the level of $u^\eta$ is a direct consequence of the Jacobian estimates, see for instance \cite{SanSerBook}*{Chapter 6}, it is also true for $u$ by the a priori bound on the energy minimized by $u^\eta$, see \cite{Ser}*{Lemma 4.2}.

Using again the shortcut $p$ for $\sqrt{g_{33}}$, we now have the following result obtained by growing the balls from the prior lemma by  a two-stage ball growth process.
\begin{lemma}
	Under the  assumptions of Proposition \ref{lowerprecise}, we have  $d_i=1$ for all $i\in \mathcal I$ and $\#\mathcal I=N$. 
	Also all the $a_i$'s, $i\in \mathcal I$,  belong to $B_{g^\perp} (0,C \dl^z)\subset B_{g^\perp}(0, C\lep^{-1/2}) $. 	Moreover, either
	\begin{enumerate}[leftmargin=*,itemsep=3pt,label=\normalfont{(\roman*)}]
		\item
		the balls $B_{g^\perp}(a_i, r)$ with $r= \lep^{-8}$ are such that $|a_i-a_j| \gg  r$ for all $i\neq j \in \mathcal I$ and for all $i\in \mathcal I$,
		\begin{equation*}
			\int_{B_{g^\perp}(a_i, r)\backslash B_i} 
			e_\ep^\perp(u^\eta,0) d\vol_{g^\perp}  \ge \pi p(a_i)    \log \frac{r}{\rho}+o(1)\end{equation*}
		or \item there exist a family of disjoint geodesic balls $\bar B_k= B_{g^\perp}(b_k, r_k)$, containing the $B_i$'s,   of total radius $\bar r= Nr_k= \frac{1}{\lep^2}$ such that, letting $\bar d_k = \deg (u^\eta, \partial \bar B_k)$, we have 
		$$
			\sum_k \int_{B_k''} e_\ep^\perp(u^\eta,0)   d\vol_{g^\perp}\\
			\ge \pi \sum_{k}  (\min_{\bar B_k} p  ) |\bar d_k|  \log \frac{\bar r}{\ep} + \pi \log \lep.$$
			\end{enumerate}
\end{lemma}
\begin{proof}
First we prove that $d_i\ge 0$ for all $i\in \mathcal I$.
Since $\sqrt{g_{33}}\ge 1-o(1)$ in the support of $\chi$, and choosing the constant $\mu$ close enough to $1$, we get in view of \eqref{minodansbi} combined with \eqref{bsupfictive} that 
$\sum_{i\in \mathcal I} |d_i|<N+2m<N+1$.
On the other hand, as seen in the proof of Proposition \ref{2Dlb}, denoting by $\L$ a minimal connection between  $\sum_{i} d_i \delta_{a_i}$ and $N \delta_0$ relative to the metric $g^\perp$ and allowing connections to the boundary, in view of \eqref{proxjac} we have $|\L|\le C(\dl^z+ \ep^\kappa)\le C \dl^z$.  This implies that there must be at least $N$ points (counted with multiplicity) of $d_i \ge 0$  connected to $0$. Since $\sum_{i\in \mathcal I} |d_i|\le N$, this is only possible if    for each $i\in \mathcal I$, $d_i\ge 0 $, and no  point is connected to the boundary.  This shows that all the points $a_i$ are at distance $\le C\dl^z$ from $0$.

We now turn to the lower bounds.	Let us  grow the geodesic balls $B_i$ via  the ball construction method with weight and metric, exactly as was done in Proposition \ref{2Dlb}, using as final total radius parameter $r'= \frac{1}{\lep^4}$.  This gives a collection of geodesic balls $B_j'$. We then regrow the balls into geodesic balls of final total radius $r''= \frac{1}{\lep^2}$, this gives a collection $B_k''$.
	We have for every $j$,
	\begin{equation}\label{minodansbj}
	\int_{B_j'\backslash \cup_{i\in \mathcal I} B_i}  
	e_\ep^\perp(u^\eta, 0)d\vol_{g^\perp} \ge \pi |d_j'| \Big(\min_{B_j'} p\Big)\( \log \frac{r'}{\ep}-C\)\end{equation}
	and for every $k$, 
	\begin{equation}\label{minodansbk}\int_{B_k''\backslash \cup_{j} B_j'} e_\ep^\perp(u^\eta, 0) d\vol_{g^\perp} \ge \pi |d_k''| \Big(\min_{B_k''} p\Big)\( \log \frac{r''}{r'}-C\).\end{equation}
	Let us now consider a final ball $B_k''$  and $d_k'' $ its degree. Since all  initial degrees were seen above to be nonnegative, we have 
	$|d_k''|= \sum_{i, B_i \subset B_k''} |d_i|$.
	We next add all the energy contributions inside $B_k''$ from \eqref{minodansbi}, \eqref{minodansbj} and \eqref{minodansbk}.  If $B_k''$ contains an initial  ball $(B_i)_{i\in \mathcal I} $ of degree $d_i=0$, then we use  the lower bound by $c\lep$ in \eqref{minodansbi} and we deduce that 
	\begin{multline*}\int_{B_k''}  e_\ep^\perp(u^\eta, 0) d\vol_{g^\perp} \\ 
		\ge \pi |d_k''|\Big(\min_{B_k''} p\Big) \log \frac{r''}{\ep} + c\lep - C
		\ge  \pi |d_k''|\Big(\min_{B_k''} p\Big) \log \frac{r''}{\ep}+ \hal c\lep ,\end{multline*} for $\ep $ small enough.
	Summing over all the balls and comparing with the upper bound \eqref{bsupfictive}, we conclude to a contradiction if $m$ is taken small enough compared to $c$. 
	We have thus shown that $d_i\ge 1$ for all $i \in \mathcal I$, i.e.~for all the balls $B_i$ obtained from Lemma \ref{lemab}.
		Since all degrees are nonnegative, the ball growth procedure from the $B_i$'s to the $B_j'$'s yields really an energy at least $$\pi \min_{B_j'} p \sum_{i, B_i \subset B_j'} |d_i|^2\( \log \frac{r'}{C\rho}-C\).$$
		If $d_i\ge 2$ for some $i\in \mathcal I$, this gives at least an extra energy of $c'\lep$, for some $c'>0$, than announced, which again yields a contradiction with \eqref{bsupfictive} if $m$ is chosen small enough.
	Thus, we have shown that $d_i=1$ for all $i \in \mathcal I$.

	Next, let us first consider the case where for some $k$, one $B_j'\subset B_k''$ contains more than one $B_i$. Since   all the $d_i= 1$,  this implies  $d_j'\ge 2$. 
	Then, since all degrees are nonnegative, the ball growth procedure from the $B_j'$'s to the $B_k''$'s yields really an energy $$\pi \min_{B_k''} p \sum_{j, B_j' \subset B_k''} |d_j'|^2\( \log \frac{r''}{r'}-C\),$$ so at least an extra $\pi \log\lep$ energy compared to what was announced.
	Summing over all balls, we may thus deduce in this case that 
	$$
		\sum_k \int_{B_k''} e_\ep^\perp(u^\eta, 0) d\vol_{g^\perp}\\
		\ge \pi \sum_k \Big(\min_{B_k''} p\Big) |d_k''| \log \frac{r''}{\ep} + \pi \log \lep.$$
	The proof is concluded via item (ii) with the $\bar B_k$'s equal to the $B_k''$'s.
	There remains to consider the case where none of the $B_j'$'s contain more than one $B_i$, which means that $d_j'=1$ and that there are no mergings in the growth from $B_i$ to $B_j'$: each $B_j'$ contains a unique $B_i$, which is simply inflated.
	Let us then consider new $B_i'$'s equal to the geodesic balls of centers equal to the $a_i$ and radii $r=\frac{1}{\lep^8}$, i.e. we restart the ball construction from the $B_i$. Since there were no mergings previously, it means that these balls $B_i'$ are disjoint, and separated by at least $\lep $ times their radii. Moreover, the energy over the annulus $B_j'\backslash B_i$ is bounded below by $\pi p(a_i) \log \frac{r}{\ep}-o(1)$ (the error $o(1)$ is due to the variation in $p$, and to $1-|u^\eta|$, which is very small by \eqref{modug}). The proof is then concluded in this case as well.\end{proof}

\begin{proof}[Proof of Proposition \ref{lowerprecise}]
	Now that we know that all $d_i=1$, we may also refine the upper bound \eqref{bsupfictive} into 
	$$\int_{\Sigma^z}
      e_\ep^\perp(u,A) d\vol_\gperp \le \pi (N+o(1))\lep$$
	because otherwise what we want to prove is true. 
	Thanks to that, an examination of the proof in \cites{almeidabethuel}, \cite{bethuelsaut}*{Appendix}  shows that we may refine \eqref{bornbor} into 
	\begin{equation*} 
		\int_{\partial B_i}
		e_\ep^\perp(u^\eta,0)d\vol_{g^\perp} \le \frac{\pi p(a_i)+o(1)}{\rho}
	\end{equation*}
	and thanks to this bound we have 
	\begin{equation*}
		\int_{\partial B_i}e_\ep^\perp(u^\eta,0) d\vol_{g^\perp}
		\ge \pi p(a_i) \log \frac{\rho}{\ep}+ \gamma+o(1).
	\end{equation*}
	This is an adaptation from the analysis of \cite{BetBreHel}, with metric, and  here $\gamma $ is the constant of \cite{BetBreHel}.
	
	Combining all these results, we have obtained either 
	a collection of $N$ balls $B_{g^\perp}(a_i, r)$, $r= \lep^{-8}$ such that $|a_i-a_j|\gg r$ and  for each $i$, 
	\begin{equation*}
		\int_{B_{g^\perp}(a_i, r)}e_\ep^\perp(u^\eta, 0) d\vol_{g^\perp}\ge \pi p(a_i) \log \frac{r}{\ep}+ \gamma +o(1)\end{equation*}
	or a collection of balls $B_{g^\perp}(b_k, r_k)$ with $r_k \le \lep^{-2}$ and 
	\begin{equation}\label{minonj2}
		\int_{\cup_k B_{g^\perp}(b_k, r_k)}e_\ep^\perp(u^\eta, 0) d\vol_{g^\perp}\ge \pi \sum_i p(a_i) \log \frac{1}{\lep^2 \ep}+\pi\log \lep.\end{equation}Let us call the two cases Case 1 and Case 2. 
	Let $\ell = C\lep^{-1/2}$ with $C \ge 2$ be such that all the balls $B_i$  are in $B_{g^\perp}(0, \ell/2)$. For any $r \in (\ell, \delta)$, in view of \eqref{modug}, the fact that $d_i=1$ and $\#\mathcal I=N$,  we have that $\deg(u^\eta, \partial B_{g^\perp}(0,r))=N$. 
	
	In Case 2, we may grow the balls $B_{g^\perp}(b_k,r_k)$ further to reach a final total radius $\lep^{-1/2}$, and still all the balls will  be included in $B(0, \ell)$. We retrieve this way an extra energy 
	\begin{equation}
		\label{minonj22}
		\int_{B(0,\ell) \backslash \cup_k B_{g^\perp}(b_k, r_k)}
		e_\ep^\perp(u^\eta, 0)
				d\vol_{g^\perp} \ge
		\pi \sum_{i\in \mathcal I}  p(a_i)
		\log (\ell \lep^2)+o(\log \lep) .\end{equation}
 Integrating then  over circles of radius $r \in (\ell, \delta)$, for instance as in Step 2 of  the proof of Proposition \ref{2Dlb},  and using $p \ge 1 -o(1)$, we also obtain
	\begin{equation}\label{intcircles}\int_{B(0, \delta)\backslash B(0, \ell)} 
	e_\ep^\perp(u^\eta, 0)d\vol_{g^\perp}\ge  \pi (1-o(1)) N^2 \log \frac{\delta}{\ell}.\end{equation}
	Adding  \eqref{minonj2}, \eqref{minonj22}  and \eqref{intcircles}, we obtain that
	\begin{equation*}
		\int_{\cup_k B_{g^\perp}(b_k, r_k)}e_\ep^\perp(u^\eta, 0)d\vol_{g^\perp} \ge \pi \sum_{i\in \mathcal I} p(a_i) \log \frac{1}{\ep}+ \pi N^2 \log \delta+ \pi N(N-1) \log \frac1\ell+  \frac\pi2 \log \lep,
	\end{equation*}
	which implies the desired inequality.

	Let us now turn to Case 1.  We consider the energy in $B(0, \lep^{-1/8}) \backslash \cup_i B(a_i, r)$. In this punctured domain, $p=1+O(\lep^{-1/8})$, so we may use this as a bound from below and get 
	\begin{multline*}
		\int_{B_{g^\perp} (0,\lep^{-1/8} )\backslash \cup_i B_{g^\perp}(a_i, r)}
		e_\ep^\perp(u^\eta, 0)
d\vol_{g^\perp}\\ \ge
		(1-O(\lep^{-1/8})  \int_{B_{g^\perp} (0,\lep^{-1/8} )\backslash \cup_i B_{g^\perp}(a_i, r)} \hal  \(|du^\eta|_{g^\perp}^2 + \frac{1}{2\ep^2}(1-|u^\eta|^2)^2 \)
d\vol_{g^\perp}\end{multline*}
	To bound from below the right-hand side we may use the Bethuel--Brezis--H\'elein theory with metric $g$, for instance as written down in \cite{ignatjerrard}*{Section 2.2}.
	This yields 
	\begin{multline*}
		\int_{B_{g^\perp} (0,\lep^{-1/8} )\backslash \cup_i B_{g^\perp}(a_i, r)}e_\ep^\perp(u^\eta, 0)
d\vol_{g^\perp}\\ \ge
		(1-O(\lep^{-1/8}) \( \pi N \log \frac1r+  W_{g^\perp} (a_i, \dots, a_N)+o(1) \)\end{multline*}
	where 
	$$W_{g^\perp}(a_1, \dots , a_N)= - \pi \sum_{i\neq j} \log \dist_{g^\perp}(a_i,a_j)+ \pi \sum_{i, j} R(a_i, a_j)$$
	and 
	$$R(x,y)= 2\pi G(x,y) + \log \dist_{g^\perp}(x,y),$$
where $G$ is the Green function, solution of (see for instance \cite{Aubin}*{Chapter 2})
	$$\left\{\begin{array}{ll}-\Delta_{g^\perp} G(x,y)= \delta_y  & \text{in}\  B_{g^\perp}(0, \lep^{-1/8}) \\
		G(x,y)= 0 & \text{for} \ x\in \partial B_{g^\perp}(0, \lep^{-1/8}).\end{array}\right.$$
		Since we know that all $a_i \in B_{g^\perp}(0,C \lep^{-\hal})$ we have 
	that $R(a_i, a_j) \sim R(0,0)$ as $\ep \to 0$. Moreover,  $R(0,0)$ is easily computed to be $\log \lep^{-1/8} $.
	We thus have found 
	\begin{multline*}
		\int_{B_{g^\perp} (0,\lep^{-1/8} )\backslash \cup_i B_{g^\perp}(a_i, r)}
		e_\ep^\perp(u^\eta, 0)
d\vol_{g^\perp} \\ \ge
		\pi N \log \frac1r+
		\pi \sum_{i\neq j} \log \dist_{g^\perp}(a_i,a_j) + \pi N^2 \log \lep^{-1/8}+o(1).
	\end{multline*}
	Finally, we bound from below the energy over $B_{g^\perp} (0, \delta) \backslash B_{g^\perp} (0, \lep^{-1/8})$. As in \cite{Ser01}*{Lemma A.1},  the co-area formula and the energy upper bound yield the existence of  $t\in \left[1-\frac{3}{\lep^2}, 1-\frac{2}{\lep^2}\right]$  such that 
	$\mathcal H^1 (|u^\eta(x)|=t) \le C \ep \lep^3$. Since $|u^\eta|\ge 1-\frac{2}{\lep^{2}}$ on $\partial B_i$, this implies that the set $S$ of $r$'s such that $\{|u^\eta|<t\}$ intersects  $\partial B_{g^\perp} (a_i, r)$ is of measure less than $C \ep \lep^3$.
 Using also that $p(x)\ge 1- O(|x|)$ and $\deg (u^\eta)= N$,  we may now bound from below 
	\begin{multline*}
		\int_{B_{g^\perp} (0,\delta )\backslash \cup_i B_{g^\perp}(a_i, r)}
	e_\ep^\perp(u^\eta, 0)
	d\vol_{g^\perp}\\
		\ge \hal \(1- \frac{3}{\lep^2}\) \int_{[\lep^{-1/8}, \delta]\setminus S} (1- C r) \frac{(2\pi N)^2} {2\pi r} dr \ge \pi N^2 \log \frac{\delta}{\lep^{-1/8}} - C \delta. \end{multline*} Here we have optimized by checking that the smallest value of the integral is taken when $S$ is at the lower end of the interval.
		
	Adding all the results we conclude that 
	\begin{equation*}
		\int_{B_{g^\perp} (0,\delta )}e_\ep^\perp(u^\eta, 0)
 d\vol_{g^\perp}
		\ge \pi \sum_{i=1}^N p(a_i) \log \frac\delta\ep + N\gamma - \pi \sum_{i\neq j} \log \dist_{g^\perp}(a_i,a_j)+o(1)- C \delta\end{equation*}
	hence the result is proved.
	\end{proof}
\section{Distance of filaments to \texorpdfstring{$\ga_0$}{Γ\textzeroinferior} and main lower bound}\label{sec8}

In this section, we complete the proof of the lower bound part of the main theorem, by proving the following.
\begin{proposition}
   Assume that the smooth simple curve $\ga_0$ is a unique nondegenerate maximizer of the ratio $\R$. For any $\ep>0$, assume $h_\ex=H_{c_1}^0+K\log \lep$ with $K$ bounded independently of $\ep$, and let $(\u,\A)$ be a minimizer  of $GL_\ep$ and $(u,A)= (e^{-ih_\ex\phi_0}\u,\A-h_\ex A_0 )$.

   Then for any sequence $\{\ep\}$ tending to $0$, there exists a subsequence such that $\mu(u,A)/2\pi$ is well approximated by a sum of $N$ simple curves $\ga_1,\dots,\ga_N$, where $N$ is independent of $\ep$ in the sense that,  for an arbitrarily small $\delta>0$ in the tube coordinates definition,   \begin{equation}
      \label{flatstar}
      \left\|\sum_i\ga_i-\mu(u,A)\right\|_{*} \leq  C\frac{(\log\lep)^2}{\lep^{4/3}}, \quad \left\|\sum_i\ga_i-\mu(u,A)\right\|_{\F(\Omega_\ep)} \leq  C\frac{(\log\lep)^2}{\lep^{4/3}},
   \end{equation}
   where
   \begin{equation}\label{omeps}\Omega_\ep=\left\{x\in \Omega\mid \mathrm{dist}(x,\partial\Omega)> |\log\ep|^{-2}\right\}.\end{equation}

   Moreover, the curves $\ga_i$ converge as $\ep\to 0$ to $\ga_0$ uniformly, and writing $\ga_i(t) = \ga_0(z_i(t))+\vu_i(t)$  in tube coordinates as piecewise graphs over $\ga_0$, then the rescaled curves
   $$\tga_i(t) = \ga_0(z_i(t)) + \sqrt{\frac{h_\ex}{N}}\vu_i(t)$$
   converge as $\ep\to 0$ in $\|\cdot \|_*$ norm to $\ga_i^*(z) = \ga_0(z)+\vu_i^*(z)$ and
   \begin{multline}\label{lbw} GL_\ep(\u,\A)\ge \he^2J_0 + \frac\pi2 \ellzero N(N-1)\log h_\ex  -2\pi K\R_0 \ellzero   N\log\lep\\-\frac\pi2 \ellzero N(N-1)\log N + W_N(\ga_1^*,\dots,\ga_N^*) + \pi \ellzero N\gamma + N^2 C_\Omega+o_\ep(1) + C_\delta o_\ep(1) + o_\delta(1),
   \end{multline}
   where  $W_N$ is as in \eqref{defW}.
   Finally, if $N>1$, then $\max\limits_{1\leq i\leq N}\|\vu_i^*\|_{L^\infty} >0$.
\end{proposition}

The proof of this proposition involves several steps, the first goal being to compute a lower bound for $GL_\ep(\u,\A)$ up to $O(1)$ in terms of a suitable vortex filament approximation of the vorticity measure $\mu(u,A)$, which then allows to determine the typical distance from the filaments to $\ga_0$, and then improve the lower-bound to $o(1)$ precision.

The first step is to choose the scale $\ell$ of the horizontal blow-up  in a way such that the vorticity remains concentrated near $\ga_0$ at this scale (Step~1), which in turn implies (Step~2)  that we may bound from below $\fperp(u,A)$ in terms of the vorticity in the $\ell$-tube around $\ga_0$. In Step~3 we construct vortex filaments for the tube blown-up at scale $\ell$ horizontally, and show that the distance of the vortex filaments to $\ga_0$ is in fact much smaller than $\ell$, which allows to apply Proposition~\ref{coercive} to bound from below  in a sufficiently precise way  $\fep(u,A)$. The final step uses the resulting lower bound of $GL_\ep(\u,\A)$ and, combining with the matching upper bound, draws the consequences stated above.

\begin{proof}

Throughout the proof we write $\mu$ instead of $\mu(u,A)$.
	
	We start with a preliminary claim, that 
	there exists $C>0$ such that for any curve $\ga$ with no boundary in $\Omega$ and any vector field $X$, we have
   \begin{equation}\label{smallcurves}
      |\pr{X,\ga}| \le C|\ga|^2 \|\curl X\|_{L^\infty}.
   \end{equation}
Indeed,    given $\ga$, there exists a surface $\Sigma$ such that $\ga = \partial\Sigma\cap\Omega$ and such that $|\Sigma|\le C|\ga|^2.$ Then
   $$\pr{X,\ga} = \int_\Sigma \curl X,$$
   from which the claim follows.

\begin{enumerate}[label=\textsc{\bf Step \arabic*.},leftmargin=0pt,labelsep=*,itemindent=*,itemsep=10pt,topsep=10pt]
	
  \item \emph{$\mu$ is $\ell$-concentrated near $\ga_0$.}

   From the nondegeneracy hypothesis and Proposition~\ref{weakstrong}, Condition~\ref{nondegencond} is satisfied with $P = 2$. Then, from Theorem~\ref{teo:boundedvorticity} applied for instance with $\alpha = 3/10$,  we know that for any $\ep>0$ there exists Lipschitz curves $\ga_1,\dots,\ga_k$, with $k$ bounded independently of $\ep$, and a normal current $\tga$ without boundary in $\Omega$ such that for $1\le i\le k$ we have,
   \begin{equation}\label{goodga0}   \|\ga_i - \ga_0\|_*\le \frac C{\lep^{1/7}},\quad \left||\ga_i| - \ellzero\right|\le \frac C{\lep^{1/7}}
   \end{equation}
   and such that 
   \begin{equation}\label{badga}
      |\tga|\leq \frac C{\lep^{2/3}}.
   \end{equation}
   Moreover,
   \begin{equation}\label{jacgoodbad}
      \left\|\mu-2\pi\sum_{i=1}^{N} \ga_i-2\pi \tga\right\|_*,\ \left\|\mu-2\pi\sum_{i=1}^{N} \ga_i-2\pi \tga\right\|_{\F(\Omega_\ep)}\leq C\lep^{-2},
   \end{equation}   
   where, since the number $k$ is bounded independently of $\ep$, we have assumed it  is equal to some fixed integer $N$ by going to a subsequence, and where $\Omega_\ep$ is defined in \eqref{omeps}.  (Note that the $\log\lep$ factor in Theorem~\ref{teo:boundedvorticity}, (3) has been absorbed by using a different power for $\lep$ to obtain \eqref{badga}). 

   In particular, as a consequence of \eqref{badga} and \eqref{smallcurves} we have 
   \begin{equation}\label{flatbadpart}\|\tga\|_{*}\le \frac C{\lep^{4/3}},\quad \|\tga\|_{\F(\Omega)}\le \frac C{\lep^{4/3}} .\end{equation}

   From now on, we let
   \begin{equation}\label{lell}
      \ell \colonequals \frac1{({\log\lep})^2}.
   \end{equation}
   Note that the power $2$ in \eqref{lell}, in \eqref{jacgoodbad}, and in \eqref{omeps} is arbitrary, it could be any large  number. We consider coordinates in a neighborhood of $\ga_0$ as in Proposition~\ref{prop:diffeo}, the coordinate domain being $\Cd$. For carrying out the horizontal blow-up procedure, we need to work in a smaller neighbourhood of $\ga_0$. For convenience we use a cylinder in tube-coordinates. 
   
   Let 
   $$ \cyl_r \colonequals B(0,r)\times (0,\ellzero).$$
      We let $\chi_\ell$ be a cutoff function for the cylinder $\cyl_\ell$:  $\chi_\ell$ is equal to $1$ on $\cyl_{\ell/2}$ and equal to $0$ outside $\cyl_\ell$, its gradient is bounded by $C/\ell$.

   Then, from \eqref{goodga0} and  Lemma~\ref{curveintube}, every $\ga_i$ is included in a tubular neighborhood of $\ga_0$ with radius $C\lep^{-1/14}$,  hence $\chi_\ell\ga_i = \ga_i$. Thus, in view of \eqref{jacgoodbad}, we find that
   \begin{equation}\label{muintube}
      \|(1- \chi_\ell)\mu \|_{*},\ \|(1- \chi_\ell)\mu \|_{\F(\Omega_\ep)}\le C\frac{(\log\lep)^2}{\lep^{4/3}}
   \end{equation}
     and that the same bounds hold for $(\chi_\delta- \chi_\ell)\mu$, with a constant depending on $\delta$.

\item \emph{Lower bound for $\fperp(u,A)$.} 
   Inserting into the splitting formula  \eqref{Energy-Splitting} the definition  \eqref{defhc1}, the fact that 
   $h_\ex = H_{c1}^0 + K\log\lep = \frac\lep{2\R(\ga_0)} +K\log\lep,$
   and the minimality of $(\u,\A)$ which implies that $h_\ex^2J_0\ge GL_\ep(\u,\A)$, we find    \begin{equation}\label{minimality}
      o_\ep(1) \ge \fep(u,A)- \(\frac\lep{2\R(\ga_0)} +K\log\lep\)\pr{B_0,\mu}.
   \end{equation}
   But, again using \eqref{goodga0}--\eqref{jacgoodbad} and the definition \eqref{defratio}, we have
   \begin{equation*}
   	\pr{B_0,\mu} = 2\pi \ellzero N\R(\ga_0)+O(\lep^{-1/7}),\end{equation*}
   which together with \eqref{minimality} implies that
   \begin{equation}\label{feplog} \fep(u,A)\le \pi \ellzero N\lep + O(\lep^{6/7}).\end{equation}
   It also follows directly from \eqref{goodga0}--\eqref{jacgoodbad} that
   \begin{equation}\label{mustar} \|\mu - 2\pi N \ga_0\|_* = O(\lep^{-1/7}),\end{equation}
   but to apply Proposition~\ref{lowperp} on $\Cd$, where $\Cd$ is defined in Proposition~\ref{prop:diffeo}, we  need to check instead that the flat distance between $\mu$ and $2\pi N\ga_0$ tends to $0$ with $\ep$, which we can prove is true in $\Omega_\ep$  but not in $\Omega$.

   From \eqref{goodga0} and Lemma~\ref{curveintube} we find that each $\ga_i$  is included in a tubular neighborhood of $\ga_0$ with radius $C\lep^{-1/14}$, and that, in tube coordinates, its endpoints have vertical coordinate $0$ and $\ellzero$, respectively. Then, Lemma~\ref{flatstarlem} implies that 
   $$
   \|\ga_i-\ga_0\|_{\F(\Omega)}\leq C\lep^{-\frac1{14}}.
   $$
     Hence, combining with \eqref{flatbadpart} and \eqref{jacgoodbad}, we find
    that 
   \begin{equation}\label{muflat} \dl \colonequals \max\left\{ \|\mu - 2\pi N \ga_0\|_{\F(\Omega_\ep)},\lep^{-1} \right\}= O(\lep^{-\frac1{14}}).\end{equation}

   It follows from \eqref{muflat} and \eqref{feplog} that we may apply Proposition~\ref{lowperp} in a subdomain $\Cd^\ep$ of $\Cd$ obtained by stripping layers of height $\lep^{-2}$ at the top and bottom. We find that
   $$\fepperp (u,A)\ge \frac\lep 2\ \int_{\Cd^\ep}\chi_\delta\sqrt\gz\mu_{12}+\pi \ellzero N (N-1)\log\frac1{\dl} -C(1+\dl\log\lep).$$
 But, integrating by parts on each slice, by definition of $\mu$, we have
   \begin{multline*}\int_{\Cd\setminus\Cd^\ep}\chi_\delta\,  \sqrt\gz\,\mu_{12} = \int_{\Cd\setminus\Cd^\ep} d(\chi_\delta\,  \sqrt\gz)\wedge j(u,A)+\chi_\delta\,  \sqrt\gz\,dA\\ \le C\(\int_{\Cd\setminus\Cd^\ep}e_\ep(u,A)\)^{1/2} |\Cd\setminus\Cd^\ep|^{1/2} =o_\ep(1),\end{multline*}
   so that, using also \eqref{muintube}, we conclude that 
   \begin{equation}\label{lowfperp} \fepperp(u,A) \ge \frac\lep 2\ \prem{\chi_\ell\mu}+\pi \ellzero N (N-1)\log\frac1{\dl} -C(1+\dl\log\lep).\end{equation}

\item \emph{Lower bound for $\fep(u,A)-h_\ex \pr{B_0,\mu}$.}
   We apply one more time the curve construction of Theorem~\ref{theorem:epslevel}, this time on the cylinder $\cyl_\ell$ equipped with the metric $\tg$  defined as above by  $\tilde g_{ij} = \ell^{-2} g_{ij}$ if $1\le i,j\le2$ and $\tilde g_{ij} = g_{ij}$ otherwise. We find that there exists a polyhedral 1-current $\nu$ with no boundary relative to $\cyl_\ell$ such that
   \begin{equation}\label{ellcurves} \tfep(u,A)\ge \hal (\lep - C\log\lep)|\nu|_\tg - o_\ep(1),\quad  \|\mu -\nu\|_{*,\ell}\leq \frac{C}{\lep^q},\end{equation}
   where the $\|\cdot\|_{*,\ell}$ denotes the norm in the dual space $\(C_T^{0,1}(\cyl_\ell)\)^*$ and $q$ may be chosen arbitrarily large. Here $\tfep(u,A)$ is defined in \eqref{deftfep}, the integral could in fact be taken over $\cyl_\ell$ but we will not use this fact. 
   
   We also have
   \begin{equation}\label{ellflat} \|\mu -\nu\|_{\F(\cyl_{\ell,\ep})}\leq \frac{C}{\lep^q},\end{equation}
   where $\cyl_{\ell,\ep}$ is the set of points in $\cyl_\ell$ at distance at least $\lep^{-2}$ from the boundary.
   
   Note that, even though we cannot directly apply Theorem~\ref{theorem:epslevel} to the functional $\tilde F_\ep(u,A)$, since it involves a non-Euclidean metric, a straightforward modification of the proof in \cite{Rom} reveals that it holds in this case as well. Indeed the proof involves summing lower-bounds on an appropriate grid of cubes of side-length an arbitrarily large negative power of $\lep$. In our case, we can approximate the metric by a constant metric in each cube, which will thus be Euclidean after a linear change of coordinate. We can then obtain the desired energy lower bound and Jacobian estimate in each cube, the errors due to the non constant metric will be an arbitrarily large negative power of $\lep$.

   Note also that the lower bound really involves $\tep = \ep/\ell$, but this only introduces an error of order $|\log\ell|$ which is absorbed in the term $C\log\lep$. Also, $\|\cdot\|_{*,\ell}$ should be understood relative to the metric $\tg$, but it differs from the Euclidean version by at most  a factor $C\ell^2$ which does not alter the above bound considering that $q$ is arbitrary anyway.

   It follows from \eqref{ellcurves} and \eqref{mustar}, that
   \begin{equation}\label{ganoga}\|\nu - 2\pi N\ga_0\|_{*,\ell}\le\frac C{\lep^{1/7}}.\end{equation}
   In particular, using \eqref{smallcurves} we have that
   \begin{equation}\label{lowfornu} |\nu|_\tg\ge \frac1C.\end{equation}
   Now we recall from \eqref{decompose} the relation
   $$\ell^2\tfep(u,A)\le \ell^2\fperp(u,A) + \fpar(u,A).$$
   Therefore, multiplying \eqref{lowfperp} by $(1 - \ell^2)$, using the choice of \eqref{lell} and adding $\ell^2$ times \eqref{ellcurves}, we obtain that
   \begin{multline}\label{notyet}\fep(u,A) \ge\hal \lep \((1-\ell^2) \prem{\chi_\ell\mu} + \ell^2 |\nu|_{\tg} \)+\\ \pi (1 - \ell^2) \ellzero N (N-1)\log\frac1{\dl} +  O\(1+\ell^2\log\lep|\nu_\tg|\),\end{multline}
   where we have used the fact that $\dl\log\lep = o_\ep(1)$, in view of \eqref{muflat}.

   Moreover, from \eqref{muintube}, \eqref{mustar} we have in view of \eqref{defprem} that
   $$\prem{\chi_\ell\mu} = \pr{\chi_\ell\frac{\partial}{\partial z}, 2\pi N\ga_0} + O(\lep^{-1/7}) = 2\pi N \ellzero + O(\lep^{-1/7}).$$
   Inserting this into \eqref{notyet} and comparing with  \eqref{feplog},  we find that
   \begin{equation}\label{proxiell}|\nu|_{\tg} - 2\pi \ellzero N \le O(\ell^{-2}\lep^{-1/7}) =  O(\lep^{-1/8}).\end{equation}

   We then let, as in the proof of Lemma~\ref{curveintube},
   $$X = \chi_\ell\frac{\partial_z}{\sqrt\gz}.$$
   Note that $\tgz = \gz$. Then from \eqref{ganoga}, \eqref{lowfornu}, and \eqref{proxiell}, we have
   \begin{equation}\label{prx}\pr{X,\frac{\ga_0}{\ellzero}-\frac{\nu}{|\nu|_\tg}}
      = \pr{X,\frac{2\pi N\ga_0- \nu}{|\nu|_\tg}}+\pr{X,2\pi N\ga_0}\(\frac1{2\pi N\ellzero} -\frac1{|\nu|_\tg}\)
      \leq O(\lep^{-1/8}).\end{equation}
       Here we have used the fact that $|\ga_0|_\tg = \ellzero$, that $\pr{X,\ga_0} = \ellzero$ and that the left-hand side  is positive: indeed, since $\|X\|_\infty\le 1$ and since $X$ restricted to $\ga_0$  is precisely the unit tangent vector, it holds that 
   $$ \pr{X,\frac{\nu}{|\nu|_\tg}}\le 1 = \pr{X,\frac{\ga_0}{\ellzero}}.$$ 
  	
   Next,  we decompose $\nu$ as a sum of simple curves $\{\nu_i\}_{i\in I}$, so that
   $$ \pr{X,\frac{\ga_0}{\ellzero}-\frac{\nu}{|\nu|_\tg}} = \sum_{i\in I}\alpha_i \(1 - \pr{X,\frac{\nu_i}{|\nu_i|_\tg}}\)\colonequals\sum_{i\in I}\alpha_i\Delta_i,\quad \alpha_i = \frac{|\nu_i|_\tg}{|\nu|_\tg}.$$
   The $\Delta_i$'s are nonnegative. We let $I^\good$ denote those indices for which $\Delta_i <\lep^{-1/16}$ and we denote $\{\ga_i\}_{i\in I^\good}$ the corresponding curves. The rest of the indices is denoted $I^\bad$, and the sum of corresponding curves $\nu^\bad$, so that
   \begin{equation}\label{decomposenu}\nu = \sum_{i\in I^\good}2\pi \ga_i +\nu^\bad.\end{equation}
   Then \eqref{prx} implies that the sum of the coefficients $\alpha_i$ for $i$ ranging over $I^\bad$ is $O(\lep^{-1/16})$. Therefore, since the total length of $\nu$ is bounded, the total length of bad curves is $O(\lep^{-1/16})$ as well.

   As for the good curves, that we denote $\{\ga_i\}_{i\in I^\good}$, we have if $i\in I^\good$ that
   \begin{equation}\label{goodcurves} \Delta_i = \pr{X,\frac{\ga_0}{\ellzero}-\frac{\ga_i}{|\ga_i|_\tg}} \le \lep^{-1/16}.\end{equation}
   Since this is much smaller than $\ell^2$ defined in \eqref{lell}, we deduce from  Lemma~\ref{curveintube} that the good curves are included in a tube of radius $O(\lep^{-1/32})$ around $\ga_0$ and that each of their lengths is equal to $\ellzero + O(\lep^{-1/32})$. We thus have
   \begin{equation}\label{lengths}
      |\ga_i|_\tg = \ellzero + O(\lep^{-1/32}),\quad |\nu^\bad|_\tg = O(\lep^{-1/16}).
   \end{equation}
   In view of the  estimate \eqref{ganoga}  we deduce that there are exactly $N$ good curves, that we denote from now on $\ga_1,\dots,\ga_{N}$. We recall that from \eqref{decomposenu}, \eqref{lengths}, and  and \eqref{proxiell}, $|\nu|_\tg = 2\pi N \ellzero + o_\ep(1)$.

   Going back to \eqref{notyet} we now express  $\prem{\chi_\ell\mu}$ in terms of the curves $\ga_i$, using the vorticity estimate in \eqref{ellcurves} and \eqref{decomposenu}. Since $|\nu_\tg| = O(1)$, we find
   \begin{multline}\label{lowfep}      \fep(u,A) \ge \frac\lep 2 \((1-\ell^2)  \(\sum_{i=1}^{N} 2\pi\prem{ \ga_i} +\prem{\nu^\bad}\) + \ell^2 |\nu|_{\tg} \)+\\ \pi (1 - \ell^2) \ellzero N (N-1)\log\frac1{\dl} +  O(1),\end{multline}
   where we also used the fact that $\ell^2\log\lep = o_\ep(1)$. 
   
   Similarly, let us rewrite $h_\ex \pr{B_0,\mu}$.
   First we note that, from \eqref{muintube} and \eqref{ellcurves},
   $$h_\ex \pr{B_0,\mu} = h_\ex \pr{\chi_\ell B_0,\mu} + o_\ep(1) = h_\ex \pr{\chi_\ell B_0,\nu}+o_\ep(1).$$
   Then, from \eqref{ganoga} and \eqref{lengths} that
   $$\left\|\sum_{i=1}^N\ga_i-N\ga_0\right\|_*\le C\lep^{-1/8},$$ using again \eqref{smallcurves} to bound $\|\nu^\bad\|_*$. From the vorticity estimate in \eqref{ellcurves} and since $\chi_\ell B_0\in C_T^{0,1}(\cyl_\ell)$, using \eqref{muintube}, we deduce that
   \begin{multline}\label{magterm}
      h_\ex \pr{B_0,\mu} =  \(\frac\lep{2\R(\ga_0)} +K\log\lep\)\(\pr{\chi_\ell B_0,\nu^\bad} + \sum_{i=1}^{N}2\pi\pr{B_0,\ga_i}\) +o_\ep(1) \\
      =\frac{\pi\lep}{\R(\ga_0)}\sum_{i=1}^{N}\pr{B_0,\ga_i - \ga_0} + \pi \ellzero N \lep + 2\pi N K \pr{B_0,\ga_0}\log\lep \\
      + O(\lep\pr{\chi_\ell B_0,\nu^\bad}) + o_\ep(1).
   \end{multline}
   Then we substract off \eqref{magterm} from \eqref{lowfep}, noting that
   $$\ell^2 |\nu|_{\tg} = \ell^2\(|\nu^\bad|_\tg + 2\pi\sum_{i=1}^{N} |\ga_i|_\tg\),$$
   and that, since  $\prem{\nu^\bad}$ and $\pr{B_0,\nu^\bad}$
   are $O(|\nu^\bad|^2)$ --- which is a negative power of $\lep$ times $|\nu^\bad|$ in view of \eqref{smallcurves} --- the terms  $\pi \lep(1-\ell^2) \prem{\nu^\bad}$ and $\lep\pr{\chi_\ell B_0,\nu^\bad}$ may be absorbed in the term $\frac12 \lep\ell^2 |\nu^\bad|_{\tg}$.  Also, from \eqref{muflat} and \eqref{lell} we have that $\ell^2\log(1/\dl) = o_\ep(1).$ Notice that here, the maximization with $\lep^{-1}$ in the definition of $\rho$ in \eqref{muflat} plays a key role. We thus obtain (see Definition~\ref{defqell})
   \begin{multline}\label{lowfep2} \fep(u,A) - h_\ex\pr{B_0,\mu}\ge \pi \ellzero N(N-1)\log\frac1\dl  -2\pi K\pr{B_0,\ga_0} N \log\lep  + \\ + \pi\lep \sum_{i=1}^{N} \qell(\ga_i) + c\ell^2 \lep|\nu^\bad|_\tg + O(1),\end{multline}
   for some $c>0$.

\item \emph{Convergence of blown-up curves.} 
	We write $\ga_i$ as a piecewise graph over $\ga_0$ as above, letting $\ga_i(t) = \ga_0(z_i(t)) + \vu_i(t)$. From \eqref{goodcurves} and Lemma~\ref{curveintube} we have that $\|\vu_i\|_{L^\infty}\ll \ell$ for every $i$. Thus, Proposition~\ref{coercive} implies that
   \begin{equation}\label{qellinf}\qell(\ga_i)\ge c\|\vu_i\|_{L^\infty}^2.\end{equation}
   It also follows from \eqref{muintube}, \eqref{ellflat}, Lemma~\ref{flatstarlem} applied to each of the curves $\ga_1$,\dots,$\ga_{N}$ with $\tga=\ga_0$, and \eqref{smallcurves} applied to estimate $\|\nu^\bad\|_{\F(\cyl_{\ell,\ep})}$ that
   \begin{equation}\label{flatmax}\dl\le C\(\sum_i\|\vu_i\|_{L^\infty} + |\nu^\bad|^2_\tg\) + O((\log\lep)^2\lep^{-4/3})+ \lep^{-1}  .\end{equation}
 
   On the other hand, by minimality of $(\u,\A)$,  we deduce from the upper bound of   Theorem~\ref{thm:upperbound} and \eqref{Energy-Splitting} that
   \begin{equation}\label{upfep}\fep(u,A) - h_\ex\pr{B_0,\mu}\le  \frac\pi 2 \ellzero N(N-1) \log\lep-2\pi K\pr{B_0,\ga_0} N\log\lep + O(1).\end{equation}

   Let
   $$ Y = \lep\(\sum_{i=1}^N \qell(\ga_i) + \ell^2|\nu^\bad|_\tg\).$$
   From \eqref{qellinf}, \eqref{flatmax}, and the fact that we have $\ell^2\ge |\nu^\bad|_\tg^3$ in view of \eqref{lengths}, we get
   \begin{equation}
   \label{dlX}
   \dl^2 \lep \le C Y +o_\ep(1).\end{equation}   Combining \eqref{lowfep2} and \eqref{upfep}, in view of \eqref{dlX}, we find
  \begin{equation}\label{lowL}
    \frac\pi{2} \ellzero N(N-1)\log\frac1{CY+o_\ep(1)}  +  cY \le \pi \ellzero N(N-1) \log \frac{1}{\dl\sqrt\lep} +cY \le O(1).\end{equation}
      It follows that $Y = O(1)$, $\dl \le C \lep^{-1/2}$,  and then that for every $1\le i\le N$, in view of \eqref{qellinf},
   \begin{equation}\label{sqrtlepest} \|\vu_i\|_{L^\infty}\le\frac C{\sqrt\lep},\quad \qell(\ga_i)\le \frac C{\lep}.
   \end{equation}
   It also follows that
   \begin{equation}\label{improvedbad}|\nu^\bad|_\tg\le C\frac{(\log\lep)^4}\lep.\end{equation}
If $N > 1$ then \eqref{lowL} implies in addition that $\dl$ is bounded below by $c\lep^{-1/2}$ and thus, in view of \eqref{improvedbad}, from \eqref{flatmax} we deduce that    \begin{equation}\label{lowinfty} \max_i\|\vu_i\|_{L^\infty}\ge \frac c{\sqrt\lep}.\end{equation}

   Recalling \eqref{decomposenu}, the vorticity estimates in \eqref{flatstar} follow from \eqref{muintube}, the vorticity estimate in \eqref{ellcurves}, \eqref{ellflat}, and \eqref{smallcurves} together with \eqref{improvedbad} and \eqref{decomposenu}.

   We denote $\tga_i$ the curve $\ga_i$ blown up horizontally by a factor $\sqrt{\frac{h_\ex}N}=\sqrt{\frac\lep {2\R_0 N}}+o_\ep(1)$, so that $\tga_i(t) = \ga_0(z_i(t))+\sqrt{\frac{h_\ex}N}\vu_i(t)$. From \eqref{sqrtlepest} and Proposition~\ref{coercive}, there exists a subsequence $\{\ep\}$ tending to zero such that $\tga_i$ converges, for any $i = 1,\dots,N$, as $\ep\to 0$, uniformly and as currents, to an $H^1$-graph $\ga_i^*(z) = \ga_0(z)+\vu_i^*(z)$. Notice that if $\sqrt{\frac{h_\ex}N}\|\vu_i\|_\infty\to 0$, then $\tga_i$ converges to $\ga_0$. Moreover, from \eqref{lowinfty}, if $N>1$, then $\max\limits_{1\leq i \leq N}\|\vu_i^*\|_{L^\infty} >0$.

\item \emph{Improved lower bound.}
   We return to bounding from below $\fep(u,A,\Ud)$ using \eqref{decompose2}. As above we denote by $\nu$ the polyhedral 1-current obtained by applying Theorem~\ref{teo:boundedvorticity} on the cylinder $\cyl_\ell$ equipped with the metric $\gell$. It decomposes as $\nu^\bad+\nu^\good$, where $\nu^\good =2\pi \sum_{i=1}^N\ga_i.$ 
   
   From \eqref{ellcurves}, \eqref{decomposenu}, and \eqref{improvedbad} we find
   \begin{equation}\label{ellcurves2}\tilde\fep(u,A)\ge\pi\sum_{i=1}^N |\ga_i|_{\gell} \lep + O(\log\lep).\end{equation}
      Also, we may slice $\nu$ (resp. $\nu^\good$) by the coordinate function $z$ defined on $U_\delta$. This provides a family of $0$-currents $\{\nu^z\}_z$ (resp. $\{(\nu^\good)^z\}_z$), where $z$ belongs to a set of full measure in $(0,\ellzero)$. Both $\nu^z$ and $(\nu^\good)^z$ are sums of Dirac masses with weights belonging to $2\pi\mathbb Z$ for almost every $z$.
      
   From \eqref{ellflat}, we know that 
   $$\|\chi_\ell(\mu - \nu)\|_{\F(\Omega_\ep)}\le C\ell^{-1} \lep^{-q}.$$
   Moreover $\chi_\ell \nu^\good = \nu^\good$ since, from the previous step (see \eqref{sqrtlepest}), we know that each $\ga_i$ is included in a tube of radius $O(1/\sqrt\lep)$ around $\ga_0$. Thus, in view of \eqref{muintube}, and using \eqref{smallcurves} together with \eqref{improvedbad}  to estimate $\|\nu^\bad\|_{\F(\Omega_\ep)}$, we find that 
   $$\|\mu - \nu^\good \|_{\F(\Omega_\ep)}\le C \frac{\left(\log\lep\right)^2}{\lep^{\frac43}},$$
   and then that 
   $$\int_{z=\lep^{-2}}^{\ellzero-\lep^{-2}}\|\mu^z(u,A) - (\nu^\good)^z\|_{\F(\Sigma^z)} dz\le \|\mu(u,A) - \nu^\good \|_{\F(\Omega_\ep)}\le  C \frac{\left(\log\lep\right)^2}{\lep^{\frac43}}.$$
   
   For any $\ep>0$ we define
   $$\T_\ep=\left\{z\in[\lep^{-2},\ellzero-\lep^{-2}]\mid \|\mu^z - (\nu^\good)^z\|_{\F(\Sigma^z)}\le \lep^{-\frac76}\right\},\quad \T_\ep^c = [0,\ellzero]\setminus \T_\ep.$$
   It follows from the above that 
   \begin{equation}\label{measurebadsetofslices}
   |\T_\ep^c|\le C(\log\lep)^2\lep^{-\frac16}.
   \end{equation}
  
   Let $z\in \T_\ep$. We claim that, for any $\ep>0$ small enough (depending on $z$), we have
   \begin{multline}\label{fzmino} \int_{\Sigma^z} e_\ep^\perp(u,A)\,d\vol_\gperp - \frac{\lep}2\int_{\Sigma_z}\chi_\ell  \sqrt{\gz} \mu_{12} d\vec{u} -\pi N^2 \log \delta  +\pi N(N-1)\log\sqrt{\frac{N}{h_\ex}} \\ \ge - \pi \sum_{i\neq j} \log|\vu_i^*(z) -\vu_j^*(z)|_{g_\bullet} + N\gamma  +o_\ep(1)+ o_\delta(1),
   \end{multline} 	
   if $\vu_i^*(z) \neq \vu_j^*(z)$ for all $i\neq j$, whereas if $\vu_i^*(z) = \vu_j^*(z)$ for some $i\neq j$, then
   $$
   \liminf_{\ep \to 0} \int_{\Sigma^z} e_\ep^\perp(u,A)\,d\vol_\gperp - \frac{\lep}2\int_{\Sigma_z}\chi_\ell  \sqrt{\gz} \mu_{12} d\vec{u} +\pi N(N-1)\log\sqrt{\frac{N}{h_\ex}}=+\infty.
   $$
   To prove the claim, we consider three cases: 
   \begin{enumerate}[label=\textsc{\bf Case \arabic*.},leftmargin=0pt,labelsep=*,itemindent=*,itemsep=10pt,topsep=10pt]
   \item\label{case1} If $\int_{\Sigma^z} e_\ep^\perp(u,A)\,d\vol_\gperp> \lep^3$, then, by integration by parts, we have
   $$
   \int_{\Sigma_z}\chi_\ell  \sqrt{\gz} \mu_{12} d\vec{u}=\int_{\Sigma_z}d(\chi_\ell \sqrt{\gz}) \wedge j(u,A)+\chi_\ell \sqrt{\gz}dA\leq C\ell\left(\int_{\Sigma^z} e_\ep^\perp(u,A)\,d\vol_\gperp\right)^\frac12,
   $$ 
   and therefore 
   $$
   \int_{\Sigma^z} e_\ep^\perp(u,A)\,d\vol_\gperp - \frac{\lep}2\int_{\Sigma_z}\chi_\ell  \sqrt{\gz} \mu_{12} d\vec{u}> \lep^3-C\frac{\lep^\frac52}{(\log\lep)^2},
   $$
   which implies the claim.
   
   \item
   If $\pi(N+m)\lep \leq \int_{\Sigma^z} e_\ep^\perp(u,A)\,d\vol_\gperp \leq \lep^3$ for some $m>0$, we can apply Lemma~\ref{2Dlbvar}, which provides the existence of points $a_1,\dots,a_k$ and integers $d_1,\dots,d_k$ such that 
   \begin{equation}\label{lboundcase2}
   \int_{\Sigma^z} e_\ep^\perp(u,A)\,d\vol_\gperp \geq \pi \sum_i |d_i|\sqrt{\gz(a_i)}\left(\lep -C\log\lep \right),
   \end{equation}
   and 
   $$
   \left\|2\pi \sum_i d_i \delta_{a_i}-\mu_{12} d\vec{u}\right\|_{\F(\Sigma^z)}\leq C\lep^{-3}.
   $$
   In addition, from the definition of $\T_\ep$, we deduce that
   $$
   \left\|2\pi \sum_i d_i \delta_{a_i}-(\nu^\good)^z\right\|_{\F(\Sigma^z)}\leq C\lep^{-\frac76}.
   $$
   In particular, we have that
   \begin{equation}\label{vortestcase2}
   \int_{\Sigma_z}\chi_\ell  \sqrt{\gz} \mu_{12} d\vec{u}=\int_{\Sigma^z}\chi_\ell\sqrt{\gz}(\nu^\good)^z+O(\lep^{-\frac76})=2\pi N+O(\lep^{-\frac12}),
   \end{equation}
   where in the last equality we used the fact that $g_{33}$ is equal to $1$ on the axis $z=0$ and that all good curves are contained in a tubular neighborhood of radius $O(\lep^{-\frac12})$ around $\ga_0$.
     
   Hence, combining \eqref{lboundcase2} with \eqref{vortestcase2}, we find
   \begin{multline*}
   \int_{\Sigma^z} e_\ep^\perp(u,A)\,d\vol_\gperp - \frac{\lep}2\int_{\Sigma_z}\chi_\ell  \sqrt{\gz} \mu_{12} d\vec{u}\\ \geq \pi \sum_i \sqrt{\gz(a_i)}|d_i|\lep -\pi N\lep +O(\lep^{\frac12}).
   \end{multline*}
   If $\sum_{i}|d_i|>N$, the claim immediately follows. On the other hand, if $\sum_{i}|d_i| \leq N$, then by combining $\pi(N+m)\lep \leq \int_{\Sigma^z} e_\ep^\perp(u,A)\,d\vol_\gperp$ with \eqref{vortestcase2}, we find
   $$
   \int_{\Sigma^z} e_\ep^\perp(u,A)\,d\vol_\gperp - \frac{\lep}2\int_{\Sigma_z}\chi_\ell  \sqrt{\gz} \mu_{12} d\vec{u} \geq \pi m \lep+O(\lep^{\frac12}).
   $$
   Once again, since $m>0$, the claim follows.
   
   \item If $\int_{\Sigma^z} e_\ep^\perp(u,A)\,d\vol_\gperp < \pi(N+m)\lep$,  
since all good curves are contained in a tubular neighborhood of radius $O(1/\sqrt\lep)$ around $\ga_0$, we deduce that
	$$
	\left\|2\pi N\delta_0 -(\nu^\good)^z\right\|_{\F(\Sigma^z)}\leq C\lep^{-\frac12},
	$$
	(that we also used in \eqref{vortestcase2}) which together with the definition of $\T_\ep$ gives
	$$
	\dl^z=\max\left\{\left\|\mu^z-2\pi N\delta_0\right\|_{\F(\Sigma^z)},\lep^{-1}\right\}\leq C\lep^{-\frac12}.
	$$

   We can therefore apply Proposition~\ref{lowerprecise} (choosing $m$ sufficiently small), which provides the existence of $N$ distinct points $a_1,\dots,a_N$, such that
   \begin{multline}\label{case3_1}
   \int_{\Sigma_z}e_\ep^\perp(u,A)\,d\vol_\gperp \geq \pi \sum_{i=1}^N \sqrt{\gz(a_i)}\lep \\
   +\pi N^2\log\de -\pi\sum_{i\neq j}\log \dist_{g^\perp}(a_i-a_j)+N\gamma +o_\ep(1)+o_\de(1)
   \end{multline}
   and 
   \begin{equation}\label{vorticityestimatecase3}
   \left\|2\pi \sum_{i=1}^N \delta_{a_i}-\mu_{12} d\vec{u}\right\|_{\F(\Sigma^z)}\leq C\lep^{-s},
	\end{equation}
	where $s>0$ can be chosen arbitrarily large. In addition, from the definition of $\T_\ep$, we deduce that
   \begin{equation}\label{vortestflatgoodz}
   \left\|2\pi \sum_{i=1}^N \delta_{a_i}-(\nu^\good)^z\right\|_{\F(\Sigma^z)}\leq C\lep^{-\frac76}.
   \end{equation}
   
   From \eqref{vortestflatgoodz} we deduce that, for any $i=1,\dots,N$, there exists a point, that we denote $\ga_{i,\ep}(z)$, belonging to $\ga_{i,\ep}$, such that
   $$
   \sum_{i=1}^N \distg(a_i,\ga_{i,\ep}(z))\leq \|\mu^z-(\nu^\good)^z\|_{\F(\Sigma^z)}\leq C\lep^{-\frac76}.
   $$
   From this and the fact that the blown up good curves $\tilde \ga_{i,\ep}$ converge uniformly to $\ga_{i}^*$, letting $\tilde a_i$ denote the point $a_i$ blown up horizontally by a factor $\sqrt{\frac{h_\ex}N}$, we deduce that $\tilde a_i$ converges to $\tilde \ga_i^*(z)$ as $\ep \to 0$. 
   
   Finally, we note that 
   \begin{equation}\label{case3.2}
   -\pi\sum_{i\neq j}\log \dist_{g^\perp}(a_i-a_j)=- \pi \sum_{i\neq j} \log|\tilde a_i-\tilde a_j|_{g_\bullet}-\pi N(N-1)\log\sqrt\frac{N}{h_\ex}+o_\ep(1),
   \end{equation}
   and that, when passing to the limit $\ep\to 0$ and using the lower semi-continuity of $-\log$, we have
   $$
   \liminf_{\ep \to 0} - \pi \sum_{i\neq j} \log|\tilde a_i-\tilde a_j|_{g_\bullet}=- \pi \sum_{i\neq j} \log|\vu_i^*(z) -\vu_j^*(z)|_{g_\bullet}
   $$
   if $\vu_i^*(z) \neq \vu_j^*(z)$ for all $i\neq j$, whereas if $\vu_i^*(z) = \vu_j^*(z)$ for some $i\neq j$, then
   $$
   \liminf_{\ep \to 0} - \pi \sum_{i\neq j} \log|\tilde a_i-\tilde a_j|_{g_\bullet}=+\infty.
   $$
   The claim \eqref{fzmino} thus follows from combining this with \eqref{case3_1}, \eqref{case3.2} and \eqref{vorticityestimatecase3}.
   \end{enumerate}

  We next integrate over $z\in [0, \ellzero]$. We claim that the contribution from $\T_\ep^c$ is bounded below by  $o_\ep(1)$, i.e. that
   \begin{multline}\label{integrationoverz3}
    	\int_{z\in [0,\ellzero]}
       \int_{\Sigma^z} e_\ep^\perp(u,A)\,d\vol_\gperp -  \hal\int_{\Sigma^z} \chi_\ell \sqrt{\gz} \mu_{12} d\vec{u}  \
       \log \frac1\ep
       \\ \ge \int_{z\in \T_\ep} \int_{\Sigma^z} e_\ep^\perp(u,A)\,d\vol_\gperp - \hal\int_{\Sigma^z} \chi_\ell \sqrt{\gz} \mu_{12} d\vec{u}  \   \log \frac1\ep+o_\ep(1).
       \end{multline}

    Let $z\in (\T_\ep)^c$. We consider two cases. First, if $\int_{\Sigma^z} e_\ep^\perp(u,A)\,d\vol_\gperp> \lep^3$, then, arguing exactly as in \ref{case1}, we find that
    \begin{equation}\label{estimatecasebad1}
   	\int_{\Sigma^z} e_\ep^\perp(u,A)\,d\vol_\gperp - \frac{\lep}2\int_{\Sigma_z}\chi_\ell  \sqrt{\gz} \mu_{12} d\vec{u}> \lep^3-C\frac{\lep^\frac52}{(\log\lep)^2}\geq 0.
   	\end{equation}
   	Second, if $\int_{\Sigma^z} e_\ep^\perp(u,A)\,d\vol_\gperp \leq \lep^3$, we can apply once again Lemma~\ref{2Dlbvar}, which provides the existence of points $a_1,\dots,a_k$ and integers $d_1,\dots,d_k$ such that
   	\begin{equation}\label{lboundcase2.2}
   		\int_{\Sigma^z} e_\ep^\perp(u,A)\,d\vol_\gperp \geq \pi \sum_i |d_i|\sqrt{\gz(a_i)}\left(\lep -C\log\lep \right)
   	\end{equation}
   	and 
   	$$
   	\left\|2\pi \sum_i d_i \delta_{a_i}-\mu_{12} d\vec{u}\right\|_{\F(\Sigma^z)}\leq C\lep^{-3}.
   	$$
	
	 In particular, we have that
	\begin{equation}\label{vortestcase2.2}
		\int_{\Sigma_z}\chi_\ell  \sqrt{\gz} \mu_{12} d\vec{u}=2\pi\sum_i d_i \chi_\ell(a_i)\sqrt{\gz(a_i)}+O(\lep^{-3}).
	\end{equation}
	Combining \eqref{lboundcase2.2} with \eqref{vortestcase2.2}, we find
	\begin{multline}\label{estimatecasebad2}
		\int_{\Sigma^z} e_\ep^\perp(u,A)\,d\vol_\gperp - \frac{\lep}2\int_{\Sigma_z}\chi_\ell  \sqrt{\gz} \mu_{12} d\vec{u}\\ \geq \pi \sum_i \sqrt{\gz(a_i)}\left(|d_i|-d_i\chi(a_i)\right)+O(\log\lep)\geq O(\log\lep).
	\end{multline}
	Hence, from \eqref{estimatecasebad1}, \eqref{estimatecasebad2}, and \eqref{measurebadsetofslices}, we deduce that
	\begin{equation*}
	\int_{z\in (\T_\ep)^c} \int_{\Sigma^z} e_\ep^\perp(u,A)\,d\vol_\gperp -   \hal \int_{(\T_\ep)^c} \int_{\Sigma^z} \chi_\ell \sqrt{\gz} \mu_{12} d\vec{u}  \  \log \frac1\ep\geq |(\T_\ep)^c|O(\log\lep)=o_\ep(1),
	\end{equation*}
	which proves \eqref{integrationoverz3}.

      Thus, by definition \eqref{defprem}, in view of \eqref{muintube},  and the estimate \eqref{fzmino}, we have
   \begin{multline}\label{finalFpre}
      \int_{z\in [0,\ellzero]}
      \int_{\Sigma^z} e_\ep^\perp(u,A)\,d\vol_\gperp - \hal \langle\mu^{2D}\rangle
      \log \frac1\ep  -\pi N^2\ellzero \log \delta +\pi N(N-1)\log\sqrt{\frac{N}{h_\ex}}
      \\ \ge
      \int_{z\in \T_\ep} \(-\pi \sum_{i\neq j} \log
      |\vu_i^*(z) -\vu_j^*(z)|_{g_\bullet}+ N\gamma \) dz + o_\delta(1) +o_\ep(1)+C_\delta o_\ep(1).
   \end{multline}
 
   Adding \eqref{finalFpre} times $(1-\ell^2)$ to  \eqref{ellcurves2} times $\ell^2$  we obtain, in view of \eqref{decompose2}
   \begin{multline}\label{942}
      \fep(u,A,\Ud)\ge \pi \log \frac1{\ep} \((1-\ell^2)  \(\sum_{i=1}^{N} \prem{\ga_i} +\prem{\nu^\bad}\) + \ell^2 \sum_{i=1}^N |\ga_i|_{\gell}\)\\+\pi N^2\ellzero \log \delta -\pi (1 - \ell^2) \ellzero N (N-1)\log\sqrt{\frac{N}{h_\ex}} \\+ \int_{z\in \T_\ep} \(-\pi \sum_{i\neq j} \log
      |\vu_i^*(z) -\vu_j^*(z)|_{g_\bullet}+ N\gamma \) dz +C_\delta o_\ep(1)+ o_\delta(1) +o_\ep(1),\end{multline}
	where we used again the fact that $\ell^2\log\lep=o_\ep(1)$.

   On the other hand, using \eqref{magterm}, \eqref{improvedbad} and \eqref{smallcurves} again, we have
   \begin{equation*}
      h_\ex \pr{B_0,\mu} =\frac{\pi\lep}{\R(\ga_0)}\sum_{i=1}^{N}\pr{B_0,\ga_i - \ga_0} + \pi \ellzero N \lep +  2\pi N K\langle B_0, \ga_0\rangle \log\lep+ o_\ep(1).
   \end{equation*}
   Subtracting from \eqref{942} we find, as in \eqref{lowfep2}
   \begin{multline}\label{944} F_\ep (u, A, \Ud) -    h_\ex \pr{B_0,\mu}  \ge
      \frac\pi{2} (1-\ell^2) \ellzero N(N-1)\log\frac{h_\ex}N\\ -2\pi KN \pr{B_0,\ga_0}\log\lep
      +\pi \ellzero N^2  \log \delta
      +  \pi\lep \sum_{i=1}^{N} \qell(\ga_i)\\
      + \int_{z\in \T_\ep} \(-\pi \sum_{i\neq j} \log
      |\vu_i^*(z) -\vu_j^*(z)|_{g_\bullet}+ N\gamma \) dz +C_\de o_\ep(1)+ o_\delta(1) +o_\ep(1),\end{multline}
   Using the coercivity of $Q_\ell$ from Proposition \ref{coercive} and the upper bound of Theorem~\ref{thm:upperbound},
   we deduce that
   $$ \int_{z\in \T_\ep} \(-\pi \sum_{i\neq j} \log
      |\vu_i^*(z) -\vu_j^*(z)|_{g_\bullet}+ N\gamma \) dz\le C$$ with $C>0$ independent of $\ep$ (but depending on $\delta$).

   We may extract a subsequence $\{\ep_k \}_k$ such that $\sum_{k } |(\mathcal T_{\ep_k})^c|<\infty$. From the upper bound above and the monotone convergence theorem it follows that
   \begin{multline*} \lim_{k\to \infty} \int_{z\in \T_{\ep_k} } \(-\pi \sum_{i\neq j} \log
      |\vu_i^*(z) -\vu_j^*(z)|_{g_\bullet}+ N\gamma \) dz\\
      =\int_{z\in [0, \ellzero]} \(-\pi \sum_{i\neq j} \log
      |\vu_i^*(z) -\vu_j^*(z)|_{g_\bullet}+ N\gamma \) dz\end{multline*} and the right-hand side is a convergent integral.

   Using Fatou's lemma, dominated convergence theorem, and \eqref{limq}, by taking the limit in \eqref{944}, we are led to
   \begin{multline}\label{945}\liminf_{\ep \to 0} \(  F_\ep (u, A, \Ud) -    h_\ex \pr{B_0,\mu}  -
     \frac \pi{2} \ellzero N(N-1)\log\frac{h_\ex}{N} -2\pi KN\pr{B_0,\ga_0} \log\lep\)
      \\ \ge \pi N^2 \ellzero \log \delta + \pi \ellzero N \sum_{i=1}^{N} Q(\vu_i^*)
      + \int_{z\in [0,\ellzero]} \(-\pi \sum_{i\neq j} \log
      |\vu_i^*(z) -\vu_j^*(z)|_{g_\bullet}+ N\gamma \) dz + o_\delta(1) .\end{multline}
   There remains to  bound from below $\fep(u,A,\Omega\setminus U_\delta)$. We can assume without loss of generality that $A$ is divergence free in $\RR^3$. From 
   \eqref{Energy-Splitting}, \eqref{945}, and the upper bound of Theorem~\ref{thm:upperbound}  we find
   $\fep(u,A,\Omega\setminus \Ud)\le C$ independent of $N$, for every $\delta$. Thus, taking a subsequence if necessary, $\nab_{A} u$ is bounded in $L^2(\Omega \backslash \Ud)$ and $A$ is bounded in $H^1(\RR^3\setminus \Ud)$. Hence,
   $j(u, A)$ converges weakly to some $j^*$ in $L^2(\Omega \backslash \Ud)$ and $A$ converges weakly to some $A^*$ in $H^1(\RR^3\backslash \Ud)$.

   Since $(\u,\A)$ is a minimizer of $GL_\ep$, it (weakly) satisfies the Ginzburg--Landau equations \eqref{GLeq}. In particular,
   $$
   \curl (\curl A -H_\ex)=(i\u,\nabla_{\A}\u)\chi_\Omega \quad \mathrm{in}\ \RR^3.
   $$
   On the other hand, since $(u,A)= (e^{-ih_\ex\phi_0}\u,\A-h_\ex A_0 )$, we deduce that $(\u,\A)$ is gauge equivalent to $(u,A+h_\ex\curl B_0)$ (recall that $A_0=\nabla \phi_0+\curl B_0$ in $\Omega$), and therefore, it holds (weakly )in $\RR^3$ that
   \begin{align*}
   \curl \big(\curl(A+h_\ex \curl B_0)-H_\ex \big)&=(iu,\nabla_{A+h_\ex \curl B_0}u)\chi_\Omega\\
   &=\big(j(u,A)-h_\ex \curl B_0 |u|^2\big)\chi_\Omega.
\end{align*}
   But $A_0$ (weakly) solves
   $$
   \curl (h_\ex \curl A_0-H_\ex)=-\curl B_0\chi_\Omega \quad \mathrm{in}\ \RR^3.
   $$
   Hence,
   $$
   -\Delta A= \curl\curl A=\big(j(u,A)+h_\ex \curl B_0 (1-|u|^2)\big)\chi_\Omega\quad \mathrm{in}\ \RR^3.
   $$
   Since $h_\ex\curl B_0 (1-|u|^2)$ strongly converges to $0$ in $L^2(\Omega)$, by passing to the limit in the previous equation we find
   $$
   -\Delta A^*=j^*\chi_\Omega \quad \mbox{in }\RR^3.
   $$
   Moreover, from \eqref{mustar}, we find $\mu(u,A) \to 2\pi N \Gamma_0$. Hence, in the sense of distributions, we have
   $$
   \curl(j^*+A^*)=2\pi N \ga_0\quad \mbox{in }\Omega.
   $$
   Finally, from Proposition~\ref{jA} combined with the lower semicontinuity of the energy and \eqref{comega}, we deduce that
   $$\liminf_{\ep \to 0} \(\fep(u, A, \Omega\setminus \Ud) +\frac12\int_{\RR^3\setminus \Ud}|\curl A|^2\) \ge N^2 C_\Omega +\pi N^2 \ellzero \log \frac1\delta + o_\delta(1).$$
   Combining with \eqref{945} and \eqref{Energy-Splitting}, in view of the definition of $\R_0$ \eqref{R0}, we have proved \eqref{lbw}.

\end{enumerate}
\end{proof}

\begin{proof}[Proof of Theorem~\ref{thmain}]
Theorem \ref{thm:upperbound} shows that for minimizers there is equality in \eqref{lbw}, and therefore \eqref{lbw0} holds.
To conclude the proof of Theorem \ref{thmain}, there remains to optimize over $N$ integer.
From \eqref{lbw0}, we know that the minimal energy of a solution with $N$ vortex filaments, assuming $N$ is independent of $\ep$, is given by $g_\ep(N)+o(1)$, where 
$$
g_\ep(N)=f_\ep(N)+\min W_N+\gamma \ellzero N
$$
with 
$$
f_\ep(N)=h_\ex^2 J_0+\pi \ellzero N\lep -2\pi N \ellzero \R_0 h_\ex \\
+\pi \ellzero N(N-1)\log \sqrt{\frac{h_\ex}{N}}+N^2C_\Omega.
$$
This has exactly the same form as the minimal energy of a solution with $N$ vortex points in 2D; see \cite{SanSerBook}*{Chapter 12}. Hence, 
 the optimization is identical to that in \cite{SanSerBook}*{Lemma 12.1 and Theorem 12.1} and yields that  if $\he\in (H_N-o_\ep(1), H_{N+1}+o_\ep(1)) $, then $N$ is the optimal number of curves.
\end{proof}

\section{Upper bound}\label{sec:upperbound}
We now present our upper bound for the 3D Ginzburg--Landau functional.
\subsection{Statement of the main result}

\begin{theorem}\label{thm:upperbound}
Assume $\Omega$ is a smooth bounded domain such that the maximum of the ratio is achieved at a smooth simple curve $\ga_0$ which can be extended to a smooth simple closed curve in $\RR^3$, still denoted $\ga_0$. 

For any $\ep>0$, assume  $\he= H_{c_1}^0 + K\log \lep$ with $K$ bounded independently of $\ep$,  and let $N$ be an integer independent of $\ep$.

Define tube coordinates $(x,z)\in\cyl_\delta$ in a neighborhood of $\ga_0$ using Proposition~\ref{prop:diffeo}.
Assume that for each $\ep$, the curves $\ga_{1,\ep},\dots,\ga_{N,\ep}$ are defined in these coordinates by
\begin{equation}\label{gaiep} \ga_{i,\ep}(z) = \ga_0(z)+\sqrt{\frac{N}{h_\ex}}\vu_i(z),\end{equation}
where $\vu_i:[0,\ellzero]\to\RR^2$ is smooth and independent of $\ep$.

Then, for any $\ep$ sufficiently small, there exists a configuration $(\u_\ep,\A_\ep)$ such that
\begin{multline}\label{upperboundestimate}
	GL_\ep(\u_\ep,\A_\ep)
	\leq  h_\ex^2 J_0
	+\frac\pi2 \ellzero  N(N-1)\log h_\ex -2\pi K  \R_0 \ellzero N  \log\lep\\
	-\frac\pi2 \ellzero N(N-1)\log N+W_N(\Gamma_1,\dots,\Gamma_N)\\
	+\gamma \ellzero N +N^2C_\Omega +o_\de(1)+C_\de o_\ep(1)+o_\ep(1),
\end{multline}
where $\ga_i(z)=\ga_0(z)+\vu_i(z)$ and $W_N$ is defined in \eqref{defW}.

\end{theorem}

The upper bound is computed using the velocity field  given by the Biot--Savart law (see Definition \ref{defbs})  associated to a collection of $N$ vortex filaments nearly parallel and close to $\ga_0$, as $\ep\to 0$. Outside of a fixed but small tube around $\ga_0$, this velocity field will coincide up to a small error (as $\ep\to 0$) with the field associated to $\ga_0$ by Proposition \ref{jA}. However we must estimate the energy from our construction in the tube, for which we need a simple enough approximation to our velocity field. The rest of this section is devoted to the proof of Theorem~\ref{thm:upperbound}.

\subsection{Definition of the test configuration} We let $\u_\ep=e^{ih_\ex\phi_0}u_\ep$ and $\A_\ep=h_\ex A_0+A_\ep$, where $(e^{ih_\ex\phi_0},h_\ex A_0)$ is the approximate Meissner state. In order to define $(u_\ep,A_\ep)$, we proceed as follows.  
First, we let $r_\ep\colonequals\lep^{-3}$. We then define 
$$
\rho_{i,\ep}(x)=
\left\{
\begin{array}{cl}
	\dfrac1{\degone \left(\frac{r_\ep}\ep\right)}\degone \(\dfrac{\dist(x,\ga_{i,\ep})}\ep \) & \mbox{if }\dist(x,\ga_{i,\ep})\leq r_\ep\\
	1&\mbox{otherwise}.
\end{array}
\right.
$$
where hereafter $\degone$ denotes the modulus of the (unique nonconstant) degree-one radial vortex solution $u_0$, see for instance   \cite{SanSerBook}*{Proposition 3.11}. It is important to recall that, as $R\to \infty$, $\degone(R)\to 1$ and 
\begin{equation}\label{asymptoticsf}
	\frac12 \int_0^R \(|\degone'|^2+\frac{\degone^2}{r^2}+\frac{(1-\degone^2)^2}2 \)rdr=\frac{1}{2\pi}\left(\pi \log R+\gamma+o(1)\right),
\end{equation}
where $\gamma>0$ is still the   fixed constant from \cite{BetBreHel}.

We also define
$$
\rho_\ep(x)\colonequals \min_{i\in \{1,\dots,N\}} \rho_{i,\ep}(x).
$$

On the other hand, we let $\varphi_\ep$ be defined by the relation 
\begin{equation}\label{derivativephi}
	\nabla \varphi_\ep=\sum_{i=1}^N X_{\ga_{i,\ep}}+\nabla f_{\ga_{i,\ep}},
\end{equation}
where $X_{\ga_{i,\ep}}$ and $f_{\ga_{i,\ep}}$ are defined by applying Proposition~\ref{jA}\footnote{To be precise, $f_{\ga_{i,\ep}}$ is defined in the proof of the proposition.} with $\ga=\ga_{i,\ep}$. 

Let us remark that since $\curl \sum_{i=1}^N X_{\ga_{i,\ep}}=2\pi \sum_{i=1}^n \ga_{i,\ep}$, if $\sigma$ denotes a smooth, simple, and closed curve that does not intersect any of the curves $\ga_{i,\ep}$, then, by Stokes' theorem, 
$$
\int_\sigma \left(\sum_{i=1}^N X_{\ga_{i,\ep}}+\nabla f_{\ga_{i,\ep}}\right)=2\pi m,\quad \mbox{for some }m\in \mathbb{Z}.
$$
This ensures that $\varphi_\ep$ is a well-defined function modulo $2\pi$.

We finally let
$$
u_\ep(x)\colonequals \rho_\ep(x) e^{i\varphi_\ep(x)}
$$
and
$$
A_\ep(x)=\sum_{i=1}^N A_{\ga_{i,\ep}}(x),
$$
where, once again, $A_{\ga_{i,\ep}}$ is defined by applying Proposition~\ref{jA} with $\ga=\ga_{i,\ep}$.

\subsection{The difference between the covariant gradient and the gradient is negligible in \texorpdfstring{$T_\delta(\ga_0)$}{T(Γ\textzeroinferior)}}
Hereafter, $T_\delta(\ga_0)$ denotes the tube defined in Proposition \ref{prop:diffeo}. A straightforward computation shows that
$$
|\nabla_{A_\ep}u_\ep|^2-|\nabla u_\ep|^2=\rho_\ep^2\left(|A_\ep|^2-2\nabla \varphi_\ep\cdot A_\ep\right).
$$
Let $p<2$ and $q>2$ be such that  $\frac1p+\frac1q=1$. From H\"older's inequality, we deduce that
$$
\left| \int_{T_\delta(\ga_0)} |\nabla_{A_\ep}u_\ep|^2-|\nabla u_\ep|^2\right|\leq \|A_\ep\|^2_{L^2(T_\delta(\ga_0))}+ C\|\nabla \varphi_\ep\|_{L^p(T_\delta(\ga_0))}\|A_\ep\|_{L^q(T_\delta(\ga_0))}.
$$
Since $A_\ep\in W^{2,\frac32}(\Omega)$, from Sobolev embedding it follows that $A_\ep\in L^r(\Omega)$ for any $r\geq 1$. Moreover, $\nabla \varphi_\ep \in L^r(\Omega)$ for any $r<2$, which follows from Proposition \ref{jA}. Hence, from H\"older's inequality, we deduce that
\begin{equation}\label{covariantderivativeenergy}
	\left| \int_{T_\delta(\ga_0)} |\nabla_{A_\ep}u_\ep|^2-|\nabla u_\ep|^2\right|\leq C\|A_\ep\|_{L^q(T_\de(\ga_0))}\leq  C |T_\delta(\ga_0)|^{\frac1{2q}}\|A_\ep\|_{L^{2q}(T_\de(\ga_0))}\leq C\delta^\frac{1}{q}
\end{equation}
for any $q>2$. The constant $C$ depends on $q$ and blows up as $q\to 2$ and as $q\to \infty$. By fixing its value, we ensure that the RHS is $o_\delta(1)$.

\subsection{Energy estimate in small tubes around the curves}
We first smoothly extend $\ga_{i,\ep}$ to a smooth simple closed curve in $\RR^3$, which we denote $\tilde \ga_{i,\ep}$.  The extension is supported in the complement of $\Omega$, without intersecting its boundary. We have considerable freedom in choosing this extension, except that $\tilde \ga_{i,\ep}$ must intersect $\partial \Omega$ transversally, so that we can apply Proposition \ref{jA} with $\ga=\tilde\ga_{i,\ep}$. We then consider the tube $\tilde T_{r_\ep}(\ga_{i,\ep})$ of radius $r_\ep$ around $\tilde \ga_{i,\ep}$, and its restriction to $\Omega$, that is, $T_{r_\ep}(\ga_{i,\ep})\colonequals \left\{x\in \Omega \ : \ \dist(x,\ga_{i,\ep})<r_\ep \right\}$. We claim that
\begin{equation}\label{energyestimatesmalltube}
	\frac12\int_{T_{r_\ep}(\ga_{i,\ep})}|\nabla u_\ep|^2+\frac{1}{2\ep^2}(1-|u_\ep|^2)^2\leq \pi|\ga_{i,\ep}|\log \frac{r_\ep}{\ep}+\gamma |\ga_{i,\ep}|+o_\ep(1),
\end{equation}
where $\gamma$ is the constant that appears in \eqref{asymptoticsf}.

To prove this we proceed as follows. First, given a point in $\tilde T_{r_\ep}(\ga_{i,\ep})$ we denote by $p_{\ga_{i,\ep}}$ its nearest point on $\ga_{i,\ep}$. We then define
$$
D_z\colonequals \left\{p\in \tilde T_{r_\ep}(\ga_{i,\ep}) \ |\ p_{\ga_{i,\ep}}=\ga_i(z) \right\}.
$$
Observe that $D_z$ is a disk in $\RR^2$ centered at $\ga_{i,\ep}(z)$ with radius $r_\ep$. We now appeal to \eqref{approxBS}, which yields that
$$
h_{i,\ep}(p)\colonequals X_{\ga_{i,\ep}}(p) - Y_{\ga_{i,\ep}}(p)\in L^q\left(\tilde T_{r_\ep}(\ga_{i,\ep})\right)
$$
for any $q\geq 1$, where
$$
Y_{\ga_{i,\ep}}(p)\colonequals \frac{p_{\ga_{i,\ep}} - p}{|p_{\ga_{i,\ep}} - p|^2}\times \tau_{\ga_{i,\ep}}(p_{\ga_{i,\ep}}).
$$
Here, $\tau_{\ga_{i,\ep}}(p_{\ga_{i,\ep}})$ denotes the tangent vector to $\ga_{i,\ep}$ at $p_{\ga_{i,\ep}}$.

Once again, from the proof of Proposition \ref{jA} we know that $f_{\ga_{i,\ep}}\in W^{1,q}(\Omega)$ for any $q<4$ and $i=1,\dots,N$. Then, recalling \eqref{derivativephi} and using H\"older's inequality, we deduce that
\begin{multline}\label{difphiY}
	\int_{T_{r_\ep}(\ga_{i,\ep})}\left| \nabla \varphi_\ep -\sum_{j=1}^N Y_{\ga_{j,\ep}} \right |^2 =\int_{T_{r_\ep}(\ga_{i,\ep})}\left|\sum_{j=1}^N h_{j,\ep}+\nabla f_{\ga_{j,\ep}} \right|^2
	\\
	\leq |T_{r_\ep}(\ga_{i,\ep})|^\frac13 \left\|\sum_{j=1}^N h_{j,\ep}+\nabla f_{\ga_{j,\ep}}\right\|_{L^3(T_{r_\ep}(\ga_{i,\ep}))}^2\leq C r_\ep^\frac23.
\end{multline}
In addition, note that for a.e. $p\in T_{r_\ep}(\ga_{i,\ep})$ and $j\neq i$, we have (recall \eqref{gaiep})
$$
\frac{|Y_j(p)|}{|Y_i(p)|}\leq \frac{|p_{\ga_{i,\ep}} - p|}{|p_{\ga_{j,\ep}} - p|}\leq C \frac{r_\ep}{\sqrt\lep}.
$$
We then deduce that 
\begin{multline}\label{sumY}
	\int_{T_{r_\ep}(\ga_{i,\ep})} \rho_\ep^2 \left|\sum_{j=1}^N Y_{\ga_{j,\ep}} \right |^2 =\int_{T_{r_\ep}(\ga_{i,\ep})} \rho_\ep^2|Y_{\ga_{i,\ep}}|^2\left(1+O\left(\frac{r_\ep}{\sqrt\lep}\right)\right)^2
	\\
	=\left(1+O\left(\frac{r_\ep}{\sqrt\lep}\right)\right)\int_{T_{r_\ep}(\ga_{i,\ep})} \rho_\ep^2|Y_{\ga_{i,\ep}}|^2.
\end{multline}
Using \eqref{difphiY} and \eqref{sumY}, we then deduce that
\begin{multline}\label{energysmalltube1}
	\frac12 \int_{T_{r_\ep}(\ga_{i,\ep})}\left(|\nabla u_\ep|^2+\frac1{2\ep^2}(1-|u_\ep|^2)^2\right)
	=\frac12\int_{T_{r_\ep}(\ga_{i,\ep})}\left(|\nabla \rho_\ep|^2+\rho_\ep^2|\nabla \varphi_\ep|^2+\frac1{2\ep^2}(1-\rho_\ep^2)^2\right)
	\\
	=\left(\frac12+O\left(\frac{r_\ep}{\sqrt\lep}\right)\right)\int_{T_{r_\ep}(\ga_{i,\ep})}\left(|\nabla \rho_\ep|^2+\rho_\ep^2|Y_{\ga_{i,\ep}}|^2+\frac1{2\ep^2}(1-\rho_\ep^2)^2\right)+O\left(r_\ep^\frac23\right).
\end{multline}
On the other hand, by change of coordinates, we have
\begin{multline}\label{energysmalltube2}
	\int_{T_{r_\ep}(\ga_{i,\ep})}\left(|\nabla \rho_\ep|^2+\rho_\ep^2|Y_{\ga_{i,\ep}}|^2+\frac1{2\ep^2}(1-\rho_\ep^2)^2\right)\\
	=\int_0^{|\ga_{i,\ep}|} \left( \int_{D_z \cap \Omega }\left(|\nabla \rho_\ep|^2+\rho_\ep^2|Y_{\ga_{i,\ep}}|^2+\frac1{2\ep^2}(1-\rho_\ep^2)^2\right)|\mathrm{Jac}(\phi)|dxdy\right)dz
	\\
	\leq \int_0^{|\ga_{i,\ep}|} \left( \int_{D_z}\left(|\nabla \rho_\ep|^2+\rho_\ep^2|Y_{\ga_{i,\ep}}|^2+\frac1{2\ep^2}(1-\rho_\ep^2)^2\right)|\mathrm{Jac}(\phi)|dxdy\right)dz,
\end{multline}
where $\phi:D_z\times (0,|\ga_{i,\ep}|)\to \RR^3$ is such that 
$$
\phi\left(\ga_{i,\ep}(z)+(x,y,z)\right)=\ga_{i,\ep}(z)+xv_1(z)+yv_2(z)
$$
with $(v_1(z),v_2(z))$ orthonormal and such that $v_1(z),v_2(z)$ are perpendicular to $\ga'_{i,\ep}$ and smooth with respect to $z$. Notice that here we used the fact that $D_z=\ga_{i,\ep}(z)+D(0,r_\ep)$, where $D(0,r_\ep)$ denotes the disk in $\RR^2$ centered at $0$ and with radius $r_\ep$. Observe that
$$
|\mathrm{Jac}(\phi)|=|\mathrm{det}(\ga_{i,\ep}'+xv_1'(z)+yv_2'(z),v_1,v_2)|=1+O(r_\ep).
$$
Hence, by combining this with \eqref{energysmalltube1} and \eqref{energysmalltube2}, we deduce that
\begin{multline}\label{energysmalltube3}
	\frac12 \int_{T_{r_\ep}(\ga_{i,\ep})}\left(|\nabla u_\ep|^2+\frac1{2\ep^2}(1-|u_\ep|^2)^2\right)+O\left(r_\ep^\frac23\right)
	\\
	\leq \left(\frac12+O\left(\frac{r_\ep}{\sqrt\lep}\right)\right)\int_0^{|\ga_{i,\ep}|} \left( \int_{D_z}\left(|\nabla \rho_\ep|^2+\rho_\ep^2|Y_{\ga_{i,\ep}}|^2+\frac1{2\ep^2}(1-\rho_\ep^2)^2\right)dxdy\right)dz.
\end{multline}
To compute the integral over $D_z$ we use polar coordinates centered at $\ga_{i,\ep}(z)$. Letting $r$ denote the distance to this point and $\theta$ the polar angle, we have
$$
|\nabla \rho_\ep(x)|=\frac{\left|\degone'\left(\frac{r}{\ep}\right)\right|}{\ep \degone\left(\frac{r_\ep}\ep\right)}|\nabla r|=\frac{\left|\degone'\left(\frac{r}{\ep}\right)\right|}{\ep \degone\left(\frac{r_\ep}\ep\right)} \quad \mbox{and}\quad |Y_{\ga_{i,\ep}}|=|\nabla \theta|=\frac{1}{r}.
$$
It follows that
\begin{multline*}
	\int_{D_z}\left(|\nabla \rho_\ep|^2+\rho_\ep^2|Y_{\ga_{i,\ep}}|^2+\frac1{2\ep^2}(1-\rho_\ep^2)^2\right)dxdy\\
	=2\pi\int_0^{r_\ep}\( \frac{\degone'\left(\frac{r}{\ep}\right)^2}{\ep^2\degone\left(\frac{r_\ep}\ep\right)^2}+\frac{\degone\left(\frac{r}{\ep}\right)^2}{r^2\degone\left(\frac{r_\ep}\ep\right)^2}+\frac1{2\ep^2}\left(1-\frac{\degone\left(\frac{r}{\ep}\right)^2}{\degone\left(\frac{r_\ep}\ep\right)^2}\right)^2\)rdr\\
	=2\pi \int_0^{\frac{r_\ep}{\ep}}\( \frac{\degone'(s)^2}{\degone\left(\frac{r_\ep}\ep\right)^2}+\frac{\degone(s)^2}{s^2\degone\left(\frac{r_\ep}\ep\right)^2}+\frac1{2}\left(1-\frac{\degone(s)^2}{\degone\left(\frac{r_\ep}\ep\right)^2}\right)^2\)sds,
\end{multline*}
where in the last equality we used the change of variables $r=\ep s$. Finally, using that $\lim_{\ep \to 0}\degone\left(\frac{r_\ep}\ep\right)=1$, since $\lim_{\ep \to 0}\frac{r_\ep}{\ep}=+\infty$, from \eqref{asymptoticsf} we obtain 
$$
\int_{D_z}\left(|\nabla \rho_\ep|^2+\rho_\ep^2|Y_{\ga_{i,\ep}}|^2+\frac1{2\ep^2}(1-\rho_\ep^2)^2\right)dxdy=2\left(\pi \log \frac{r_\ep}{\ep}+\gamma +o_\ep(1)\right).
$$
Inserting this in \eqref{energysmalltube3}, we obtain \eqref{energyestimatesmalltube}.

\subsection{Energy estimate in the perforated tube}
We consider the perforated tube $T_{r_\ep}^\delta\colonequals T_\delta(\ga_0)\setminus \bigcup_{i=1}^N T_{r_\ep}(\ga_{i,\ep})$. In this region we have $\rho_\ep\equiv 1$, and therefore
$$
E_\ep(u_\ep,T_{r_\ep}^\delta)\colonequals \frac12 \int_{T_{r_\ep}^\delta}\left(|\nabla u_\ep|^2+\frac1{2\ep^2}(1-|u_\ep|^2)^2\right)=\frac12\int_{T_{r_\ep}^\delta}|\nabla \varphi_\ep|^2.
$$
Arguing as when obtaining \eqref{difphiY}, we find
$$
\int_{T_{r_\ep}^\delta}\left| \nabla \varphi_\ep -\sum_{j=1}^N Y_{\ga_{j,\ep}} \right |^2 =\int_{T_{r_\ep}^\delta}\left|\sum_{j=1}^N h_{j,\ep}+\nabla f_{\ga_{j,\ep}} \right|^2
\leq |T_{r_\ep}^\delta|^\frac13 \left\|\sum_{j=1}^N h_{j,\ep}+\nabla f_{\ga_{j,\ep}}\right\|_{L^3({T_{r_\ep}^\delta})}^2\leq C \delta^\frac23,
$$
which yields
\begin{equation}\label{energyperforatedtube}
	E_\ep(u_\ep,T_{r_\ep}^\delta)=\frac12\int_{T_{r_\ep}^\delta}\left|\sum_{j=1}^N Y_{\ga_{j,\ep}} \right|^2+o_\delta(1)=\frac12 \int_0^{\ellzero}\int_{\Sigma_z\cap T_{r_\ep}^\delta} \left|\sum_{j=1}^N Y_{\ga_{j,\ep}} \right|^2\sqrt{g_{33}}d\vol_\gperp+o_\delta(1),
\end{equation}
where in the last equality we used the coordinates we defined in Section \ref{sec:coordinates}.

In order to estimate the integral on the RHS, we proceed in several steps. 

\begin{enumerate}[label=\textsc{\bf Step \arabic*.},leftmargin=0pt,labelsep=*,itemindent=*,itemsep=10pt,topsep=10pt]
	\item \emph{From $g^\perp$ to the Euclidean metric.}	
	Given $z\in [0,\ellzero]$, we define $\Pi:\Sigma_z\to \langle \ga_0'(z)\rangle ^\perp$ be the orthogonal projection of $\Sigma_z$ onto the perpendicular plane to $\ga_0'(z)$. Observe that $D\Pi(\ga_0(z))$ is the identity map from $\langle \ga_0'(z)\rangle ^\perp$ to itself, and, therefore, $\Pi^{-1}$ is well defined (for any sufficiently small $\de$) and such that $D\Pi^{-1}(0)$ is the identity map from $\Sigma_z$ to itself. Moreover
	\begin{equation}\label{estimatederivativePhi}
		\|D\Pi^{-1}(0)-D\Pi^{-1}(x)\|_\infty\leq C|x|
	\end{equation}
	for any $\delta$ sufficiently small. In particular, for any $x,y\in \langle \ga_0'(z)\rangle ^\perp$, we have
	\begin{equation}\label{expansionPhi}
		\Pi^{-1}(x)-\Pi^{-1}(y)=x-y+O\left(|x-y|(|x|+|y|)\right).
	\end{equation}
	Let us now observe that, since $g_{33}=1$ on $\ga_0$, for any $p\in \Sigma_z$ we have
	$$
	g_{33}(p)=1+O(\dist(p,\ga_0)).
	$$
	In particular, if $p=\Pi^{-1}(x)$, using \eqref{expansionPhi}, we obtain
	\begin{equation}\label{estimateg33}
		g_{33}(\Pi^{-1}(x))=1+O(\dist(\Pi^{-1}(x),\ga_0))=1+O(|x|).
	\end{equation}
	Moreover, using again \eqref{expansionPhi}, we deduce that
	\begin{equation}\label{Phi*}
		(\Pi^{-1})^*(g^\perp)(x)=\mbox{Euclidean metric}+O(|x|),
	\end{equation}
	where $(\Pi^{-1})^*(g^\perp)$ denotes the pullback of $g^\perp$ by $\Pi^{-1}$.
	
	\item \emph{Projection on $\ga_{i,\ep}$.} Let $p\in \Sigma_z$. Recall that $p_{\ga_{i,\ep}}$ denotes the projection of $p$ on $\ga_{i,\ep}$. Noting that $\ga_{i,\ep}(z)=\ga_{i,\ep}\cap \Sigma_z$, we also define $d_i=|p-p_{\ga_{i,\ep}}|$, $d_{i,z}=|p-\ga_{i,\ep}(z)|$. 
	
	Notice that we can write $p_{\ga_{i,\ep}}=\ga_{i,\ep}(z+\eta_i)$ for some $\eta_i \in \RR$, which we now proceed to estimate. From Taylor's expansion we have
	\begin{equation}\label{taylorga_i1}
		p_{\ga_{i,\ep}}-\ga_{i,\ep}(z)=\eta_i \ga_{i,\ep}'(z)+O(\eta_i^2)
	\end{equation}
	and
	\begin{equation}\label{taylorga_i2}
		p_{\ga_{i,\ep}}-\ga_{i,\ep}(z)=\eta_i \ga'_{i,\ep}(z+\eta_i)+O(\eta_i^2).
	\end{equation}
	which, in particular, implies that 
	\begin{equation}\label{estimatep_ia_i}
		|p_{\ga_{i,\ep}}-\ga_{i,\ep}(z)|^2=\eta_i^2+O(\eta_i^3).
	\end{equation} 
	On the other hand, by definition of $p_{\ga_{i,\ep}}$, we directly have
	\begin{equation}\label{orthogon}
		(p-p_{\ga_{i,\ep}})\cdot \ga_{i,\ep}'(z+\eta_i)=0.
	\end{equation}
	Moreover, we also have 
	\begin{equation}\label{estimatepa_i}
		(p-\ga_{i,\ep}(z))\cdot \ga_{i,\ep}'(z)=O\left(\frac{d_{i,z}}{\sqrt{\lep}}+d_{i,z}^2\right).
	\end{equation}
	Indeed, this follows from the following facts. First, if we let $\nu$ denote the normal vector to $\Sigma_z$ at $\ga_{i,\ep}(z)$ (with $\Sigma_z$ oriented according to $\ga_{i,\ep}'(z)$), we have that
	$$
	\ga_{i,\ep}'(z)=\ga_0'(z)+O\left(\frac{1}{\sqrt\lep} \right)=\nu+O\left(\frac{1}{\sqrt\lep} \right).
	$$
	Second, since $\Sigma_z$ is smooth and $p,\ga_{i,\ep}(z)\in \Sigma_z$, we have
	$$
	(p-\ga_{i,\ep}(z))\cdot \nu =O(|p-\ga_{i,\ep}(z)|^2)=O(d_{i,z}^2).
	$$
	Therefore
	$$
	(p-\ga_{i,\ep}(z))\cdot \ga_{i,\ep}'(z)=(p-\ga_{i,\ep}(z))\cdot \nu +O\left(\frac{|p-\ga_{i,\ep}(z)|}{\sqrt{\lep}}\right)=O\left(\frac{d_{i,z}}{\sqrt{\lep}}+d_{i,z}^2\right).
	$$

	Observe that
	\begin{equation}\label{expd_i^2}
		d_i^2=|p-\ga_{i,\ep}(z)+\ga_{i,\ep}(z)-p_{\ga_{i,\ep}}|^2=d_{i,z}^2+|p_{\ga_{i,\ep}}-\ga_{i,\ep}(z)|^2-2(p-\ga_{i,\ep}(z))\cdot (p_{\ga_{i,\ep}}-\ga_{i,\ep}(z)).
	\end{equation}
	and
	\begin{equation}\label{expd^2}
		d_{i,z}^2=|p-p_{\ga_{i,\ep}}+p_{\ga_{i,\ep}}-\ga_{i,\ep}(z)|^2=d_i^2+|p_{\ga_{i,\ep}}-\ga_{i,\ep}(z)|^2+2(p-p_{\ga_{i,\ep}})\cdot (p_{\ga_{i,\ep}}-\ga_{i,\ep}(z)).
	\end{equation}
	Using \eqref{orthogon}, we can write
	\begin{align*}
		(p-p_{\ga_{i,\ep}})\cdot (p_{\ga_{i,\ep}}-\ga_{i,\ep}(z))&=(p-p_{\ga_{i,\ep}})\cdot \left(p_{\ga_{i,\ep}}-\ga_{i,\ep}(z)-\eta_i\ga'_{i,\ep}(z+\eta_i)+\eta_i\ga'_{i,\ep}(z+\eta_i)\right)\\
		&=(p-p_{\ga_{i,\ep}})\cdot \left(p_{\ga_{i,\ep}}-\ga_{i,\ep}(z)-\eta_i\ga'_{i,\ep}(z+\eta_i)\right),
	\end{align*}
	which combined with \eqref{taylorga_i2} yields 
	\begin{equation}\label{estimatedotpp_i}
		(p-p_{\ga_{i,\ep}})\cdot (p_{\ga_{i,\ep}}-\ga_{i,\ep}(z))=O(\eta_i^2d_i).
	\end{equation}
	On the other hand, we can write
	\begin{align*}
		(p-\ga_{i,\ep}(z))&\cdot (p_{\ga_{i,\ep}}-\ga_{i,\ep}(z))\\
		&=(p-\ga_{i,\ep}(z))\cdot \left(p_{\ga_{i,\ep}}-\ga_{i,\ep}(z)-\eta_i\ga'_{i,\ep}(z)+\eta_i\ga'_{i,\ep}(z)\right)\\
		&=(p-\ga_{i,\ep}(z))\cdot \left(p_{\ga_{i,\ep}}-\ga_{i,\ep}(z)-\eta_i\ga'_{i,\ep}(z)\right)+\eta_i(p-\ga_{i,\ep}(z))\cdot \ga'_{i,\ep}(z),
	\end{align*}
	which combined with using \eqref{estimatepa_i} and \eqref{taylorga_i1} yields 
	\begin{equation*}
		(p-\ga_{i,\ep}(z))\cdot (p_{\ga_{i,\ep}}-\ga_{i,\ep}(z))=O(\eta_i^2d_{i,z})+\eta_i O\left(\frac{d_{i,z}}{\sqrt{\lep}}+d_{i,z}^2\right).
	\end{equation*}
	Finally, by adding up \eqref{expd_i^2} with \eqref{expd^2}, and using \eqref{estimatep_ia_i}, \eqref{estimatedotpp_i}, and \eqref{estimatepa_i}, we deduce that
	\begin{equation}\label{estimateeta}
		\eta_i=O\left(\frac{d_{i,z}}{\sqrt{\lep}}+d_{i,z}^2\right).
	\end{equation}
	
	\item \emph{Estimating $Y_{\ga_{i,\ep}}$ on $\Sigma_z\cap T_{r_\ep}^\delta.$}
	Let $p\in \Sigma_z$. From \eqref{estimateeta}, we know that
	$$
	\left|p_{\ga_{i,\ep}}-\ga_{i,\ep}(z)\right|=O\left(\frac{|p-\ga_{i,\ep}(z)|}{\sqrt{\lep}}+|p-\ga_{i,\ep}(z)|^2\right)
	$$
	and therefore
	\begin{equation}\label{difpip}
		p_{\ga_{i,\ep}}-p=\ga_{i,\ep}(z)-p+O\left(\frac{|p-\ga_{i,\ep}(z)|}{\sqrt{\lep}}+|p-\ga_{i,\ep}(z)|^2\right).
	\end{equation}
	We now define $x=\Pi(p)$ and $x_{i,\ep}^z=\Pi(\ga_{i,\ep}(z))$. From \eqref{expansionPhi}, since $|x_{i,\ep}^z|=O\left(\frac1{\sqrt\lep}\right)$, we have
	$$
	\ga_{i,\ep}(z)-p=x_{i,\ep}^z-x+O\left(\left(|x|+\frac1{\sqrt\lep}\right)|x_{i,\ep}^z-x|\right),
	$$
	which combined with \eqref{difpip} yields
	\begin{align}\label{difpip2}
		p_{\ga_{i,\ep}}-p&=x_{i,\ep}^z-x+O\left(\frac{|x_{i,\ep}^z-x|}{\sqrt\lep}+|x_{i,\ep}^z-x|^2\right)+O\left(|x_{i,\ep}^z-x|\left(|x|+\frac1{\sqrt\lep}\right)\right)\notag\\
		&=x_{i,\ep}^z-x+O\left(|x_{i,\ep}^z-x|^2 + |x_{i,\ep}^z-x||x|+\frac{|x_{i,\ep}^z-x|}{\sqrt\lep}\right).
	\end{align}
	In particular, we have that 
	\begin{equation}\label{estga_i-p}
		|\ga_{i,\ep}(z)-p|=O(|p_{\ga_{i,\ep}}-p|)=O(|x_{i,\ep}^z-x|).
	\end{equation}
	In addition, from \eqref{difpip2} we find
	\begin{align}\label{estga_i-p2}
		\frac1{|p_{\ga_{i,\ep}}-p|^2}&=\frac1{|x_{i,\ep}^z-x|^2+O\left(|x_{i,\ep}^z-x|^3 + |x_{i,\ep}^z-x|^2|x|+\frac{|x_{i,\ep}^z-x|^2}{\sqrt\lep}\right)}\notag\\
		&=\frac1{|x_{i,\ep}^z-x|^2}\frac1{1+O\left(|x_{i,\ep}^z-x| + |x|+\frac{1}{\sqrt\lep}\right)}\notag\\
		&=\frac1{|x_{i,\ep}^z-x|^2}\left(1+O\left(|x_{i,\ep}^z-x| + |x|+\frac{1}{\sqrt\lep}\right)\right).
	\end{align}
	
	On the other hand, from Taylor's expansion, we deduce that
	$$
	\ga_{i,\ep}'(z+\eta_i)=\ga_{i,\ep}'(z)+O(|p_{\ga_{i,\ep}}-\ga_{i,\ep}(z)|)=\ga_0'(z)+O\left(\frac1{\sqrt\lep}\right)+O(|p_{\ga_{i,\ep}}-\ga_{i,\ep}(z)|).
	$$
	Therefore, using \eqref{estimateeta} and \eqref{estga_i-p}, we find
	$$
	\tau_{\ga_{i,\ep}}(p_{\ga_{i,\ep}})=\ga_{i,\ep}'(z+\eta_i)=\ga_0'(z)+O\left(\frac{1}{\sqrt\lep}+|x_{i,\ep}^z-x|^2\right).
	$$
	Finally, from this, \eqref{difpip2}, and \eqref{estga_i-p2}, we find
	\begin{equation}\label{estimateY_i}
		Y_{\ga_{i,\ep}}(p)=\frac{p_{\ga_{i,\ep}}-p}{|p_{\ga_{i,\ep}}-p|^2}\times \tau_{\ga_{i,\ep}}(p_{\ga_{i,\ep}})=\frac{(x_{i,\ep}^z-x)^\perp}{|x_{i,\ep}^z-x|^2}+O\left(1+\frac{1}{|x_{i,\ep}^z-x|\sqrt\lep}+\frac{|x|}{|x_{i,\ep}^z-x|}\right),
	\end{equation}
	where we recall that $x_{i,\ep}^z-x=\Pi(\ga_{i,\ep}(z))-\Pi(p)$.
	
	\item \emph{Combining the previous steps.}
	Let $z\in [0,\ellzero]$. By change of coordinates, we have
	\begin{multline*}
	\int_{\Sigma_z\cap T_{r_\ep}^\delta} \left|\sum_{j=1}^N Y_{\ga_{j,\ep}} \right|^2\sqrt{g_{33}}d\vol_\gperp
	\\= \int_{\Pi\left(\Sigma_z\cap T_{r_\ep}^\delta\right)} \left|\sum_{j=1}^N Y_{\ga_{j,\ep}}\left(\Pi^{-1}(x)\right) \right|^2\sqrt{g_{33}\left(\Pi^{-1}(x)\right)}d(\Pi^{-1})^*(\vol_\gperp)(x). 
	\end{multline*}
	From \eqref{Phi*}, we find
	$$
	d(\Pi^{-1})^*(\vol_\gperp)(x)=(1+O(|x|))dx.
	$$
	On the other hand, from \eqref{difpip2}, we deduce that 
	$$
	D\left(x_{i,\ep}^z,r_\ep(1-o_\ep(1))\right)\subset \Pi(\Sigma_z\cap T_{r_\ep}(\ga_{i,\ep}))
	$$
	and 
	$$
	\Pi(\Sigma_z\cap T_{\de}(\ga_0)) \subset D\left(0,\de(1+o_\delta(1))\right).
	$$
	In particular,
	$$
	\Pi\left(\Sigma_z \cap T_{r_\ep}^\de\right)\subset A_{r_\ep}^\de(z)\colonequals D\left(0,\de(1+o_\delta(1))\right)\setminus \cup_{i=1}^N  D\left(x_{i,\ep}^z,r_\ep(1-o_\ep(1))\right).
	$$
	From the previous estimates and \eqref{estimateg33}, we then deduce that
	\begin{equation}\label{integralY_jperforated}
		\int_{\Sigma_z\cap T_{r_\ep}^\delta} \left|\sum_{j=1}^N Y_{\ga_{j,\ep}} \right|^2\sqrt{g_{33}}d\vol_\gperp\leq 
		\int_{A_{r_\ep}^\de(z)} \left|\sum_{j=1}^N Y_{\ga_{j,\ep}}\left(\Pi^{-1}(x)\right) \right|^2(1+O(|x|))dx.
	\end{equation}
	
	On the other hand, for $x\in A_{r_\ep}^\de(z)$, using \eqref{estimateY_i}, we have
	\begin{align}\label{errortermY0}
		\Bigg|\sum_{j=1}^N  Y_{\ga_{j,\ep}} & \left(\Pi^{-1}(x)\right) \Bigg|^2(1+O(|x|))\notag
		\\
		&=\left |\sum_{j=1}^N \frac{(x_{j,\ep}^z-x)^\perp}{|x_{j,\ep}^z-x|^2}+O\left(1+\frac{1}{|x_{j,\ep}^z-x|\sqrt\lep}+\frac{|x|}{|x_{j,\ep}^z-x|}\right)\right|^2(1+O(|x|))\notag\\
		&=\left |\sum_{j=1}^N \frac{(x_{j,\ep}^z-x)^\perp}{|x_{j,\ep}^z-x|^2}\right|^2+
		\sum_{j=1}^N O\left(\frac{1}{|x_{j,\ep}^z-x|}+\frac{1}{|x_{j,\ep}^z-x|^2\sqrt\lep}+\frac{|x|}{|x_{j,\ep}^z-x|^2}+1\right).
	\end{align}
	
	Let us now observe that
	\begin{equation}\label{errortermY1}
		\sum_{j=1}^N\int_{A_{r_\ep}^\de(z)}\left(\frac1{|x_{j,\ep}^z-x|}+1\right)dx=O(\de^2),
	\end{equation}
	which directly follows from the integrability of the integrand in compact subets of $\RR^2$.
	
	On the other hand, from a direct computation, recalling that $r_\ep=\lep^{-q}$ for some $q>0$, it follows that
	$$
	\sum_{j=1}^N \int_{A_{r_\ep}^\de(z)}\frac1{|x_{j,\ep}^z-x|^2}dx=O(\log\lep),
	$$
	and therefore
	\begin{equation}\label{errortermY2}
		\frac{1}{\sqrt\lep}\sum_{j=1}^N \int_{A_{r_\ep}^\de(z)}\frac1{|x_{j,\ep}^z-x|^2}dx=o_\ep(1).
	\end{equation}
	
	In addition, we have
	$$
	\sum_{j=1}^N \int_{A_{r_\ep}^\de(z)} \frac{|x|}{|x_{j,\ep}^z-x|^2}dx=
	\sum_{j=1}^N \left(\int_{A_{r_\ep}^\de(z)} \frac{1}{|x_{j,\ep}^z-x|}dx+
	\int_{A_{r_\ep}^\de(z)} \frac{|x_{j,\ep}^z|}{|x_{j,\ep}^z-x|^2}dx\right).
	$$
	Since $|x_{j,\ep}^z|=O\left(\frac1{\sqrt\lep} \right)$, from \eqref{errortermY1} and \eqref{errortermY2}, we find
	\begin{equation}\label{errortermY3}
		\sum_{j=1}^N \int_{A_{r_\ep}^\de(z)} \frac{|x|}{|x_{j,\ep}^z-x|^2}dx=O(\de^2)+o_\ep(1).
	\end{equation}
	By using \eqref{errortermY0}, \eqref{errortermY1}, \eqref{errortermY2}, and \eqref{errortermY3}, from \eqref{integralY_jperforated} it follows that
	\begin{equation}\label{integralY_jperforated2}
		\int_{\Sigma_z\cap T_{r_\ep}^\delta} \left|\sum_{j=1}^N Y_{\ga_{j,\ep}} \right|^2\sqrt{g_{33}}d\vol_\gperp\leq 
		\int_{A_{r_\ep}^\de(z)} \left |\sum_{j=1}^N \frac{(x_{j,\ep}^z-x)^\perp}{|x_{j,\ep}^z-x|^2}\right|^2 dx+ O(\de^2)+o_\ep(1).
	\end{equation}
	
	\item \emph{Renormalized energy.} We claim that
	\begin{align}\label{renormalized1}
		\frac12 \int_{A_{r_\ep}^\de(z)}\Bigg | \sum_{j=1}^N & \frac{(x_{j,\ep}^z-x)^\perp}{|x_{j,\ep}^z-x|^2} \Bigg|^2 dx \notag \\
		&=-\pi \sum_{\substack{i,j=1\\i\neq j}}^N \log |x_{i,\ep}^z-x_{j,\ep}^z|+\pi N \log \frac1{r_\ep}+\pi N^2\log \de +o_\de(1)+C_\delta o_\ep(1)+o_\ep(1).
	\end{align}
	To prove this, we consider 
	$$
	\Phi_{\ep}^z(x)=-\sum_{i=1}^N \log |x-x_{i,\ep}^z|,
	$$
	which satisfies
	\begin{equation}\label{equationPhi_ep}
		-\Delta \Phi_{\ep}^z=2\pi \sum_{i=1}^N \delta_{x_{i,\ep}^z}\quad \mbox{in }\RR^2.
	\end{equation}
	Notice that
	$$
	\nabla \Phi_{\ep}^z(x)=-\sum_{i=1}^N \frac{x-x_{i,\ep}^z}{|x-x_{i,\ep}^z|^2}.
	$$
	In addition
	$$
	|\nabla \Phi_{\ep}^z(x)|=|\nabla^\perp \Phi_{\ep}^z(x)|=\left| \sum_{i=1}^N \frac{(x-x_{i,\ep}^z)^\perp}{|x-x_{i,\ep}^z|^2} \right|.
	$$
	Therefore
	\begin{align}\label{intermediateintegral}
		\frac12 \int_{A_{r_\ep}^\de(z)}\Bigg | \sum_{j=1}^N & \frac{(x_{j,\ep}^z-x)^\perp}{|x_{j,\ep}^z-x|^2} \Bigg|^2 dx=\frac12 \int_{A_{r_\ep}^\de(z)} |\nabla \Phi_{\ep}^z(x)|^2dx\notag\\
		&=-\frac12 \int_{A_{r_\ep}^\de(z)} -\Delta \Phi_{\ep}^z \Phi_{\ep}^z dx +\frac12 \int_{\partial A_{r_\ep}^\de(z)} \Phi_{\ep}^z \frac{\partial \Phi_{\ep}^z }{\partial \nu}dS(x)\notag\\
		&=\frac12 \int_{\partial D\left(0,\de(1+o_\delta(1))\right)}\Phi_{\ep}^z \frac{\partial \Phi_{\ep}^z}{\partial \nu}dS(x)-\frac12 \sum_{i=1}^N \int_{\partial D\left(x_{i,\ep}^z,r_\ep(1-o_\ep(1))\right)}\Phi_{\ep}^z \frac{\partial \Phi_{\ep}^z}{\partial \nu}dS(x),
	\end{align}
	where we used integration by parts and \eqref{equationPhi_ep}.
	
	To estimate the first integral in the RHS of \eqref{intermediateintegral}, we observe that, for $$x\in \partial D\left(0,\de(1+o_\delta(1))\right),$$ we have
	$$
	\Phi_{\ep}^z(x)=-\sum_{i=1}^N \left(\log \frac{|x-x_{i,\ep}^z|}{|x|}+\log|x|\right)=-N\log |x|+O\left(\frac{\sqrt{\lep}}{\delta}\right)
	$$
	and 
	$$
	\frac{\partial\Phi_{\ep}^z(x)}{\partial \nu}=-\sum_{i=1}^N\frac{x-x_{i,\ep}^z}{|x-x_{i,\ep}^z|^2}\cdot \frac{x}{|x|}=-\sum_{i=1}^N\left(\frac1{|x|}+\frac{x_{i,\ep}^z\cdot x-|x_{i,\ep}^z|^2}{|x||x-x_{i,\ep}^z|^2}\right)=-\frac{N}{|x|}+O\left(\frac{\sqrt{\lep}}{\delta^2}\right).
	$$
	Hence
	\begin{multline}\label{intermediateintegral1}
		\frac12 \int_{\partial D\left(0,\de(1+o_\delta(1))\right)}\Phi_{\ep}^z \frac{\partial \Phi_{\ep}^z}{\partial \nu}dS(x)\\
		=\pi N^2 \log\delta+o_\delta(1) +O\left(\frac{\log\delta\sqrt{\lep}}{\delta}+\frac{\lep}{\de^2}\right)=\pi N^2 \log\delta+o_\delta(1)+C_\delta o_\ep(1).
	\end{multline}
	We now proceed to estimate the second integral in the RHS of \eqref{intermediateintegral}. Note that, for $x\in \partial D\left(x_{i,\ep}^z,r_\ep(1-o_\ep(1))\right)$, we have
	\begin{multline*}
		\Phi_{\ep}^z(x)=-\log|x-x_{i,\ep}^z|-\sum_{\substack{j=1\\j\neq i}}^N \left(\log \frac{|x-x_{j,\ep}^z|}{|x_{i,\ep}^z-x_{j,\ep}^z|}+\log |x_{i,\ep}^z-x_{j,\ep}^z|\right) \\
		=-\log|x-x_{i,\ep}^z|-\sum_{\substack{j=1\\j\neq i}}^N \log |x_{i,\ep}^z-x_{j,\ep}^z|+O\left(\frac{r_\ep}{\sqrt\lep}\right)
	\end{multline*}
	and
	\begin{multline*}
		\frac{\partial\Phi_{\ep}^z(x)}{\partial \nu}=-\sum_{j=1}^N\frac{x-x_{j,\ep}^z}{|x-x_{j,\ep}^z|^2}\cdot \frac{x-x_{i,\ep}^z}{|x-x_{i,\ep}^z|}\\
		=-\frac{1}{|x-x_{i,\ep}^z|}\left(1+\sum_{\substack{j=1\\j\neq i}}^N \frac{(x-x_{j,\ep}^z)\cdot (x-x_{i,\ep}^z)}{|x-x_{j,\ep}^z|^2}\right)=-\frac{1}{|x-x_{i,\ep}^z|}\left(1+O\left(\frac{r_\ep}{\sqrt\lep}\right) \right).
	\end{multline*}
	We then obtain that
	$$
	\frac12 \int_{\partial D\left(x_{i,\ep}^z,r_\ep(1-o_\ep(1))\right)}\Phi_{\ep}^z \frac{\partial \Phi_{\ep}^z}{\partial \nu}dS(x)=\pi \log r_\ep +\pi \sum_{\substack{j=1\\j\neq i}}^N \log |x_{i,\ep}^z-x_{j,\ep}^z|+O\left(\frac{r_\ep \log r_\ep}{\sqrt\lep} \right)+o_\ep(1)
	$$
	and thus
	\begin{equation}\label{intermediateintegral2}
		-\frac12 \sum_{i=1}^N \int_{\partial D\left(x_{i,\ep}^z,r_\ep(1-o_\ep(1))\right)}\Phi_{\ep}^z \frac{\partial \Phi_{\ep}^z}{\partial \nu}dS(x)=-\pi N \log r_\ep -\pi \sum_{\substack{i,j=1\\i\neq j}}^N \log |x_{i,\ep}^z-x_{j,\ep}^z|+o_\ep(1).
	\end{equation}
	By inserting \eqref{intermediateintegral1} and \eqref{intermediateintegral2} into \eqref{intermediateintegral}, we obtain the claim \eqref{renormalized1}.
	
	\medskip
	On the other hand, since
	$$
	\ga_{i,\ep}(z)=\Pi^{-1}(x_{i,\ep}^z)\quad \mbox{and}\quad \ga_{j,\ep}(z)=\Pi^{-1}(x_{j,\ep}^z),
	$$
	for $i\neq j$ we have that
	$$
	\distg\left(\ga_{i,\ep}(z),\ga_{j,\ep}(z)\right)\leq \int_0^1 \|D\Pi^{-1}\left(x_{i,\ep}^z+t(x_{j,\ep}^z-x_{i,\ep}^z)\right)\left((x_{j,\ep}^z-x_{i,\ep}^z)\right)\|dt,
	$$
	which combined with \eqref{estimatederivativePhi} yields
	$$
	\distg\left(\ga_{i,\ep}(z),\ga_{j,\ep}(z)\right)\leq (1+O(|x_{i,\ep}^z|+|x_{j,\ep}^z|))|x_{j,\ep}^z-x_{i,\ep}^z|.
	$$
	The concavity of the logarithm function then implies that
	\begin{align}\label{distancegest}
		-\log |x_{j,\ep}^z-x_{i,\ep}^z|&\leq -\log \distg\left(\ga_{i,\ep}(z),\ga_{j,\ep}(z)\right)+O(|x_{i,\ep}^z|+O\left(|x_{j,\ep}^z|)\right)\notag \\
		&\leq -\log \distg\left(\ga_{i,\ep}(z),\ga_{j,\ep}(z)\right)+O\left(\frac1{\sqrt\lep}\right).
	\end{align}
	Finally, by combining \eqref{integralY_jperforated2} with \eqref{renormalized1} and \eqref{distancegest}, we find
	\begin{multline}\label{renormalized2}
		\frac12 \int_{A_{r_\ep}^\de(z)}\left |\sum_{j=1}^N \frac{(x_{j,\ep}^z-x)^\perp}{|x_{j,\ep}^z-x|^2}\right|^2 dx
		\\
		\leq -\pi \sum_{\substack{i,j=1\\i\neq j}}^N \log\distg\left(\ga_{i,\ep}(z),\ga_{j,\ep}(z)\right)+\pi N \log \frac1{r_\ep}+\pi N^2\log \de +o_\de(1)+C_\delta o_\ep(1)+o_\ep(1).
	\end{multline}
	
	\item \emph{Conclusion.} By integrating from $z=0$ to $z=\ellzero$ \eqref{integralY_jperforated2} and \eqref{renormalized2}, and combining with \eqref{energyperforatedtube}, we find
	\begin{multline}\label{renormalizedenergyperforatedtube}
		E_\ep(u_\ep,T_{r_\ep}^\delta)
		\leq 
		-\pi \sum_{\substack{i,j=1\\i\neq j}}^N \int_0^{\ellzero} \log\distg\left(\ga_{i,\ep}(z),\ga_{j,\ep}(z)\right)dz\\
		+\pi N\ellzero \log \frac1{r_\ep}
		+\pi N^2\ellzero\log \de +o_\de(1)+C_\delta o_\ep(1)+o_\ep(1).
	\end{multline}
\end{enumerate}

\subsection{Energy estimate far from \texorpdfstring{$\ga_0$}{Γ\textzeroinferior}}
We now work in $\Omega\setminus \overline{T_\de(\ga_0)}$. In this region $\rho_\ep\equiv 1$ and therefore, using \eqref{derivativephi}, we find
$$
|\nabla_{A_\ep}u_\ep|=|\nabla \varphi_\ep -A_\ep|=\left|\sum_{i=1}^N X_{\ga_{i,\ep}}+\nabla f_{\ga_{i,\ep}}-A_{\ga_{i,\ep}}\right|=\left|\sum_{i=1}^N j_{\ga_{i,\ep}}\right|,
$$
where in the last equality we used once again Proposition \ref{jA}.
Hence
\begin{multline*}
	I_\delta\colonequals \frac12 \int_{\Omega\setminus \overline{T_\delta(\ga_0)}}|\nabla_{A_\ep} u_\ep|^2+\frac{1}{2\ep^2}(1-\rho_\ep^2)^2 +\frac12 \int_{\RR^3}|\curl A_\ep|^2
	\\=\frac12\int_{\Omega\setminus \overline{T_\delta(\ga_0)}}\left|\sum_{i=1}^N j_{\ga_{i,\ep}}\right|^2+\frac12 \int_{\RR^3}\left|\sum_{i=1}^N \curl A_{\ga_{i,\ep}}\right|^2.
\end{multline*}
From \eqref{diffjA} and the explicit formula \eqref{biotsavart}, which yields a control on $\|X_{\ga_{i,\ep}}-X_{\ga_0}\|_{L^2\left(\Omega\setminus \overline{T_\delta(\ga_0)}\right)}$, we deduce that, for any $i=1,\dots,N$, 
$$
\left\|j_{\ga_{i,\ep}}-j_{\ga_0} \right\|^2_{L^2\left(\Omega \setminus \overline{T_\delta(\ga_0)}\right)}+\left\|A_{\ga_{i,\ep}}-A_{\ga_0} \right\|^2_{H^1(\RR^3)}\leq C_\delta o_\ep (1),
$$
which combined with the previous equality yields
$$
I_\delta=N^2\left( \frac12\int_{\Omega\setminus T_\delta(\ga_0)}\left|j_{\ga_{0}}\right|^2+\frac12 \int_{\RR^3}\left| \curl A_{\ga_{0}}\right|^2\right) +C_\delta o_\ep(1).
$$
Recalling \eqref{comega}, we obtain
\begin{equation}\label{energyestimateoutsidetube}
	I_\delta= N^2C_\Omega -\pi N^2\ellzero\log\delta +o_\delta(1)+C_\delta o_\ep(1).
\end{equation}

\subsection{Vorticity estimate}
Let $B\in C^{0,1}_T(\Omega)$. Observe that, by integration by parts, we have
\begin{multline}\label{vortest1}
	\int_\Omega\mu(u_\ep,A_\ep)\cdot B= \int_\Omega \curl (j(u_\ep,A_\ep)+A_\ep)\cdot B =\int_\Omega \left( j(u_\ep,A_\ep)+A_\ep\right) \cdot \curl B\\
	=\int_\Omega \rho_\ep^2\nabla \varphi_\ep\cdot \curl B+\int_\Omega (1-\rho_\ep^2)A_\ep \cdot \curl B\\
	=\int_\Omega \nabla \varphi_\ep \cdot \curl B + \int_\Omega (1-\rho_\ep^2)(A_\ep -\nabla \varphi_\ep)\cdot \curl B.
\end{multline}
Let $q<2$. Recall that $A_\ep -\nabla \varphi \in L^q(\Omega)$, therefore, letting $p>2$ be such that $\frac1p+\frac1q=1$, by H\"older's inequality we have
\begin{align}
	\left|\int_\Omega (1-\rho_\ep^2)(A_\ep -\nabla \varphi_\ep)\cdot \curl B\right|
	&\leq \|\curl B\|_{L^\infty(\Omega)}\|A_\ep -\nabla \varphi_\ep\|_{L^q(\Omega)}\|(1-\rho_\ep^2)\|_{L^p(\Omega)} \notag \\
	&\leq C\|\curl B\|_{L^\infty(\Omega)}\|(1-\rho_\ep^2)\|_{L^2(\Omega)}^\frac{2}{p} \notag \\
	&\leq C\|\curl B\|_{L^\infty(\Omega)}\ep^\frac{2}{p}E_\ep(|u_\ep|) \notag \\
	&\leq C\|\curl B\|_{L^\infty(\Omega)}\ep^\frac{2}{p}\lep,\label{vortest2}
\end{align}
where we used that $1-\rho_\ep^2\in [0,1]$, $\rho_\ep=1$ in $\Omega\setminus \cup_{i=1}^N T_{r_\ep}(\ga_{i,\ep})$, and \eqref{energyestimatesmalltube}.

\medskip
On the other hand, using \eqref{derivativephi}, we have
\begin{equation}\label{vortest3}
	\int_\Omega \nabla \varphi_\ep \cdot \curl B= \sum_{i=1}^N \int_\Omega\left(X_{\ga_{i,\ep}}+\nabla f_{\ga_{i,\ep}}\right)\cdot \curl B=\sum_{i=1}^N \int_\Omega \curl X_{\ga_{i,\ep}}\cdot B,
\end{equation}
where the last equality follows from integration by parts. Finally, since $\curl X_{\ga_{i,\ep}}=2\pi \ga_{i,\ep}$, from \eqref{vortest1}, \eqref{vortest2}, and \eqref{vortest3}, we deduce that
\begin{equation}\label{vorticityestimateconstruction}
	\left\|\mu(u_\ep,A_\ep)-2\pi \sum_{i=1}^N \ga_{i,\ep} \right\|_{(C_T^{0,1}(\Omega))^*}\leq C\ep^\frac{2}{p}\lep ,
\end{equation}
for any $p>2$, where $C$ is a constant that depends on $p$ and blows up as $p\to 2$. We observe that the same estimate holds for the flat norm.

\subsection{Putting everything together}
By putting together \eqref{covariantderivativeenergy}, \eqref{energyestimatesmalltube}, \eqref{renormalizedenergyperforatedtube}, and \eqref{energyestimateoutsidetube}, we obtain
\begin{multline*}
	F_\ep(u_\ep,A_\ep)
	\leq  \sum_{i=1}^N\left(\pi|\ga_{i,\ep}|\log \frac{r_\ep}{\ep}+\gamma |\ga_{i,\ep}|\right)
	-\pi \sum_{\substack{i,j=1\\i\neq j}}^N \int_0^{\ellzero} \log \distg\left(\ga_{i,\ep}(z),\ga_{j,\ep}(z)\right)dz\\
	+\pi N\ellzero \log \frac1{r_\ep}
	+\pi N^2\ellzero\log \de+N^2C_\Omega -\pi N^2\ellzero\log\delta +o_\de(1)+C_\de o_\ep(1)+o_\ep(1).
\end{multline*}
From the computations in the proof of Lemma \ref{lemformQ}, we have
$$
|\ga_{i,\ep}|=\ellzero+O\left(\frac1\lep \right),
$$
which yields (recall that $r_\ep=\lep^{-3}$)
$$
\pi \sum_{i=1}^N|\ga_{i,\ep}|\log r_\ep+\pi N\ellzero\log \frac1{r_\ep}=o_\ep(1)
$$
and that $\gamma |\ga_{i,\ep}|=\gamma \ellzero+o_\ep(1)$. Hence
\begin{multline}\label{freeenergyestiamteconstruction}
	F_\ep(u_\ep,A_\ep)
	\leq  \pi\lep \sum_{i=1}^N|\ga_{i,\ep}|
	-\pi \sum_{\substack{i,j=1\\i\neq j}}^N \int_0^{\ellzero} \log \distg\left(\ga_{i,\ep}(z),\ga_{j,\ep}(z)\right)dz
	\\
	+\gamma N \ellzero+N^2C_\Omega +o_\de(1)+C_\de o_\ep(1)+o_\ep(1).
\end{multline}

Recall that $\u_\ep=e^{ih_\ex\phi_0}u_\ep$ and $\A_\ep=A_\ep+h_\ex A_0$, where $(e^{ih_\ex\phi_0},h_\ex A_0)$ is the approximate Meissner state. By inserting \eqref{freeenergyestiamteconstruction} and \eqref{vorticityestimateconstruction} in the splitting formula \eqref{Energy-Splitting}, we find
\begin{multline*}
	GL_\ep(\u_\ep,\A_\ep)
	\leq  h_\ex^2 J_0+\pi\lep \sum_{i=1}^N|\ga_{i,\ep}|
	-\pi \sum_{\substack{i,j=1\\i\neq j}}^N \int_0^{\ellzero} \log \distg\left(\ga_{i,\ep}(z),\ga_{j,\ep}(z)\right)dz
	\\
	+\gamma N \ellzero+N^2C_\Omega-2\pi \sum_{i=1}^N h_\ex\pr{B_0,\ga_{i,\ep}} +o_\de(1)+C_\de o_\ep(1)+o_\ep(1).
\end{multline*}
Since $h_\ex=\dfrac{\lep}{2\R(\ga_0)}+K\log\lep$, for some $K\geq 0$, we have that
\begin{align*}
	2\pi \sum_{i=1}^N  h_\ex\pr{B_0,\ga_{i,\ep}}&=\pi \lep\sum_{i=1}^N \left(|\ga_{i,\ep}| \frac{\R(\ga_{i,\ep})}{\R(\ga_0)}+2\pi K \log\lep \pr{B_0,\ga_{i,\ep}}\right)\\
	&=\pi \lep\sum_{i=1}^N |\ga_{i,\ep}| \frac{\R(\ga_{i,\ep})}{\R(\ga_0)}+2\pi K N \log\lep \pr{B_0,\ga_0},
\end{align*}
where in the last equality we used the fact that $\pr{B_0,\ga_{i,\ep}}=\pr{B_0,\ga_0}+O\left(\dfrac1{\sqrt\lep}\right)$ (which follows from Proposition \ref{bexpansion}).
Hence
\begin{align*}
	GL_\ep(\u_\ep,\A_\ep)
	\leq & h_\ex^2 J_0+\pi\lep \sum_{i=1}^N|\ga_{i,\ep}|\left(1-\frac{\R(\ga_{i,\ep})}{\R(\ga_0)}\right) 
	-\pi \sum_{\substack{i,j=1\\i\neq j}}^N \int_0^{\ellzero} \log \distg\left(\ga_{i,\ep}(z),\ga_{j,\ep}(z)\right)dz
	\\
	&+\gamma N \ellzero+N^2C_\Omega-2\pi K N \log\lep \pr{B_0,\ga_0} +o_\de(1)+C_\de o_\ep(1)+o_\ep(1).
\end{align*}
Moreover, from Lemma \ref{lemformQ}, for any $i=1,\dots N$ we have
\begin{align*}
\R(\ga_{i,\ep})&=\R(\ga_0)-\frac12  Q\left(\sqrt{\frac{N}{h_\ex}}\vu_i\right)+O\left(\frac1{\lep^\frac32}\right)\\
&=\R(\ga_0)-\frac12  \frac{N}{h_\ex}Q\left(\vu_i\right)+O\left(\frac1{\lep^\frac32}\right),
\end{align*}
which combined with the fact that $|\ga_{i,\ep}|=\ellzero+o_\ep(1)$, yields
$$
\pi\lep \sum_{i=1}^N|\ga_{i,\ep}|\left(1-\frac{\R(\ga_{i,\ep})}{\R(\ga_0)}\right) =\pi\ellzero N\sum_{i=1}^N Q(\vec u_i)+o_\ep(1).
$$
Finally, by observing that 
\begin{multline*}
	-\pi \sum_{\substack{i,j=1\\i\neq j}}^N \int_0^{\ellzero} \log \distg\left(\ga_{i,\ep}(z),\ga_{j,\ep}(z)\right)dz=-\pi\sum_{\substack{i,j=1\\i\neq j}}^N \int_0^{\ellzero} \log |\vu_i(z) -\vu_j(z)|_{g_\bullet}dz\\
	+\frac{\pi}2 \ellzero N(N-1)\log\frac{h_\ex}{N}+o_\ep(1),
\end{multline*}
we obtain \eqref{upperboundestimate}. The proof of Theorem \ref{thm:upperbound} is thus concluded.\qed

\bibliography{referencesRSS}
\end{document}